\documentclass{article}
\usepackage[utf8]{inputenc}

\usepackage[section]{placeins}
\usepackage{gensymb}
\usepackage{multicol}
\usepackage{graphicx}
\usepackage{framed}
\usepackage{setspace}
\usepackage{amsmath}
\usepackage{amssymb}
\usepackage{amsfonts}
\usepackage{mathrsfs}
\usepackage{mathtools}
\usepackage{array}
\usepackage{indentfirst}
\usepackage{amsopn}
 
\usepackage{float}
\usepackage{IEEEtrantools}
\usepackage{multirow}
\usepackage{hyperref}
\usepackage{cases}
\usepackage{caption}
\usepackage{listings}
\usepackage{stmaryrd}
\usepackage[margin=1in]{geometry}
\usepackage{setspace}
\usepackage[abbreviations = false]{siunitx}
\usepackage{wrapfig}
\usepackage{xfrac} 
\usepackage{array} 
\usepackage[text]{esdiff}
\usepackage{subcaption}

\usepackage{framed} 
\usepackage{pifont} 
\usepackage{enumerate} 
\usepackage{wasysym} 
\usepackage[super]{nth}

\usepackage{amsthm}  
\usepackage{authblk}
\newtheorem{thm}{\bf Theorem}[section] 
\newtheorem{lem}[thm]{\bf Lemma}

\newtheorem{conj}{\bf Conjecture}

\theoremstyle{definition}
\newtheorem{defn}{\bf Definition}[section]

\theoremstyle{remark} 
\newtheorem{rem}{\bf Remark}[section]
\newtheorem*{note}{Note}

\usepackage{cite}
\usepackage{url}

\title{Travelling-wave analysis of a model of tumour invasion with degenerate, cross-dependent diffusion}

\author[1]{Chloé Colson}
\author[2]{Faustino S{\'a}nchez-Gardu{\~n}o}
\author[1]{Helen M. Byrne}
\author[1]{Philip K. Maini}
\author[3]{Tommaso Lorenzi}
\affil[1]{Wolfson Centre for Mathematical Biology, Mathematical Institute, University of Oxford, Radcliffe Observatory Quarter, OX2 6GG, Oxford, UK}
\affil[2]{Departamento de Matem{\'a}ticas, Facultad de Ciencias, UNAM, Ciudad Universitaria, Circuito Exterior, Cd. de México, C.P. 04510, México}
\affil[3]{Department of Mathematical Sciences ‘G. L. Lagrange’, Politecnico di Torino, 10129 Torino, Italy}
\date{}                     
\begin{document}
\numberwithin{equation}{section}
 \numberwithin{rem}{section}

\maketitle

\begin{abstract}
In this paper, we carry out a travelling-wave analysis of a model of tumour invasion with degenerate, cross-dependent diffusion. We consider two types of invasive fronts of tumour tissue into extracellular matrix (ECM), which represents healthy tissue. These types differ according to whether the density of ECM far ahead of the wave front is maximal or not. In the former case, we use a shooting argument to prove that there exists a unique travelling wave solution for any positive propagation speed. In the latter case, we further develop this argument to prove that there exists a unique travelling wave solution for any propagation speed greater than or equal to a strictly positive minimal wave speed. Using a combination of analytical and numerical results, we conjecture that the minimal wave speed depends monotonically on the degradation rate of ECM by tumour cells and the ECM density far ahead of the front. 
\end{abstract}


\section{Introduction}

Tissue invasion is a hallmark of malignant tumours \cite{35} and a classical mathematical approach to study this process involves reaction-diffusion (R-D) partial differential equations (PDEs) \cite{11,17,28}. A key feature of many such models of tumour invasion is the inclusion of degenerate, cross-dependent diffusion. The aim of this paper is to study this common characteristic by proposing a minimal model which captures the main components of the tumour invasion process and is analytically tractable. We seek two types of constant profile, constant speed travelling wave solutions (TWS) for our model. Both types represent invasive fronts of tumour tissue into extracellular matrix (ECM), which represents healthy tissue, but they differ according to whether the density of ECM far ahead of the wave front is maximal or not. For the former, we prove the existence and uniqueness of TWS for all positive propagation speeds using the shooting argument developed by Gallay and Mascia \cite{10}. For the latter, we expand this shooting argument to prove the existence and uniqueness of TWS for propagation speeds greater than or equal to a strictly positive minimal value. Finally, we characterise this minimal wave speed using a conjecture motivated by a combination of analytical results and numerical simulations.

\paragraph{Reaction-diffusion partial differential equation models of tumour invasion.}
To invade the surrounding healthy tissue, a tumour must overcome the defenses developed by the body to maintain homeostatic control. An important barrier to tumour invasion is the ECM, a strong scaffold of proteins that holds tissue cells in place and initiates signalling pathways for cellular processes such as migration,  differentiation and proliferation \cite{36, 37}. The healthy cells encased by the ECM form another barrier to invasion by creating a competitive environment for the tumour cells. However, tumour cells have developed mechanisms to overcome both of these barriers. First, they can remodel or degrade the ECM by producing specific matrix degrading enzymes, which act in close proximity to the cells producing them \cite{38,42}. Second, by favouring glycolytic metabolism even in aerobic conditions (i.e. the "Warburg Effect"), tumour cells may acidify the tissue microenvironment, resulting in healthy cell death \cite{40,41}. Matrix remodelling is a very localised process, in contrast to the diffusion of lactic acid which occurs on a longer spatial range.

The pioneering model by Gatenby and Gawlinksi \cite{11} describes the spatio-temporal dynamics of acid-mediated tumour invasion by considering the interactions of healthy tissue, tumour tissue and the lactic acid produced by the tumour cells. Denoting the dimensionless tumour and healthy tissue densities and acid concentration by $N(x,t), M(x,t)$ and $L(x,t)$, respectively, for ${(x,t) \in \mathbb{R} \times (0,\infty)}$, their model takes the form 
\begin{equation}
 \begin{cases}
 \displaystyle \frac{\partial N}{\partial t}  =  \beta N(1-N)+ \frac{\partial}{\partial x}\left[D_N (1-M)\frac{\partial N}{\partial x}\right],\\[5pt]
  \displaystyle \frac{\partial M}{\partial t}   =  M(1-M-\alpha L),\\[5pt]
  \displaystyle \frac{\partial L}{\partial t}  =  \gamma(N-L)+\frac{\partial^2 L}{\partial x^2}.
 \end{cases}
 \label{eq:1.4}
\end{equation}
Here, it is assumed that healthy cells do not move, while tumour cells can invade in a density-dependent manner. Depending on the value of $\alpha$, the model describes the total or partial destruction of normal tissue following tumour invasion. We refer the reader to the original paper for full details of the model. A numerical study of the TWS of system \eqref{eq:1.4}, with $0 < D_N \ll 1$, showed the existence of an {\em interstitial gap}, i.e. a region devoid of cells, formed locally ahead of the invading tumour front, for large values of $\alpha$ \cite{11}. Experimental evidence has confirmed that such an interstitial gap can exist and, in this way, the model has led to novel and accurate predictions regarding tumour invasion. This is one of the reasons why this model and its variations have been widely investigated \cite{9,12,17,19,20,46}.

\paragraph{Nonlinear, degenerate diffusion: from scalar to multi-dimensional analysis.}
A key common component of the Gatenby-Gawlinksi model and its variations is the degenerate, cross-diffusion term in the equation for the tumour cell density. For scalar R-D equations with nonlinear, degenerate diffusion, TWS have been extensively studied, see for instance \cite{2, 3, 4, 5, 6, 14, 25, 26, 27}. In general, if the dimensionless equation has a reaction term, $f$, of Fisher-KPP type, i.e. $f \in C[0,1]$ with ${f(0) = f(1)= 0}$ and $f(n)>0 \, \forall n \in (0,1)$, then TWS exist and are unique if and only if the wave speed is greater than or equal to a minimal speed, $c^* > 0$, defined as the threshold speed below which no TWS exist. Further, if $c=c^*$, then the TWS is of sharp type (that is, there is a discontinuity in the spatial derivative at the front) and, for each $c > c^*$, there exists a TWS of front-type (that is, smooth). It is non-trivial to extend such an existence result to R-D systems with multiple equations due to the added complexity of studying trajectories in a phase space, rather than a phase plane. Kawasaki et al. \cite{13} do so for a R-D system with cross-dependent diffusion developed to describe spatio-temporal pattern formation in colonies of bacteria. More specifically, numerical and analytical investigations \cite{15,27} have shown the existence of TWS for wave speeds above or equal to a critical value, $c^* > 0$. Until recently, most comprehensive results on the existence of TWS for spatially-resolved models of tumour invasion focussed on models in which invasion is driven by haptotaxis or chemotaxis \cite{16,22,23,24}. In particular, the existence of TWS for the Gatenby-Gawlinski model has been largely supported by a combination of numerical and analytical results \cite{9,12,19,20,32,33}. This also holds for a simplified model of invasion by Browning et al. \cite{1,7}. However, key results were recently proved by Gallay and Mascia \cite{10} for a reduced version of the Gatenby-Gawlinski model: they showed the existence of a form of weak TWS for any positive wave speed, $c > 0$.

\paragraph{The mathematical model.} We now present a minimal model of tumour invasion. There is increasing evidence that phenotypically heterogeneous tumours can contain sub-populations of cells with different traits, e.g. matrix-degrading cells and acid-producing cells \cite{46}. Therefore, we make the simplifying assumption that the healthy tissue compartment solely comprises ECM, disregarding healthy cells, and we focus on the interactions of ECM-degrading tumour cells and ECM. Using a standard law for conservation of mass and denoting the tumour cell and ECM densities by $N(x,t)$ and $M(x,t)$, respectively, for $(x,t) \in \mathbb{R} \times (0,\infty)$, we propose the following system of PDEs:
\begin{equation}
\begin{cases}
\displaystyle \frac{\partial N}{\partial t} =\underbrace{\frac{\partial}{\partial x} \left[D_N \left(1-\frac{M}{M_{\text{Max}}} \right) \frac{\partial N}{\partial x} \right]}_{\text{tumour cell movement}}+\underbrace{\rho \left(1 - \frac{N}{K}\right) N,}_{\text{tumour growth}}\\
\displaystyle \frac{\partial M}{\partial t} = \underbrace{-k M N.}_{\text{ECM degradation}}
\end{cases}
\label{eq:1.2}
\end{equation}
We assume that the tumour grows logistically, with maximum growth rate, $\rho$, and carrying capacity, $K$. Further, the ECM acts as a physical barrier that inhibits tumour cell movement, but not proliferation. Thus, following Gatenby and Gawlinski \cite{11} and others \cite{17,18,46}, we define the diffusivity of tumour cells as a monotonically decreasing function of the ECM density to model the obstruction of movement by the ECM.  The diffusivity of tumour cells in the absence of ECM is denoted by $D_N$ and the ECM density that inhibits all tumour cell movement is denoted by $M_{\text{Max}}$.  Finally, we assume that the ECM does not to grow and is degraded at a rate that is proportional to the local tumour cell density, with a per cell degradation rate of $k$. We use a mass-action term to reflect the localised nature of matrix degradation. 

To reduce the number of free parameters in the system and facilitate the analysis that follows, we non-dimensionalise equations \eqref{eq:1.2} and, retaining the same dimensional state variables for notational convenience, we obtain the following system:
\begin{equation}
 \begin{cases}
\displaystyle \frac{\partial N}{\partial t} =\frac{\partial}{\partial x} \left[ \left(1-M\right) \frac{\partial N}{\partial x} \right]+\left(1 - N\right)N,\\
\displaystyle \frac{\partial M}{\partial t} = -\kappa M N,
\end{cases}
\label{eq:1.3}
\end{equation}
where $\kappa = \frac{K}{\rho} k$. We note that system \eqref{eq:1.3} is similar to a reduced version of the model \eqref{eq:1.4} from Moschetta and Simeoni \cite{20} and a reduced model of melanoma invasion from Browning et al. \cite{1}. In these models, the healhy tissue compartment comprises cells and ECM, and, as such, they include additional reaction terms that represent logistic growth of the healthy tissue density and healthy tissue competition with tumour tissue, respectively. To derive our model, we assumed that the ECM, which constitutes a physical barrier to tumour cell invasion, represents the healthy tissue. Crucially, this leads to the minimal model \eqref{eq:1.3} that retains the degeneracy in the cross-diffusion term, which is the key focus of this paper.

\paragraph{Structure of the paper.}
We will seek constant profile, constant speed TWS for \eqref{eq:1.3}, which are heteroclinic trajectories of a 3-D dynamical system connecting two of its steady states. These  correspond to spatially homogeneous, steady state solutions of \eqref{eq:1.3}, which are given by:
\begin{equation}
    (N_0^*,M_0^*) =(0,0), \,\, (N_1^*,M_1^*) = (1,0), \,\,(N_2^*,M_2^*) = (0,1),\,\, (N_3^*,M_3^*)= (0,\bar{M}), \, \bar{M} \in [0,1).
\end{equation}
Here, $(N_0^*,M_0^*)$ is the trivial state, $(N_1^*,M_1^*)$ is a state in which the tumour has successfully invaded and degraded all ECM, and $(N_2^*,M_2^*)$ and $(N_3^*,M_3^*)$ are a continuum of healthy, tumour-free states. We distinguish $(N_2^*,M_2^*)$ from $(N_3^*,M_3^*)$ because of the degeneracy at $M=1$ in system \eqref{eq:1.3}. Since we are interested in studying the existence of TWS that describe the invasion of a tumour into healthy tissue, we will look for two types of heteroclinic trajectories: those connecting $(N_1^*,M_1^*)$ to $(N_2^*,M_2^*)$ and those connecting $(N_1^*,M_1^*)$ to $(N_3^*,M_3^*)$. In Section 2, we define the TWS we seek, prove preliminary results and derive the ordinary differential equation (ODE) system they must satisfy. In Section 3, we use the shooting argument developed by Gallay and Mascia \cite{10} to show that system \eqref{eq:1.3} has a unique TWS connecting $(N_1^*,M_1^*)$ to $(N_2^*,M_2^*)$ for any positive wave speed. We then show that, for each $\bar{M} \in [0,1)$, system \eqref{eq:1.3} has a unique TWS connecting $(N_1^*,M_1^*)$ to  $(N_3^*,M_3^*)$ for any wave speed greater than or equal to a strictly positive minimum value. Motivated by our numerical simulations and partial analytical results, we make a conjecture about the dependence of the minimal wave speed on $\bar{M} \in [0,1)$ and $\kappa > 0$, the rescaled degradation rate of the ECM. In Section 4, we present numerical simulations of system \eqref{eq:1.3} which support and complement the preceding analytical results. We conclude the paper in Section 5, where we discuss our results alongside future research perspectives.

\section{The travelling-wave problem}
\subsection{Preliminaries}
We seek constant profile, constant speed TWS of system \eqref{eq:1.3} by introducing the travelling wave coordinate $\xi = x-ct$. We require the wave speed $c>0$ so that the tumour invades the ECM from left to right in the spatial domain. Substituting the ansatz $ N(x,t)=\mathcal{N}(\xi) $ and $M(x,t)=~\mathcal{M}(\xi)$ into system \eqref{eq:1.3}, we deduce that TWS must satisfy the following ODE system:
   \begin{subnumcases}{}
    {\displaystyle \diff{}{\xi}\left((1-\mathcal{M})\diff{ \mathcal{N}}{ \xi }\right) + c \diff{\mathcal{N}}{\xi} + (1-\mathcal{N})\mathcal{N} = 0,} \label{eq:2.1a} \\
 {\displaystyle c \diff{\mathcal{M}}{\xi} - \kappa \mathcal{M}\mathcal{N} = 0.} \label{eq:2.1b}
   \end{subnumcases}

The TWS we seek connect spatially homogeneous steady states of system \eqref{eq:1.3} and, equivalently, steady states of system \eqref{eq:2.1a}-\eqref{eq:2.1b}. Thus, we require one of the following sets of asymptotic conditions to be satisfied:
   \begin{align}
      &{\lim_{\xi \to -\infty}(\mathcal{N}(\xi),\mathcal{M}(\xi)) = (1,0), \, \lim_{\xi \to +\infty}(\mathcal{N}(\xi),\mathcal{M}(\xi)) = (0,1),} \label{eq:2.2a}\\ 
     &{\lim_{ \xi \to -\infty}(\mathcal{N}(\xi),\mathcal{M}(\xi)) = (1,0), \lim_{\xi \to +\infty}(\mathcal{N}(\xi),\mathcal{M}(\xi)) = (0,\bar{\mathcal{M}}) \, \text{with } \bar{\mathcal{M}}\in[0,1).} 
   \label{eq:2.2b}
    \end{align}
In other words, far behind the wave, the tumour density is at carrying capacity and the ECM has been fully degraded, whereas, far ahead of the wave, the tumour density is zero and the ECM density is either at carrying capacity (i.e. $\mathcal{M}=1$) or at any value $\mathcal{M}\in[0,1)$. As noted previously, the first equation in system \eqref{eq:1.3} is a degenerate parabolic equation since the cross-diffusion coefficient $D(M)=1-M$ is zero when $M=1$. The existence of global classical solutions of this PDE system and the corresponding ODE system \eqref{eq:2.1a}-\eqref{eq:2.1b} is therefore unclear in cases where $M=1$ or, correspondingly, where $\mathcal{M}=1$. We therefore define a weak TWS in a similar way to the definition of a propagation front in \cite{10}.

\begin{defn}
The triple $(\mathcal{N},\mathcal{M};c)$ is called a weak TWS for system \eqref{eq:1.3} if 
\begin{enumerate}
    \item  $(\mathcal{N},\mathcal{M}) \in C(\mathbb{R},[0,1]) \times C(\mathbb{R},[0,1])$ and $(1-\mathcal{M}) \diff{\mathcal{N}}{\xi} \in L^2(\mathbb{R})$;
    \item $(\mathcal{N},\mathcal{M})$ is a weak solution of \eqref{eq:2.1a}-\eqref{eq:2.1b}, i.e. for all $(\phi,\psi) \in C^1(\mathbb{R}) \times C^1(\mathbb{R})$ with compact support
    \begin{equation}
      \int_\mathbb{R} \left\{\left[c \mathcal{N} + (1-\mathcal{M}) \diff{\mathcal{N}}{\xi} \right]\diff{\phi}{\xi} - (1-\mathcal{N})\mathcal{N}\phi \right\} \mathrm{d}\xi =0,
      \label{eq:2.3}
    \end{equation}
    \begin{equation}
     \int_\mathbb{R} \mathcal{M} \left\{c \diff{\psi}{\xi} + \kappa \mathcal{N} \psi \right\} \mathrm{d}\xi =0;
     \label{eq:2.4}
    \end{equation}
\item one of the pairs of asymptotic conditions given by \eqref{eq:2.2a} and \eqref{eq:2.2b}, respectively, are satisfied. 
\end{enumerate}
We refer to $(\mathcal{N},\mathcal{M})$ as the travelling wave profile and $c$ as the propagation speed.
\label{def:2.1}
\end{defn}

\begin{note}
Henceforth, unless otherwise stated, we refer to weak TWS in the sense of Definition \ref{def:2.1} as TWS.
\end{note}

If $(\mathcal{N},\mathcal{M};c)$ is a TWS for system \eqref{eq:1.3}, then we can show that $\mathcal{N}(1-\mathcal{N}) \in L^1(\mathbb{R})$ and $c > 0$ using a proof identical to that of Lemma 2.1 in \cite{10} and, thus, we omit it. The following lemma, whose proof is in Supplementary Material S1, states that, if $(\mathcal{N},\mathcal{M};c)$ is a TWS for system \eqref{eq:1.3}, then $\mathcal{N}$ and $\mathcal{M}$ are non-negative and bounded and, thus, the TWS is biologically realistic.
 \begin{lem}
 If $(\mathcal{N},\mathcal{M};c)$ is a weak TWS, in the sense of Definition \ref{def:2.1}, that satisfies the asymptotic conditions \eqref{eq:2.2b} for $\bar{\mathcal{M}} \in (0,1)$, then there exists a unique point $\bar{\xi} \in \mathbb{R} \cup \{+\infty\}$ such that
 \begin{enumerate}
     \item $\mathcal{N},\mathcal{M} \in C^\infty((-\infty,\bar{\xi}))$ and $0 < \mathcal{N}(\xi) < 1$, $0 < \mathcal{M}(\xi) < \bar{\mathcal{M}}$ for $\xi < \bar{\xi}$;
     \item If $\bar{\xi} < +\infty$, then $\mathcal{N}(\xi) = 0$ and $\mathcal{M}(\xi) = \bar{\mathcal{M}}$ for all $\xi \geq \bar{\xi}.$ 
\end{enumerate}
\label{lem:2.1}
\end{lem}

\begin{rem}
The case of TWS that satisfy the asymptotic conditions \eqref{eq:2.2b} for $\bar{\mathcal{M}}=0$ is not considered in Lemma \ref{lem:2.1}. By definition, such solutions satisfy $\displaystyle \lim_{\xi\to \pm \infty} \mathcal{M}(\xi) = 0$ for $\mathcal{N} \geq 0$, which is only possible if $\mathcal{M} \equiv 0$ on $\mathbb{R}$ since $\mathcal{M}$ is increasing for $\mathcal{N},\mathcal{M} > 0$. In this case, system \eqref{eq:2.1a}-\eqref{eq:2.1b} reduces to the Fisher-KPP equation, which has been extensively studied \cite{21,30,31}. It is known that the Fisher-KPP equation admits classical TWS that satisfy the asymptotic conditions $\displaystyle \lim_{\xi\to -\infty } \mathcal{N}(\xi) = 1$, $\displaystyle \lim_{y\to +\infty} \mathcal{N}(\xi) = 0$ and $\displaystyle \lim_{\xi\to \pm \infty } \diff{\mathcal{N}}{\xi}(\xi)=0$ for all $c \geq 2$. This result, therefore, holds for TWS of \eqref{eq:2.1a}-\eqref{eq:2.1b} satisfying the asymptotic conditions \eqref{eq:2.2b} for $\bar{\mathcal{M}}=0$.
\label{rem:2.1}
\end{rem}

A version of Lemma \ref{lem:2.1} for TWS that satisfy the asymptotic conditions $\eqref{eq:2.2a}$ follows similarly \cite{10}. These results highlight that the solutions we seek are classical solutions of system \eqref{eq:2.1a}-\eqref{eq:2.1b} on intervals of the form $(-\infty,\bar{\xi})$. Further, they can be smooth ($\bar{\xi} = \infty)$ or sharp ($\bar{\xi} < \infty$).  

\subsection{Desingularisation of the ODE system}
Definition \ref{def:2.1} describes two types of TWS of system \eqref{eq:2.1a}-\eqref{eq:2.1b}, which differ in the asymptotic conditions they satisfy at infinity. One type of solution converges to $(\mathcal{N},\mathcal{M})=(0,1)$ at infinity. Therefore, we need to elucidate the behaviour of solutions as they approach $\mathcal{M}=1$, which is precisely when system \eqref{eq:2.1a}-\eqref{eq:2.1b} is {\em singular}. A common approach to simplify the analysis is to remove this singularity by re-parametrising the system. 
Given a solution $(\mathcal{N},\mathcal{M})$ of system \eqref{eq:2.1a}-\eqref{eq:2.1b} satisfying either \eqref{eq:2.2a} or \eqref{eq:2.2b}, we introduce a new independent variable $y = \Phi(\xi)$ defined such that
\begin{equation}
\diff{y}{\xi} \equiv \Phi'(\xi) = \frac{1}{1-\mathcal{M}(\xi)} \,\, \forall \xi \in \mathbb{R}.
\label{eq:2.5}
\end{equation}
Further introducing the following dependent variables
\begin{equation}
n(y)=\mathcal{N}(\Phi^{-1}(y)), \qquad m(y) = \mathcal{M}(\Phi^{-1}(y)), \quad y \in \mathbb{R},
\label{eq:2.6}
\end{equation}
we can apply the chain rule and use \eqref{eq:2.5} to find that, for $0 \leq m \leq 1$, the trajectories satisfy the following ODE system, for $y \in \mathbb{R}$:
\begin{subnumcases}{}
{ \frac{\mathrm{d}^2 n}{\mathrm{d}y^2}+ c \diff{n}{y} + (1-n)n(1-m) = 0},  \label{eq:2.7a}
\\
{\diff{m}{y} - \frac{\kappa}{c} m (1-m) n = 0.} \label{eq:2.7b}
\end{subnumcases}

In line with the asymptotic conditions \eqref{eq:2.2a} and \eqref{eq:2.2b}, we require one of the following to hold:
\begin{align}
   &{\lim_{y \to -\infty}(n(y),m(y)) = (1,0), \, \lim_{y \to +\infty}(n(y),m(y)) = (0,1),} \label{eq:2.7c}\\ 
   &{\lim_{y \to -\infty}(n(y),m(y)) = (1,0), \lim_{y \to +\infty}(n(y),m(y)) = (0,\bar{m}) \, \text{with } \bar{m}\in[0,1).} 
   \label{eq:2.7d}
\end{align}

Importantly, system \eqref{eq:2.7a}-\eqref{eq:2.7b} is topologically equivalent to system \eqref{eq:2.1a}-\eqref{eq:2.1b} for $(\mathcal{N},\mathcal{M}) \in (0,1)^2$. This follows from the fact that \eqref{eq:2.6} defines a homeomorphism that maps the orbits of \eqref{eq:2.1a}-\eqref{eq:2.1b} onto the orbits of \eqref{eq:2.7a}-\eqref{eq:2.7b}, while preserving their orientation - \eqref{eq:2.5} implies that $y$ is an increasing function of $\xi$ for all $0 \leq \mathcal{M} < 1$. We also observe that, in contrast to system \eqref{eq:2.1a}-\eqref{eq:2.1b}, system \eqref{eq:2.7a}-\eqref{eq:2.7b} has an additional continuum of steady states of the form $(n,m) = (\bar{n},1)$, $\bar{n} \in (0,1]$. These are not spatially homogeneous steady states of the original PDE system \eqref{eq:1.3}, so we do not consider them as asymptotic conditions in the context of TWS. 

We finally obtain a system of three first order ODEs by introducing the additional variable $p = \frac{\mathrm{d}n}{\mathrm{d}y}$ and, using primes to denote derivatives with respect to $y$, we have:
\begin{subnumcases}{}
{ n' = p},  \label{eq:2.8a}
\\
{p' = -c p - (1-n)n(1-m),}\label{eq:2.8b}
\\
{m'= \frac{\kappa}{c} m (1-m) n.} \label{eq:2.8c}
\end{subnumcases}

In the following section, we set up a framework, first proposed in \cite{10} for a different system, to study two distinct types of solutions of \eqref{eq:2.8a}-\eqref{eq:2.8c}. First, those that remain in the region $\mathcal{D}_1$, defined as
\begin{equation}
\mathcal{D}_1 \coloneqq \{ (n,p,m) \in \mathbb{R}^3 \mid m \in (0,1), n\in (0,1), p \in (-\infty,0) \},
\label{eq:2.9}
\end{equation}
and that satisfy $\displaystyle \lim_{y\to-\infty}(n(y),p(y),m(y)) = (1,0,0)$,   $\displaystyle \lim_{y\to+ \infty}(n(y),p(y),m(y)) = (0,0,1)$. Second, for $\bar{m}\in (0,1)$, those that remain in the region $\mathcal{D}_{\bar{m}}$, defined similarly to \eqref{eq:2.9} as
\begin{equation}
\mathcal{D}_{\bar{m}} \coloneqq  \{ (n,p,m) \in \mathbb{R}^3 \mid m \in (0,\bar{m}), n\in (0,1), p \in (-\infty,0) \},
\label{eq:2.10}
\end{equation}
and that satisfy  $\displaystyle \lim_{y\to- \infty}(n(y),p(y),m(y)) = (1,0,0)$, $\displaystyle \lim_{y\to +\infty}(n(y),p(y),m(y)) = (0,0,\bar{m})$. 

\section{Travelling-wave analysis} \label{sec:3}

In this section, we study the existence of TWS. To do so, we apply the shooting argument developed by Gallay and Mascia \cite{10}. The crucial difference between Gallay and Mascia's model and system \eqref{eq:1.3} is that the latter has an additional continuum of steady states of the form $(0,\bar{M})$, $\bar{M}\in (0,1)$. We find that the results of \cite{10} for TWS connecting the equilibrium points $(1,0,0)$ and $(0,0,1)$ apply, with minor modifications, to the TWS of system \eqref{eq:2.8a}-\eqref{eq:2.8c} that satisfy the same asymptotic conditions \eqref{eq:2.2a}. Therefore, in what follows, we state the key results and present only those proofs which require a different approach (all other proofs are provided in Supplementary Material S1). For TWS of system \eqref{eq:2.8a}-\eqref{eq:2.8c} that satisfy the asymptotic conditions \eqref{eq:2.2b}, we further develop the shooting argument to obtain new results. 

\subsection{Local analysis of the equilibrium point $(1,0,0)$: defining the shooting parameter}
The TWS of interest satisfy $\displaystyle \lim_{y\to-\infty}(n(y),p(y),m(y)) = (1,0,0)$. We therefore study the behaviour of solutions of \eqref{eq:2.8a}-\eqref{eq:2.8c} in a neighbourhood of the equilibrium point $P_1 \coloneqq (1,0,0)$ by performing a linear stability analysis. The Jacobian matrix at $P_1$ reduces to
\[
J \bigr \rvert_{(1,0,0)} =
\left[\begin{array}{ccc}
		0 & 1 & 0 \\
		1 & -c & 0 \\
		0 & 0 &  \frac{\kappa}{c}
\end{array}\right],
\]
and it has the following eigenvalues and eigenvectors: 
\begin{equation}
    \lambda_1 = \frac{-c-\sqrt{c^2+4}}{2}, \quad \lambda_2 = \frac{-c+\sqrt{c^2+4}}{2}, \quad \lambda_3 = \frac{\kappa}{c},
    \label{eq:3.2a}
\end{equation}
\begin{equation}
   \vec{v}_1=\left(\frac{c-\sqrt{c^2+4}}{2},1,0\right)^\top, \quad \vec{v}_2=\left(\frac{c+\sqrt{c^2+4}}{2},1,0\right)^\top, \quad \vec{v}_3=(0,0,1)^\top.
    \label{eq:3.2b}
\end{equation}
Since $\lambda_1$ is negative and $\lambda_2$ and $\lambda_3$ are positive, $P_1$ is a three-dimensional hyperbolic saddle point with a two-dimensional unstable manifold, which locally is a plane through $P_1$ generated by the eigenvectors $\vec{v}_2$ and $\vec{v}_3$. There is also a one-dimensional stable manifold which locally is a straight line spanned by the eigenvector $\vec{v}_1$. Trajectories defined by \eqref{eq:2.8a}-\eqref{eq:2.8c} that leave $P_1$ must do so via the two-dimensional unstable manifold at $P_1$. We therefore compute asymptotic expansions of all solutions of \eqref{eq:2.8a}-\eqref{eq:2.8c} in a neighbourhood of $P_1$ that lie on the unstable manifold. Requiring that $n \in (0,1)$ and $p < 0$, so that solutions leaving $P_1$ remain in $\mathcal{D}_1$, we obtain the following result.
\begin{lem}
Fix $c > 0$. For any $\alpha \geq 0$, the system \eqref{eq:2.8a}-\eqref{eq:2.8c}  has a unique solution such that, as $y \to - \infty$,
\begin{equation}
\begin{aligned}
n(y) & = 1-e^{\lambda_2 y} + \mathcal{O}(e^{(\lambda_2 + \mu)y}), \\
p(y) & = -\lambda_2e^{\lambda_2 y} + \mathcal{O}(e^{(\lambda_2 + \mu)y}),\\
m(y) & = \alpha e^{\lambda_3 y} + \mathcal{O}(e^{(\lambda_3 + \mu)y}),
\label{eq:3.3}
\end{aligned}
\end{equation}
where $\lambda_2$ and $\lambda_3$ are given by \eqref{eq:3.2a} and $\mu = \min(\lambda_2,\lambda_3) >0$. 
\label{lem:3.1}
\end{lem}

\begin{rem}
The free parameter, $\alpha$, arises because the form taken by the unstable manifold at $P_1$ does not impose any condition on $m$. In a sense, the choice of $\alpha$ is a choice of how fast $m$ increases from $0$ and, accordingly, $\alpha$ will influence the value that $m$ attains at $y=+\infty$. We illustrate this in Figure \ref{fig:3.1.1} and present some corresponding travelling wave profiles in Supplementary Material S3. In addition, by Remark \ref{rem:2.1}, it is clear that $\alpha = 0$ is the unique value of the shooting parameter such that the solution of \eqref{eq:2.8a}-\eqref{eq:2.8c} that satisfies \eqref{eq:3.3} stays in a region where $n \in (0,1)$, $p < 0$ and $m=0$ and satisfies the asymptotic conditions \eqref{eq:2.7d} for $\bar{m} =  0$.
\label{rem:3.1}
\end{rem}

\begin{figure}[!ht]
\begin{subfigure}[t]{0.5\textwidth}
\centering
\includegraphics[scale=0.17]{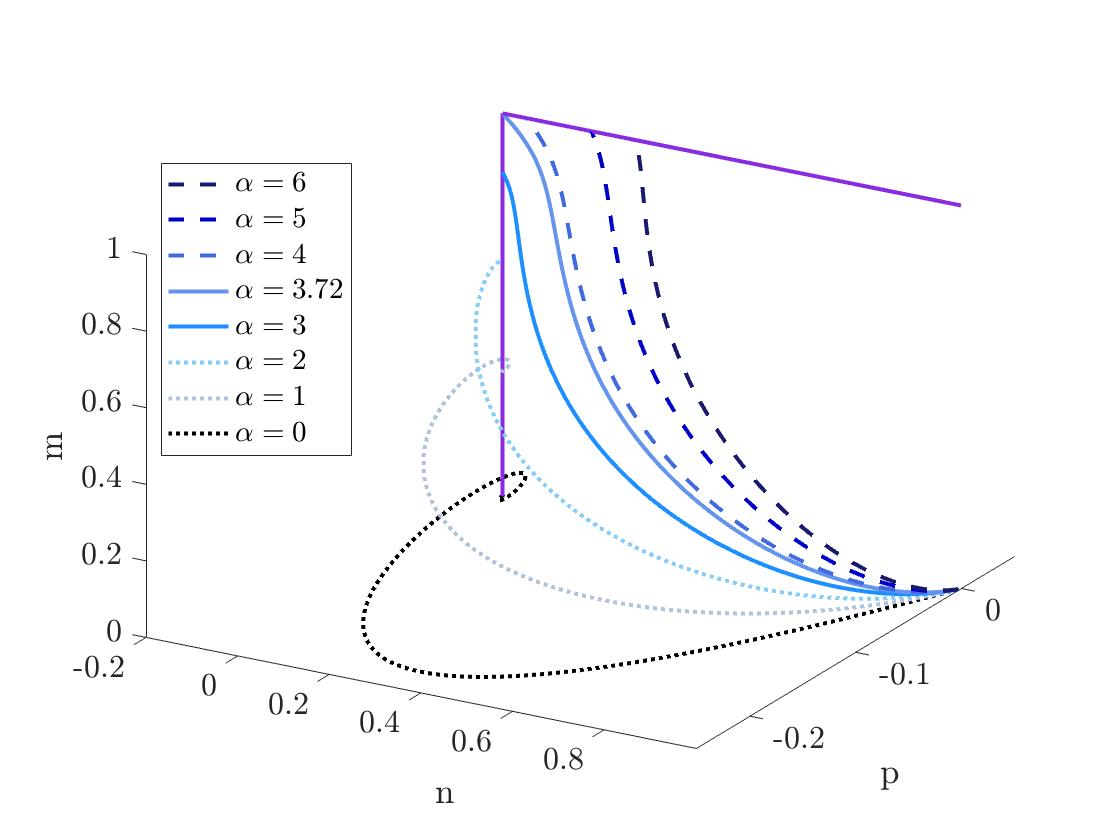}
\caption{}
\label{fig:3.1.1a}
\end{subfigure}
\hfill
\begin{subfigure}[t]{0.5\textwidth}
\centering
\includegraphics[scale=0.17]{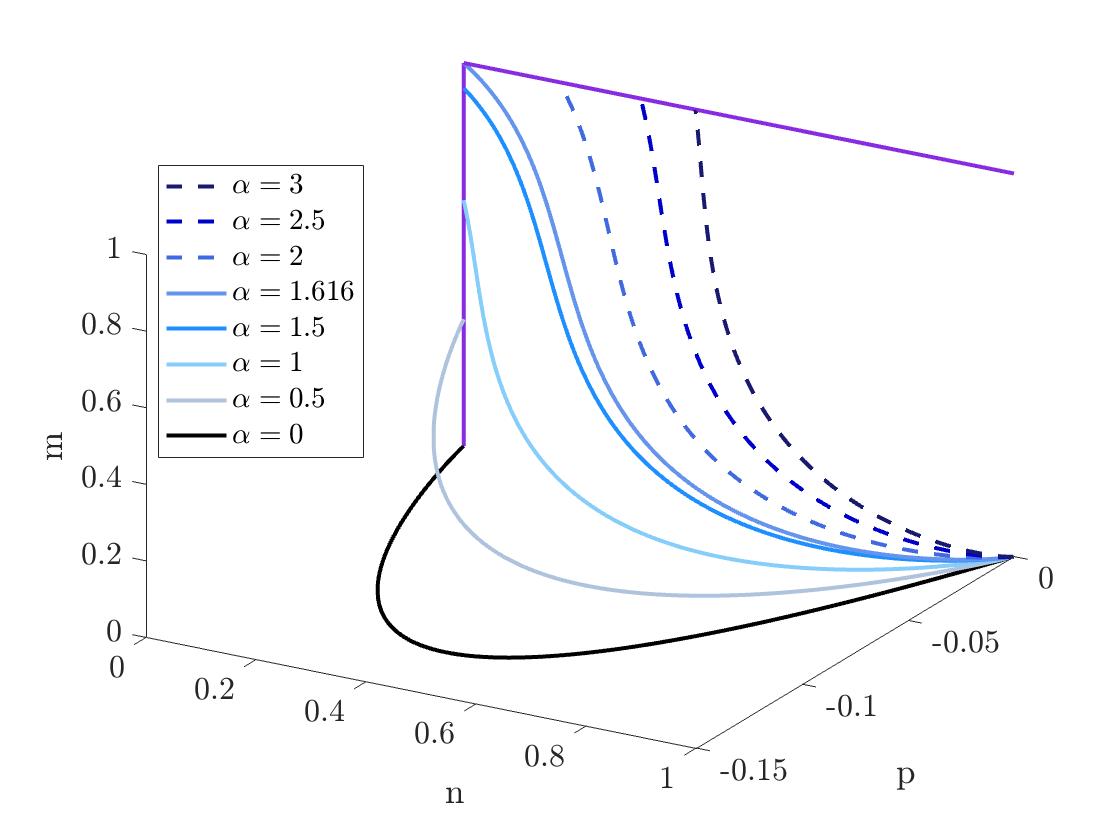} 
\caption{}
\label{fig:3.1.1b}
\end{subfigure}
\caption{Solutions of \eqref{eq:2.8a}-\eqref{eq:2.8c} subject to the asymptotic conditions \eqref{eq:3.3} for different values of the shooting parameter $\alpha$, $\kappa =1$ and $c = 1$ (a) or $c =2$ (b). The purple lines are the two continua of steady states of the system \eqref{eq:2.8a}-\eqref{eq:2.8c}, given by $(0,0,\bar{m})$, $\bar{m} \in [0,1]$, and $(\bar{n},0,1)$, $\bar{n} \in [0,1]$, respectively. Since $(n,m) = (\bar{n},1)$, $\bar{n} \in [0,1)$, are not spatially homogeneous steady states of \eqref{eq:1.3}, the dashed curves represent solutions that are not TWS of system \eqref{eq:1.3}. The dotted curves represent physically unrealistic solutions for which the $n$-component becomes negative. The values $m$ and $n$ attain at infinity appear to increase monotonically (between $0$ and $1$) with $\alpha$.}
\label{fig:3.1.1}
\end{figure}

Now, the idea is to view solutions of \eqref{eq:2.8a}-\eqref{eq:2.8c} that satisfy \eqref{eq:3.3} as functions of $\alpha$, which we define to be our shooting parameter, and $c$, which is the wave speed. In particular, we denote by $(n_{\alpha,c},p_{\alpha,c},m_{\alpha,c})$ the unique solution of \eqref{eq:2.8a}-\eqref{eq:2.8c} satisfying \eqref{eq:3.3}. Our first result, whose proof is in Supplementary Material S1, is the following:
\begin{lem}
If the solution $(n_{\alpha,c},p_{\alpha,c},m_{\alpha,c})$  is defined on some interval $J\coloneqq(-\infty,y_0)$, with $y_0 \in \mathbb{R}$, and satisfies $n_{\alpha,c}(y) >0$ for all $y \in J$, then $(n_{\alpha,c}(y),p_{\alpha,c}(y),m_{\alpha,c}(y)) \in \mathcal{D}_1$ for all $y\in J$.
\label{lem:3.2}
\end{lem}

Given Lemma \ref{lem:3.2}, we introduce the following variable for any $\alpha > 0$ and $c>0$:
\begin{equation}
T(\alpha,c) \coloneqq \sup{\lbrace y_0 \in \mathbb{R} \mid n_{\alpha,c}(y) >0 \text{ for all } y < y_0\rbrace} \in \mathbb{R} \cup \{+\infty\}.
\label{eq:3.4}
\end{equation}
Then, Lemma \ref{lem:3.2} implies that only one of the following holds:
\begin{itemize}
\item $T(\alpha,c) < +\infty$, so $n_{\alpha,c}(T(\alpha,c)) = 0$ and $p_{\alpha,c}(T(\alpha,c)) < 0$. In this case, $n_{\alpha,c}(y)$ becomes negative for some $y > T(\alpha,c)$ and $(n_{\alpha,c},p_{\alpha,c},m_{\alpha,c})$ does not represent a valid TWS; we disregard these values of the shooting parameter $\alpha$.

\item $T(\alpha,c) = +\infty$, which means that we have a global solution which stays in $\mathcal{D}_1$ for all $y \in \mathbb{R}$. We are interested in finding TWS for these values of $\alpha$. 
\end{itemize}

\begin{rem}
Given $\bar{m} \in (0,1)$, Lemma \ref{lem:3.2} provides a condition under which solutions of \eqref{eq:2.8a}-\eqref{eq:2.8c} that satisfy \eqref{eq:3.3} remain in $\mathcal{D}_1$, but not necessarily in $\mathcal{D}_{\bar{m}} \subset \mathcal{D}_1$. In particular, even if $n_{\alpha,c}(y) > 0$ for all $y\in J$, a solution can leave $\mathcal{D}_{\bar{m}}$. In that case, for a solution of \eqref{eq:2.8a}-\eqref{eq:2.8c} that satisfies \eqref{eq:3.3} to converge to $(0,0,\bar{m})$ as $y \to +\infty$, we must have $n(y) < 0$ for some values of $y$ (since $m$ is increasing for positive $n$). Therefore, searching for solutions that satisfy $T(\alpha,c) =+ \infty$ is a necessary condition for the existence of physically realistic TWS that converge to $(0,0,\bar{m})$ as $y \to +\infty$, but not a sufficient one ($m$ may not attain the value $\bar{m}$ for positive $n$). 
\end{rem}

\subsection{Monotonicity of solutions with respect to the shooting parameter}

A key component of our analysis is that solutions of system \eqref{eq:2.8a}-\eqref{eq:2.8c} satisfying \eqref{eq:3.3} are monotonic functions of the shooting parameter, $\alpha$, provided $n > 0$. This result, which is proven in Supplementary Material S1, can be formulated as follows:
\begin{lem}
Fix $c > 0$. If $\alpha_2 > \alpha_1 > 0$, then $T(\alpha_2,c) \geq T(\alpha_1,c)$ and the solutions of \eqref{eq:2.8a}-\eqref{eq:2.8c} defined by \eqref{eq:3.3} satisfy
\begin{equation}
n_{\alpha_2,c}(y) > n_{\alpha_1,c}(y), \quad m_{\alpha_2,c}(y) > m_{\alpha_1,c}(y),
\label{eq:3.5}
\end{equation}
for all $y \in (-\infty, T(\alpha_1,c))$. 
\label{lem:3.3}
\end{lem}

Lemma \ref{lem:3.3} shows that, for fixed $c > 0$, $T(\alpha,c)$ is an increasing function of $\alpha$. Since we seek TWS that satisfy $T(\alpha,c) =+ \infty$, we define the following critical value of $\alpha$, which depends on $c$:
\begin{equation}
\alpha_0(c) \coloneqq \inf{\lbrace \alpha > 0 \mid T(\alpha,c) = + \infty \rbrace} \in [0,+\infty) \cup \{+\infty\}.
\label{eq:3.6}
\end{equation}
We then characterise $\alpha_0$ as a function of $c$ in the following lemma, whose proof is included in Supplementary Material S1:
\begin{lem}
If $c\geq2$, then $\alpha_0(c)= 0$. If $0<c<2$, then $0 < \alpha_0(c) <+\infty$.
\label{lem:3.4}
\end{lem}

Lemma \ref{lem:3.4} ensures that for all $c > 0$ there exist some values of $\alpha$ for which $T(\alpha,c) = + \infty$, and thus for all $c>0$ there exist some TWS. We still need to elucidate the behaviour of solutions at infinity to determine which TWS exist for each $c >0$. 

\begin{rem}
The proof of Lemma \ref{lem:3.4} relies on showing that, for any $0<c<2$, we can choose $\alpha(c) > 0$ sufficiently large such that there exists a solution of system \eqref{eq:2.8a}-\eqref{eq:2.8c} that satisfies \eqref{eq:3.3}, remains in region $\mathcal{D}_1$ and converges to $(\bar{n},0,1)$ as $y \to +\infty$, with $\bar{n}\in (0,1)$. Such solutions are not TWS, but their existence will be crucial in proving the existence of the TWS we seek. 
\label{rem:3.3}
\end{rem}

\subsection{Behaviour of solutions at infinity}
By Lemma \ref{lem:3.4}, we know that, for any $c > 0$, there exist solutions of system \eqref{eq:2.8a}-\eqref{eq:2.8c} that satisfy \eqref{eq:3.3} and remain in region $\mathcal{D}_1$ for all $y\in \mathbb{R}$. It remains to characterise the behaviour of these solutions as $y \to +\infty$, and, in so doing, to establish whether they are TWS. Denoting the limits of components of the solution at infinity as
\begin{equation*}
    n_\infty(\alpha,c) \coloneqq \lim_{y\to +\infty} n_{\alpha,c}(y), \quad m_\infty(\alpha,c) \coloneqq \lim_{y\to +\infty} m_{\alpha,c}(y), \quad p_\infty(\alpha,c) \coloneqq \lim_{y\to +\infty} p_{\alpha,c}(y),
\end{equation*}
we introduce the following lemma.

\begin{lem}
If $T(\alpha,c) = \infty$, then the following limits exist:
\begin{equation}
n_\infty(\alpha,c) \in [0,1), \quad m_\infty(\alpha,c) \in [0,1] \quad \text{and} \quad p_\infty(\alpha,c) = 0.
\label{eq:3.7}
\end{equation}
Moreover, if $m_\infty(\alpha,c) \in [0,1)$, then $n_\infty(\alpha,c) = 0$, and, if $n_\infty(\alpha,c) \in (0,1)$, then $m_\infty(\alpha,c) = 1$.
\label{lem:3.5}
\end{lem}

\begin{proof}
Suppose that $T(\alpha,c) =+ \infty$. Then, by Lemma \ref{lem:3.2}, the solution $(n_{\alpha,c},p_{\alpha,c},m_{\alpha,c})$ of \eqref{eq:2.8a}-\eqref{eq:2.8c} that satisfies \eqref{eq:3.3} stays in region $\mathcal{D}_1$, defined by \eqref{eq:2.9}, for all $y\in \mathbb{R}$. This implies that $n_{\alpha,c}(y) > 0$ and  $n_{\alpha,c}'(y) = p_{\alpha,c}(y) < 0 \; \forall y\in \mathbb{R}$ and, as $n_{\alpha,c}$ satisfies \eqref{eq:3.3}, we also have ${n_{\alpha,c}(y) <1 \; \forall y \in \mathbb{R}}$. This means that $n$ is strictly decreasing and bounded above and below by $1$ and $0$, respectively, in $\mathbb{R}$ and, thus, $n_\infty(\alpha,c) \in [0,1)$ exists. Since $p_{\alpha,c} \coloneqq n_{\alpha,c}'$, the latter implies that $p_\infty(\alpha,c)$ exists and $p_\infty(\alpha,c)=0$; otherwise $n_{\alpha,c}(y)$ cannot converge to a finite limit as $y \to +\infty$.  Moreover, by the definition of $\mathcal{D}_1$, we have that $0 < m_{\alpha,c}(y) < 1 \; \forall y \in \mathbb{R}$ and, since ${n_{\alpha,c}(y) > 0 \; \forall y \in \mathbb{R}}$, this implies that $m'_{\alpha,c}(y) = \frac{\kappa}{c} m_{\alpha,c}(y) (1-m_{\alpha,c}(y)) n_{\alpha,c}(y) > 0 \; \forall y \in \mathbb{R}$. Hence, $m_{\alpha,c}$ is strictly increasing and bounded above and below by $1$ and $0$, respectively, in $\mathbb{R}$, and thus $m_\infty(\alpha,c) \in (0,1]$ exists. Finally, if the components of $(n_\infty(\alpha,c),p_\infty(\alpha,c),m_\infty(\alpha,c))$ exist, then this must be an equilibrium of \eqref{eq:2.8a}-\eqref{eq:2.8c}. The final statement of the lemma follows. 
\end{proof}

Lemma \ref{lem:3.5} defines the possible limits of solutions $(n_{\alpha,c},p_{\alpha,c},m_{\alpha,c})$ of \eqref{eq:2.8a}-\eqref{eq:2.8c} that satisfy \eqref{eq:3.3} and remain in region $\mathcal{D}_1$. We must now determine for which values of $c > 0$ we can find $\alpha(c) > 0$ such that the corresponding solution $(n_{\alpha,c},p_{\alpha,c},m_{\alpha,c})$: 
\begin{enumerate}
    \item remains in $\mathcal{D}_1$ and converges to $(0,0,1)$ as $y \to +\infty$, or,
    \item for each $\bar{m} \in (0,1)$, remains in $\mathcal{D}_{\bar{m}}$ and converges to $(0,0,\bar{m})$ as $y \to +\infty$.
\end{enumerate}
We consider these cases separately in the two sections that follow.

\subsection{Solutions converging to the equilibrium point $(0,0,1)$} \label{sec:3.4}
In this section, we show that, for each $c > 0$, there exists a unique value of $\alpha > 0$ such that the solution $(n_{\alpha,c}, p_{\alpha,c},m_{\alpha,c})$ of system \eqref{eq:2.8a}-\eqref{eq:2.8c} satisfying \eqref{eq:3.3} remains in region $\mathcal{D}_1$ and converges to the equilibrium point $P_2 \coloneqq (0,0,1)$ as $y \to +\infty$. This then allows us to draw conclusions on the existence and uniqueness of TWS that satisfy the asymptotic conditions \eqref{eq:2.2a}.

By Remark \ref{rem:3.3}, we have that, for any $c > 0$, there exists $\alpha(c) > 0$ sufficiently large such that the solution of system \eqref{eq:2.8a}-\eqref{eq:2.8c} satisfying \eqref{eq:3.3} remains in region $\mathcal{D}_1$ and satisfies $m_\infty =  1$. We can therefore define
\begin{equation}
\alpha_1(c) \coloneqq \inf{\lbrace \alpha > \alpha_0(c) \mid m_\infty(\alpha,c) =  1\rbrace} \in [0,+\infty) \cup \{+\infty\},
\label{eq:2.4.2}
\end{equation}
and prove the following result:
\begin{lem}
For any $c >0$, we have $0 < \alpha_1(c) < +\infty$.
\label{lem:3.6}
\end{lem}

\begin{proof}
Fix $c > 0$. By Remark \ref{rem:3.3}, we know that there exists $\alpha = \alpha(c) <+ \infty$ large enough such that $m_\infty(\alpha,c) =1$ and, hence, $\alpha_1(c) <+ \infty$. If $0 < c < 2$, then we know by Lemma \ref{lem:3.4} that $\alpha_0(c) > 0$ and, therefore, by the definition of $\alpha_1(c)$, we must have $\alpha_1(c) \geq \alpha_0(c) >0$. 
 
If $c \geq 2$, then suppose, for a contradiction, that $\alpha_1(c) = 0$. A linear stability analysis about the equilibrium point $(0,0,\bar{m})$ with $\bar{m}\in [0,1)$ shows that it is non-hyperbolic with two negative eigenvalues $\lambda_1,\lambda_2$ and one zero eigenvalue $\lambda_3$:
\begin{equation}
\lambda_{1,2} = \frac{-c \pm \sqrt{c^2 - 4(1-\bar{m})}}{2}, \quad \lambda_3 = 0.
\label{eq:3.18}
\end{equation}
Therefore, by the Centre Manifold Theorem, in a small, open neighbourhood of $(0,0,\bar{m})$ with $\bar{m} \in [0,1)$, there exists a two-dimensional stable manifold spanned by the eigenvectors $\vec{v}_{1,2}$ corresponding to $\lambda_{1,2}$. In this neighbourhood, there also exists a one-dimensional centre manifold spanned by $\vec{v}_3 = (0,0,1)^\top$, which comprises the family of equilibria $(0,0,\bar{m}')$ with $\bar{m}'$ sufficiently close to $\bar{m}$. Therefore, for fixed $0 < \epsilon < 1$, we can find a neighbourhood $\Omega$ of $(0,0,0)$ that is foliated by two-dimensional stable leaves over a one-dimensional centre manifold, composed of points of the form $(0,0,\bar{m})$ with $ 0 \leq \bar{m} < \epsilon$. Then, any solution that enters $\Omega$ converges to $(0,0,\bar{m})$ as $y \to +\infty$ for some $\bar{m}$ that satisfies $0 \leq \bar{m} < \epsilon$. 

Now, by Remark \ref{rem:3.1}, $\alpha_1(c)=0$ implies that $(n_{\infty}({\alpha_1(c),c}),p_\infty({\alpha_1(c),c}),m_\infty({\alpha_1(c),c}))=(0,0,0)$. Thus, we can find $\bar{y}\in \mathbb{R}$ large enough such that ${(n_{\alpha_1(c),c}(y),p_{\alpha_1(c),c}(y),m_{\alpha_1(c),c}(y)) \in \Omega}$ for all $y \geq \bar{y}$. By continuity of solutions with respect to $\alpha$, we can find $\delta >0$ such that $(n_{\alpha,c}(\bar{y}),p_{\alpha,c}(\bar{y}),m_{\alpha,c}(\bar{y})) \in \Omega$ for any $0 < \alpha < \delta$. This implies that, for any such $\alpha$, $(n_{\alpha,c}(y),p_{\alpha,c}(y),m_{\alpha,c}(y))$ converges to $(0,0,\bar{m}) \in \Omega$ as $y \to +\infty$. Since $0 \leq \bar{m}<  \epsilon$ by our choice of $\Omega$, there exists $0< \alpha < \delta $ such that $0 \leq m_\infty(\alpha,c) < \epsilon$. However, since $\alpha_1(c)=0$, we must have $m_\infty(\alpha,c) = 1$ for all $\alpha > 0$ and we have reached the desired contradiction.
\end{proof}

Lemma \ref{lem:3.6} ensures that for any $c > 0$ and $\alpha \geq \alpha_1(c)$, the solution $(n_{\alpha,c},p_{\alpha,c},m_{\alpha,c})$ of \eqref{eq:2.8a}-\eqref{eq:2.8c}, subject to the asymptotic conditions \eqref{eq:3.3}, stays in region $\mathcal{D}_1$ and satisfies $(n_{\infty}(\alpha,c),p_{\infty}(\alpha,c),  m_{\infty}(\alpha,c)) = (n_{\infty}(\alpha,c), 0,1)$, where $n_{\infty}(\alpha,c) \in [0,1)$. We would now like to show that, for any $c > 0$, there exists a unique $\alpha \geq \alpha_1(c) $ such that $n_{\infty}(\alpha,c) = 0$. 

For the rest of this section, we suppose that $\alpha \geq \alpha_1(c)$. A linear stability analysis at the equilibrium point $P_2$ shows that $P_2$ is non-hyperbolic, with one negative eigenvalue and two zero eigenvalues. Therefore, at $P_2$, we have a one-dimensional stable manifold, $\mathcal{W}_S \subset \mathbb{R}^3$, generated by the eigenvector $\vec{v}_1 = (1/c,1,0)^\top$ associated with $\lambda_1 = -c$. We also have a two-dimensional centre manifold, $\mathcal{W}_C \subset \mathbb{R}^3$, which is tangent at $P_2$ to the subspace spanned by the eigenvectors $\vec{v}_2 = (1, 0,0)^\top$ and $\vec{v}_3=(0,0,1)^\top$ associated with $\lambda_2=\lambda_3 = 0$.
Solutions of \eqref{eq:2.8a}-\eqref{eq:2.8c} that satisfy \eqref{eq:3.3} and remain in a small enough neighbourhood of $P_2$ for all sufficiently large $y>0$ converge to $\mathcal{W}_C$. Therefore, in order to study the dynamics around $P_2$, we perform a nonlinear local stability analysis. We begin by transforming system \eqref{eq:2.8a}-\eqref{eq:2.8c} into normal form by introducing the following variables
\begin{equation}
\tilde{n}(y) = n(y) + p(y)/c, \quad \tilde{p}(y) = p(y), \quad \tilde{m}(y) = 1-m(y),
\label{eq:3.11}
\end{equation}
which satisfy the following system:
\begin{subnumcases}{}
{  \diff{\tilde{n}}{y} = -\frac{1}{c}\tilde{m}\left(\tilde{n}-\frac{\tilde{p}}{c}\right)\left(1-\tilde{n}+ \frac{\tilde{p}}{c}\right)},  \label{eq:3.12a}
\\
{\diff{\tilde{p}}{y} = -c \tilde{p} - \tilde{m}\left(\tilde{n}-\frac{\tilde{p}}{c}\right)\left(1-\tilde{n}+ \frac{\tilde{p}}{c}\right),}\label{eq:3.12b}
\\
{\diff{\tilde{m}}{y} = - \frac{\kappa}{c} \tilde{m} (1-\tilde{m}) \left(\tilde{n}-\frac{\tilde{p}}{c}\right).} \label{eq:3.12c}
\end{subnumcases}

Then, we know that, in a neighbourhood of the origin, the centre manifold can be described by a function $\mathcal{P}(\tilde{n},\tilde{m})$ such that $(\tilde{n},\tilde{p},\tilde{m}) \in \mathcal{W}_C$ if and only if $\tilde{p} = \mathcal{P}(\tilde{n},\tilde{m})$, where
\begin{equation}
\mathcal{P}(\tilde{n},\tilde{m}) = -\frac{1}{c} \tilde{n}\tilde{m} (1 + \mathcal{O}(|\tilde{n}| + |\tilde{m}|)).
\label{eq:3.13}
\end{equation}

Using this expression for the centre manifold in a neighbourhood of the origin, we must now prove that there is a solution of system \eqref{eq:3.12a}-\eqref{eq:3.12a} converging to the centre manifold $\mathcal{W}_C$ that converges to the origin as $y \to +\infty$. We are interested in solutions $(n,p,m)$ of \eqref{eq:2.8a}-\eqref{eq:2.8c} that satisfy \eqref{eq:3.3} and remain in region $\mathcal{D}_1$ for all $y \in \mathbb{R}$. Equivalently, we seek solutions $(\tilde{n},\tilde{p},\tilde{m})$ of \eqref{eq:3.12a}-\eqref{eq:3.12c} that satisfy \eqref{eq:3.3} and lie on a manifold $\mathcal{W}_C^+ \subset \mathcal{W}_C$, where
\begin{equation}
\mathcal{W}_C^+ = \lbrace (\tilde{n},\tilde{p},\tilde{m}) \in \mathcal{W}_C \mid \tilde{n},\tilde{m}>0 \rbrace.
\label{eq:3.14}
\end{equation}

The following lemma characterises such solutions that converge to the origin as $y \to +\infty$ (the proof is included in Supplementary Material S1). 
\begin{lem}
Up to translations in the variable $y$, there exists a unique solution of \eqref{eq:3.12a}-\eqref{eq:3.12c} that satisfies the asymptotic conditions \eqref{eq:3.3}, lies on the centre manifold $\mathcal{W}^{+}_{C}$, and whose components converge to zero as $y\to+\infty$, such that
\begin{equation}
\tilde{n}(y) = \frac{c}{\kappa y} + \mathcal{O}\left(\frac{1}{y^2}\right) \quad \text{and} \quad  \tilde{m}(y) = \frac{c}{ y}+ \mathcal{O}\left(\frac{1}{y^2}\right).
\label{eq:3.15}  
\end{equation}
\label{lem:3.7}
\end{lem}

Lemma \ref{lem:3.7} establishes the existence of at least one solution of \eqref{eq:2.8a}-\eqref{eq:2.8c} that satisfies \eqref{eq:3.3}, stays in region $\mathcal{D}_1$ and converges to $(1,0,0)$ as $y\to +\infty$. Furthermore, this solution is uniquely determined on the centre manifold $\mathcal{W}_C^+$. Given that any solution of \eqref{eq:2.8a}-\eqref{eq:2.8c} that satisfies \eqref{eq:3.3}, stays in region $\mathcal{D}_1$ and converges to $(0,0,1)$ as $y\to +\infty$ must do so via $\mathcal{W}_C^+$ and, given the monotonicity result of Lemma \ref{lem:3.3}, it is easy to prove the following lemma.

\begin{lem}
Given any $c>0$, there exists at most one value of $\alpha \geq \alpha_1(c)$ of the shooting parameter such that the solution $(n_{\alpha,c}(y),p_{\alpha,c}(y) ,m_{\alpha,c}(y) )$ of \eqref{eq:2.8a}-\eqref{eq:2.8c} satisfying the asymptotic properties in Lemma \ref{lem:3.1} converges to $P_2 = (0,0,1)$ as $y\to+\infty$.
\label{lem:3.8}
\end{lem}

The proof of Lemma \ref{lem:3.8} is identical to that of Lemma 2.13 in \cite{10} and is, therefore, omitted.

Exploiting the continuity of solutions with respect to the shooting parameter, $\alpha$, we can extend Lemma \ref{lem:3.8} to determine the unique value of $\alpha$, given $c > 0$, for which the solution $(n_{\alpha,c}(y),p_{\alpha,c}(y) ,m_{\alpha,c}(y) )$ of \eqref{eq:2.8a}-\eqref{eq:2.8c} that satisfies \eqref{eq:3.3} converges to $(0,0,1)$ as $y\to+\infty$. For the proof of the following result, we refer the reader to the proof of Lemma 2.14 in \cite{10}.

\begin{lem}
Given any $c>0$, if $\alpha = \alpha_1(c)$, then the solution $(n_{\alpha,c}(y),p_{\alpha,c}(y) ,m_{\alpha,c}(y) )$ of \eqref{eq:2.8a}-\eqref{eq:2.8c} satisfying the asymptotic properties in Lemma \ref{lem:3.1} converges to $P_2 = (0,0,1)$ as $y\to+\infty$.
\label{lem:3.9}
\end{lem}

Using Lemma \ref{lem:3.9} and reversing the change of variables \eqref{eq:2.5}, it is straightforward to construct a unique (up to translation) solution $(\mathcal{N}(\xi), \mathcal{M}(\xi)) =(n(\Phi(\xi)), m(\Phi(\xi)))$ of system \eqref{eq:2.1a}-\eqref{eq:2.1b} that satisfies the asymptotic conditions \eqref{eq:2.2a}. This leads to our first main result, whose proof is provided in Supplementary Material S1:
\begin{thm}
Fix $\kappa > 0$. For any $c > 0$, system $\eqref{eq:1.3}$ has a weak TWS $(\mathcal{N},\mathcal{M};c)$ connecting $(1,0)$ and $(0,1)$. This solution is unique (up to translation), and $\mathcal{N}$ and $\mathcal{M}$ are monotonically strictly decreasing and increasing functions of $\xi = x - ct$, respectively. 
\label{thm:3.10}
\end{thm}

\subsection{Solutions converging to the equilibrium point $(0,0,\bar{m})$ with ${\bar{m}\in [0,1)}$}\label{sec:3.5}
In this section, we consider solutions of system \eqref{eq:2.8a}-\eqref{eq:2.8c} subject to \eqref{eq:3.3} that stay in region $\mathcal{D}_{\bar{m}} \subset \mathcal{D}_1$ and converge to $(0,0,\bar{m})$ for $\bar{m} \in (0,1)$ as $y \to +\infty$. Using arguments similar to those for the previous case, we can show that, for all $\bar{m} \in [0,1)$, there exists a strictly positive, real-valued wave speed above which the solutions we seek exist and are unique. We will refer to this wave speed as the \textit{minimal wave speed} and we will observe that it depends on $\kappa$, the rescaled degradation rate of ECM. In particular, given $\kappa > 0$, we denote the minimal wave speed by $c_\kappa^*(\bar{m})$ for each $\bar{m} \in[0,1)$. This will enable us to draw some conclusions on the existence and uniqueness of TWS that satisfy the asymptotic conditions \eqref{eq:2.2b}.

At this stage, we have no information about the possible values of $c_\kappa^*(\bar{m})$ for $\bar{m} \in(0,1)$ and $\kappa > 0$. More specifically, given $\kappa > 0$, we currently have $c_\kappa^*(\bar{m}) \in \mathbb{R}_+^*$ for each $\bar{m} \in(0,1)$. For $\bar{m}=0$, by Remark \ref{rem:2.1}, it is straightforward to show that $c_\kappa^*(0)=2$ for all $\kappa > 0$. To characterise the minimal wave speed for $\bar{m}\in(0,1)$, we begin by proving a non-existence result.
\begin{lem}
Fix $\kappa > 0$ and $\bar{m} \in (0,1)$. If $0 < c < 2 \sqrt{1-\bar{m}}$, then there is no $\alpha' \in [\alpha_0(c),\infty)$ such that the solution $(n_{\alpha',c},p_{\alpha',c},m_{\alpha',c})$ of \eqref{eq:2.8a}-\eqref{eq:2.8c} that satisfies the asymptotic properties in Lemma \ref{lem:3.1} converges to $(0,0,\bar{m})$ as $y \to +\infty$.
\label{lem:3.11}
\end{lem}
\begin{proof}
Fix $\kappa > 0$ and $\bar{m} \in [0,1)$ and suppose that $ 0 < c < 2 \sqrt{1-\bar{m}}$. We suppose for a contradiction that there exists $\alpha' \in [\alpha_0(c),\infty)$ such that the solution $(n_{\alpha',c},p_{\alpha',c},m_{\alpha',c})$ of \eqref{eq:2.8a}-\eqref{eq:2.8c} that satisfies \eqref{eq:3.3} converges to $(0,0,\bar{m})$ as $y \to +\infty$. By the definition of $\alpha_0(c)$, this implies that $(n_{\alpha',c},p_{\alpha',c},m_{\alpha',c})$ stays in region $\mathcal{D}_1$ for all $y \in \mathbb{R}$. Now, we can choose $\epsilon > 0$ small enough such that $0 < c < 2 \sqrt{(1-\bar{m}-\epsilon)(1-\epsilon)}$ and we can also find $\bar{y}$ sufficiently large such that $n_{\alpha',c}(y) < \epsilon$ and $m_{\alpha',c}(y) < \bar{m}+\epsilon$ for all $y \geq \bar{y}$. Solutions of the constant coefficient second order ODE
\begin{equation*}
\begin{aligned}
  & n'' + cn'+(1-\bar{m}-\epsilon)(1-\epsilon)n =0,\\
  \text{with} \quad & \lim_{y\to- \infty} n(y) = 1, \quad \lim_{y\to+\infty} n(y) = 0, \quad \lim_{y\to \pm \infty} n'(y) = 0,
\end{aligned}
\end{equation*}
have infinitely many zeros in $(\bar{y},\infty)$ (since its characteristic equation has complex roots). Since  $(1-\bar{m}-\epsilon)(1-\epsilon) < (1-m_{\alpha',c}(y))(1-n_{\alpha',c}(y))$ for all  $y \in (\bar{y},\infty)$, Sturm's Comparison Theorem implies that $n_{\alpha',c}(y)$ must also have infinitely many zeros in $(\bar{y},+\infty)$. Therefore, $(n_{\alpha',c},p_{\alpha',c},m_{\alpha',c})$ exits region $\mathcal{D}_1$ (and $\mathcal{D}_{\bar{m}}$), contradicting the assumption that $\alpha' \geq \alpha_0(c)$.
\end{proof}

Given $\kappa > 0$ and $\bar{m} \in (0,1)$, if the minimal wave speed, $c^*_\kappa(\bar{m})$, exists, then Lemma \ref{lem:3.11} yields a lower bound for $c^*_\kappa(\bar{m})$. More specifically, for all $\kappa >0$ and $\bar{m} \in (0,1)$, $c^*_\kappa(\bar{m}) \geq 2 \sqrt{(1-\bar{m})}$. 

\begin{lem}
Fix $\kappa > 0$. If $c \geq 2$, then, for all $\bar{m} \in \left(0, 1\right)$, there exists a unique $\alpha \in (\alpha_{0}(c),\alpha_1(c))$ such that the solution of \eqref{eq:2.8a}-\eqref{eq:2.8c} satisfying the asymptotic properties in Lemma \ref{lem:3.1} converges to $(0,0,\bar{m})$ as $y\to +\infty$.
\label{lem:3.12}
\end{lem}

\begin{proof}
Fix $\kappa >0$ and suppose that $c \geq 2$. By Lemmas \ref{lem:3.4} and \ref{lem:3.6}, we know that $0 = \alpha_0(c) < \alpha_1(c)$. Then, by the definition of $\alpha_0(c)$, we have that, for any $\alpha \in (\alpha_0(c),\alpha_1(c))$, $T(\alpha,c) = +\infty$, the solution $(n_{\alpha,c},p_{\alpha,c},m_{\alpha,c})$ of \eqref{eq:2.8a}-\eqref{eq:2.8c} satisfying \eqref{eq:3.3} stays in the region $\mathcal{D}_1$ for all $y \in \mathbb{R}$ by Lemma~\ref{lem:3.2}. Then, by Lemma \ref{lem:3.5}, we know that, for every $\alpha \in (\alpha_{0}(c),\alpha_1(c))$, the limits $n_\infty(\alpha,c)$, $p_\infty(\alpha,c)$, $m_\infty(\alpha,c)$ exist. In addition, by monotonicity of $n_{\alpha,c}$ and $m_{\alpha,c}$ with respect to $\alpha$ (see Lemma \ref{lem:3.3}) and the fact that $(n_\infty(\alpha_1(c),c), m_\infty(\alpha_1(
c),c)) =(0,1)$ by Lemma \ref{lem:3.9}, we must have 
\begin{align*}
   0 = n_\infty(\alpha_{0}(c),c) & \leq  n_\infty(\alpha,c) \leq n_\infty(\alpha_1(c),c) = 0 \\
  0 = m_\infty(\alpha_{0}(c),c) & \leq  m_\infty(\alpha,c) \leq m_\infty(\alpha_1(c),c) = 1.
\end{align*}
 We recall that $\alpha_{0}(c) = 0$ and $\alpha_1(c)$ are the unique values of the shooting parameter for which the solution of \eqref{eq:2.8a}-\eqref{eq:2.8c} that satisfies \eqref{eq:3.3} remains in region $\mathcal{D}_1$ and converges to $(0,0,0)$ and $(0,0,1)$ as $y \to +\infty$, respectively. Therefore, we find that, for every $\alpha \in (\alpha_{0}(c),\alpha_1(c))$, the limits for $(n_{\alpha,c},p_{\alpha,c},m_{\alpha,c})$ as $y \to +\infty$ must satisfy:
\begin{equation}
 n_\infty(\alpha,c) = 0, \quad m_\infty(\alpha,c) \in (0,1)\quad \text{and} \quad p_\infty(\alpha,c) = 0.
 \label{eq:3.16}
\end{equation}

We will now prove that the mapping $\alpha \mapsto m_\infty(\alpha,c)$ is continuous and strictly increasing on $[\alpha_{0}(c),\alpha_1(c)]$. Choose $\alpha', \alpha''$ such that $\alpha_{0}(c) < \alpha' < \alpha'' < \alpha_1(c)$. Suppose, for a contradiction, that $$\displaystyle \lim_{y \to +\infty} m_{\alpha',c}(y) = m_\infty(\alpha',c) =\bar{m}= m_\infty(\alpha'',c) = \lim_{y \to +\infty} m_{\alpha'',c}(y).$$ Irrespective of the asymptotic conditions \eqref{eq:3.3}, we can solve equation \eqref{eq:2.8c} for $y \in \mathbb{R}$ and impose $\displaystyle \lim_{y \to +\infty} m(y) = \bar{m} \in (0,1)$ to obtain:
\begin{equation}
   m(y) = \left[1+ \frac{1-\bar{m}}{\bar{m}}\exp{\left(\frac{\kappa}{c} \int_y^\infty n(s) \mathrm{d}s\right)} \right]^{-1}.
   \label{eq:3.17}
\end{equation}
Any solution for which $\displaystyle \lim_{y \to +\infty} m(y) = \bar{m} \in (0,1)$ must therefore take the form \eqref{eq:3.17}. Thus, $m_{\alpha',c}$ and $m_{\alpha'',c}$ take the form \eqref{eq:3.17}, with $n$ replaced by $n_{\alpha',c}$ and $n_{\alpha'',c}$, respectively. Now, by Lemma \ref{lem:3.3}, we know that $m_{\alpha',c}(y) < m_{\alpha'',c}(y)$ for all $y \in \mathbb{R}$ since $\alpha' < \alpha''$. We therefore have, for any $y \in \mathbb{R}$,
\begin{align}
 & \left[1+ \frac{1-\bar{m}}{\bar{m}}\exp{\left(\frac{\kappa}{c} \int_y^\infty n_{\alpha',c}(s) \mathrm{d}s\right)} \right]^{-1}  < \left[1+ \frac{1-\bar{m}}{\bar{m}}\exp{\left(\frac{\kappa}{c} \int_y^\infty n_{\alpha'',c}(s) \mathrm{d}s\right)} \right]^{-1}\\
 &\Rightarrow \int_y^\infty (n_{\alpha',c}(s)- n_{\alpha'',c}(s)) \mathrm{d}s > 0.\label{eq:star}
\end{align}
Since $\left(n_{\alpha',c}(y)- n_{\alpha'',c}(y)\right) < 0$ for all $y \in \mathbb{R}$ by Lemma \ref{lem:3.3}, the inequality \eqref{eq:star} cannot hold and we have reached a contradiction. Since we have that $m_\infty(\alpha',c) \leq m_\infty(\alpha'',c)$, by monotonicity of solutions with respect to $\alpha$, and that $m_\infty(\alpha',c) \neq m_\infty(\alpha'',c)$, by the above argument, we conclude that ${m_\infty(\alpha',c) < m_\infty(\alpha'',c)}$. This proves that the mapping $\alpha \mapsto m_\infty(\alpha,c)$ is strictly increasing on $(\alpha_{0}(c),\alpha_1(c))$. Using the fact that $\alpha_{0}(c)$ and $\alpha_1(c)$ are, respectively, the unique values of the shooting parameter for which the solution of \eqref{eq:2.8a}-\eqref{eq:2.8c} given by Lemma \ref{lem:3.1} converges to $(0,0,0)$ and $(0,0,1)$ as $y \to +\infty$, we have that $\alpha \mapsto m_\infty(\alpha,c)$ is strictly increasing on $[\alpha_{0}(c),\alpha_1(c)]$.

We now prove that the mapping $\alpha \mapsto m_\infty(\alpha,c)$ is continuous on $[\alpha_{0}(c),\alpha_1(c)]$. For fixed $\alpha' \in [\alpha_{0}(c),\alpha_1(c))$, \eqref{eq:3.16} implies that $n_\infty(\alpha',c) = 0, p_\infty(\alpha',c) = 0$ and $m_\infty(\alpha',c) = \bar{m}' \in [0,1)$. In the proof of Lemma \ref{lem:3.6}, we performed a linear stability analysis about the equilibrium point $(0,0,\bar{m}')$ for $\bar{m}' \in [0,1)$. We showed that, for fixed $\epsilon > 0$, we can find a neighbourhood $\Omega$ of $(0,0,\bar{m}')$ that is foliated by two-dimensional stable leaves over a one-dimensional centre manifold, which comprises equilibria of the form $(0,0,\bar{m})$ for $|\bar{m} - \bar{m}'| < \epsilon$. Then, any solution that enters $\Omega$ converges to $(0,0,\bar{m})$ as $y \to +\infty$ for some $\bar{m}$ that satisfies $|\bar{m} - \bar{m}'| < \epsilon$. Since $(n_{\alpha',c}(y),p_{\alpha',c}(y),m_{\alpha',c}(y))$ converges to $(0,0,\bar{m}')$ as $y \to +\infty$, we can find $\bar{y}\in \mathbb{R}$ large enough such that $(n_{\alpha',c}(y),p_{\alpha',c}(y),m_{\alpha',c}(y)) \in \Omega$ for all $y \geq \bar{y}$. By continuity of solutions with respect to $\alpha$, we can find $\delta >0$ such that $(n_{\alpha'',c}(\bar{y}),p_{\alpha'',c}(\bar{y}),m_{\alpha'',c}(\bar{y})) \in \Omega$ for any $\alpha'' \in [\alpha_{0}(c),\alpha_1(c))$ such that $|\alpha'- \alpha''| < \delta$. This implies that, for any such $\alpha''$, $(n_{\alpha'',c}(y),p_{\alpha'',c}(y),m_{\alpha'',c}(y))$ converges to $(0,0,\bar{m}'') \in \Omega$ as $y \to +\infty$, for some $\bar{m}'' \neq \bar{m}'$ (since $m_\infty(\alpha,c)$ is strictly increasing with $\alpha$). By our choice of $\Omega$, $|\bar{m}'' - \bar{m}'| < \epsilon$, i.e. $|m_\infty(\alpha'',c) - m_\infty(\alpha',c)| < \epsilon$ for any $\alpha''\in [\alpha_{0}(c),\alpha_1(c))$ such that $|\alpha'- \alpha''| < \delta$. This proves continuity of the mapping $\alpha \mapsto m_\infty(\alpha,c)$ on $[\alpha_{0}(c),\alpha_1(c))$. 

We finally show continuity at $\alpha_1(c)$. We fix $\epsilon > 0$ and note that, since $m_\infty(\alpha_1(c),c) =1$, we can find $\bar{y} \in \mathbb{R}$ large enough such that $|m_{\alpha_1(c),c}(y) - 1| < \epsilon/2$ for all $y \geq \bar{y}$. By continuity of solutions with respect to $\alpha$, we can find $\delta >0$ such that ${|m_{\alpha',c}(\bar{y})- m_{\alpha_1(c),c}(\bar{y})|< \epsilon/2}$ for any $\alpha' \in [\alpha_{0}(c),\alpha_1(c)]$ satisfying $|\alpha'- \alpha_1(c)| < \delta$, i.e. for any $\alpha' \in (\alpha_1(c)-\delta, \alpha_1(c)]$. Therefore, we have that $|m_{\alpha',c}(\bar{y})- 1| < \epsilon$ for any $\alpha' \in (\alpha_1(c)-\delta, \alpha_1(c)]$. Moreover, for any ${\alpha' \in (\alpha_1(c)-\delta, \alpha_1(c)]}$, the function $m_{\alpha',c}(y)$ is strictly increasing for all $y \in \mathbb{R}$ and bounded above by $1$, so $m_\infty(\alpha',c) \in (m_{\alpha',c}(\bar{y}),1]$. In particular, for any $\alpha' \in (\alpha_1(c)-\delta, \alpha_1(c)]$, we have ${|m_\infty(\alpha',c)- 1| = |m_\infty(\alpha',c)- m_\infty(\alpha_1(c),c)|  < \epsilon}$. This proves continuity of the mapping ${\alpha \mapsto m_\infty(\alpha,c)}$ at $\alpha_1(c)$. 

We have now shown that the mapping $\alpha \mapsto m_\infty(\alpha,c)$ is strictly increasing and continuous on $[\alpha_{0}(c),\alpha_1(c)]$. Since $m_\infty(\alpha_{0}(c),c) = 0$ and $m_\infty(\alpha_1(c),c) =1$, application of the Intermediate Value Theorem enables us to conclude that, for any $\bar{m} \in (0,1)$, there exists a unique ${\alpha \in (\alpha_{0}(c),\alpha_1(c))}$ such that $m_\infty(\alpha,c) = \bar{m}$. 
\end{proof}

\begin{rem}
Using a similar proof to the above, we can generalise Lemma \ref{lem:3.12} to obtain the following result. Given $\kappa, c >0$, suppose that there exists a unique value of the shooting parameter, $\alpha \in [\alpha_0(c),\alpha_1(c))$, such that the solution $(n_{\alpha,c},p_{\alpha,c},m_{\alpha,c})$ of \eqref{eq:2.8a}-\eqref{eq:2.8c} satisfies \eqref{eq:3.3} and converges to $(0,0,\bar{m})$ as $y\to +\infty$ for some $\bar{m} \in [0,1)$. Then, for all $\bar{m}' \in (\bar{m},1)$, there exists a unique value of the shooting parameter, $\alpha' \in (\alpha,\alpha_1(c))$, for which the solution of \eqref{eq:2.8a}-\eqref{eq:2.8c} that satisfies \eqref{eq:3.3} stays in $\mathcal{D}_{\bar{m}'}$ and converges to $(0,0,\bar{m}')$ as $y\to +\infty$.
\label{rem:3.4}
\end{rem}

Lemma \ref{lem:3.12} implies that, for all $\bar{m} \in (0,1)$, the minimal wave speed, $c^*_\kappa(\bar{m})$, exists and is bounded above by $2$. Then, given $\bar{m} \in [0,1)$, for any $c \geq c_\kappa^*(\bar{m})$, we can define
\begin{equation}
\alpha_{\bar{m}}(c) \coloneqq {\lbrace \alpha \geq \alpha_0(c) \mid m_\infty(\alpha,c) =  \bar{m}\rbrace} \in [\alpha_0(c),\alpha_1(c)).
\label{eq:3.22}
\end{equation}

We now improve the upper bound on $c^*_\kappa(\bar{m})$ for $\bar{m} \in (0,1)$ by formulating a conjecture. We consider the following generalised Fisher-KPP equation with reaction term, $g$, of Fisher-KPP type:
\begin{equation}
\begin{aligned}
  & n'' + cn'+g(n) =0,\\
  \text{with} \quad & \lim_{y\to- \infty} n(y) = 1, \quad \lim_{y\to+\infty} n(y) = 0, \quad \lim_{y\to \pm \infty} n'(y) = 0.
\end{aligned}
\label{eq:3.25}
\end{equation}
One typically seeks TWS such that $n$ is monotonically decreasing, in which case we can invert $n(y)$ to obtain a function $Y(n), \, n\in [0,1]$. Considering the new variable $P(n) \coloneqq n'(Y(n))$, we obtain the following first order boundary value problem (BVP):
\begin{equation}
\begin{cases}
  P' = -c - \frac{g(n)}{P}, \\
  P(0)= 0,
\end{cases}
\label{eq:3.26}
\end{equation}
subject to $P(1) = 0, \, P(n) < 0 \; \forall n \in (0,1)$. Studying TWS of \eqref{eq:3.25} and solutions of \eqref{eq:3.26}, subject to their respective asymptotic and boundary conditions, is equivalent \cite{14}. Moreover, it is known that if $g''(n) < 0 \, \forall n \in [0,1]$, then \eqref{eq:3.26} subject to $P(1) = 0,  \, P(n) < 0 \; \forall n \in (0,1)$ has a unique solution if $c \geq 2\sqrt{g'(0)}$ \cite{43,44,45}. Therefore, TWS of \eqref{eq:3.25} exist and are unique if $c \geq 2\sqrt{g'(0)}$.

Returning to our original problem, by introducing $P(n) \coloneqq n'(Y(n))$ and $M(n) \coloneqq m(Y(n))$, we view the system \eqref{eq:2.8a}-\eqref{eq:2.8c} subject to the conditions \eqref{eq:2.2b} as the following BVP:
\begin{equation}
\begin{cases}
  P' = -c - \frac{(1-n)n(1-M(n))}{P}, \\
  M' = \frac{\kappa}{c}\frac{M(1-M)N}{P}, \\
  P(0)= 0, M(0) = \bar{m},
\end{cases}
\label{eq:3.27}
\end{equation}
subject to the additional conditions 
\begin{equation}
    P(n) < 0, \quad 0 < M(n) < \bar{m} \; \forall n \in (0,1), \quad P(1) = M(1) = 0.
\label{eq:3.27b}
\end{equation}
In Supplementary Material S2, we show that $g(n) = (1-n)n(1-M(n))$ is of Fisher-KPP type for $0 \leq M \leq \bar{m} < 1$ and that $g''(0) < 0$ if $\kappa \leq \kappa^*(\bar{m})$, where
\begin{equation}
    \kappa^*(\bar{m})\coloneqq \frac{(1-\bar{m})}{\bar{m}} \quad \forall \bar{m} \in (0,1).
    \label{eq:3.27c}
\end{equation}
We conjecture that, if $\kappa \leq \kappa^*(\bar{m})$, then $g''(n) < 0 \,\, \forall n \in [0,1]$. By the preceding result for the generalised Fisher-KPP equation, this would imply that, given $0 < \kappa \leq \kappa^*(\bar{m})$, the system \eqref{eq:2.8a}-\eqref{eq:2.8c} subject to the conditions \eqref{eq:2.2b} has unique TWS for $c \geq 2\sqrt{g'(0)}= 2\sqrt{1-\bar{m}}$. Now, given $\kappa > 0$, we let $m^*(\kappa) \coloneqq \frac{1}{\kappa +1}$. Noting that $0 < \kappa \leq \kappa^*(\bar{m})$ if and only if $0 < \bar{m} \leq m^*(\kappa)$, we formulate the following conjecture.

\begin{conj}
Fix $\kappa > 0$ and $\bar{m} \in \left(0, m^*(\kappa) \right]$. Given $c \geq 2 \sqrt{1-\bar{m}}$, there exists a unique $\alpha' \in~ [\alpha_0(c),\alpha_1(c))$ such that the solution $(n_{\alpha',c},p_{\alpha',c},m_{\alpha',c})$ of \eqref{eq:2.8a}-\eqref{eq:2.8c} that satisfies the asymptotic properties in Lemma \ref{lem:3.1} converges to $(0,0,\bar{m})$ as $y \to +\infty$. In particular, if $c = 2 \sqrt{1-\bar{m}}$, then $\alpha'=~\alpha_0(c)$. 
\label{conj:3.14}
\end{conj}

Conjecture \ref{conj:3.14} implies that, given $\kappa > 0$, there are values of $\bar{m} \in [0,1)$ such that the solutions of \eqref{eq:2.8a}-\eqref{eq:2.8c} that satisfy the asymptotic conditions \eqref{eq:2.2b} behave similarly to solutions of a generalised Fisher-KPP equation with reaction term $g(n)=(1-n)n(1-\bar{m})$. In particular, the minimal wave speed for these TWS is defined similarly to that of a generalised Fisher-KPP equation, i.e. it is the smallest value of $c > 0$ such that $(0,0,\bar{m})$ is a stable node, and not a stable spiral, for system \eqref{eq:2.8a}-\eqref{eq:2.8c}. In addition, using Lemmas \ref{lem:3.11} and \ref{lem:3.12} and Conjecture \ref{conj:3.14}, we make the hypothesis that, if $\bar{m} \in \left(m^*(\kappa),1\right)$ or, equivalently, if $\kappa > \kappa^*(\bar{m})$, then the minimal wave speed for TWS that converge to $(0,0,\bar{m})$ as $y \to +\infty$ should satisfy $c^*_\kappa(\bar{m}) \in \left(2 \sqrt{1-\bar{m}},2\right)$. In other words, in these cases, we expect that there is another mechanism that can lead to $n(y) < 0, \; y\in\mathbb{R}$, even if $(0,0,\bar{m})$ is a stable node for the system \eqref{eq:2.8a}-\eqref{eq:2.8c}.

The preceding hypothesis and Conjecture \ref{conj:3.14} are supported by numerical simulations of the PDE system \eqref{eq:1.3} and ODE system \eqref{eq:2.8a}-\eqref{eq:2.8c}. In Figure \ref{fig:3.5.1}, we show that solutions of system \eqref{eq:1.3} subject to the initial conditions \eqref{eq:1.3b} with $\bar{M}\in[0,1)$ evolve into travelling waves with constant propagation speed (see Supplementary Material S3 for corresponding travelling wave profiles). We observe that, for $0 < \kappa \leq \kappa^*(\bar{M})$, this speed is independent of $\kappa$, and, calculating the slopes of these lines, we find that it is approximately equal to $ 2\sqrt{1-\bar{M}}$. Additionally, when $\kappa > \kappa^*(\bar{M})$, we observe that the wave speed selected by the PDE increases with $\kappa$. We also solved numerically the system \eqref{eq:2.8a}-\eqref{eq:2.8c}, subject to the asymptotic conditions \eqref{eq:3.3}, for the same values of $\kappa > \kappa^*(\bar{M})$ and the respective values of the propagation speed estimated using the solutions of the PDE system (results not shown). We observed that, given $\kappa > \kappa^*(\bar{M})$, the wave speed selected by the PDE appears to correspond to the smallest wave speed such that the solution $(n,p,m)$ of the system \eqref{eq:2.8a}-\eqref{eq:2.8c}, subject to \eqref{eq:3.3}, satisfies $n(y) > 0 \; \forall y \in \mathbb{R}$ and converges to $(0,0,\bar{m})$, $\bar{m}=\bar{M}$.

\begin{figure}[!ht]
\begin{subfigure}[t]{0.5\textwidth}
\centering
\includegraphics[scale=0.18]{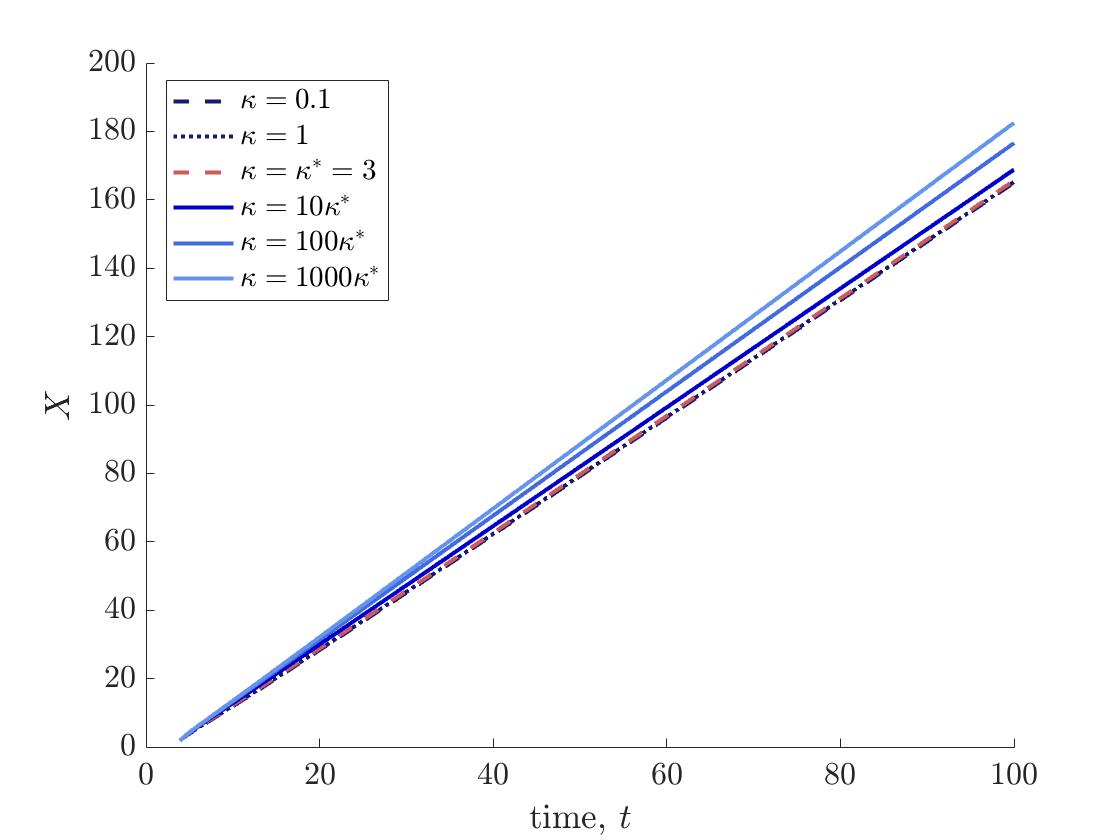} 
\caption{}
\label{fig:3.5.1a}
\end{subfigure}
\hfill
\begin{subfigure}[t]{0.5\textwidth}
\centering
\includegraphics[scale=0.18]{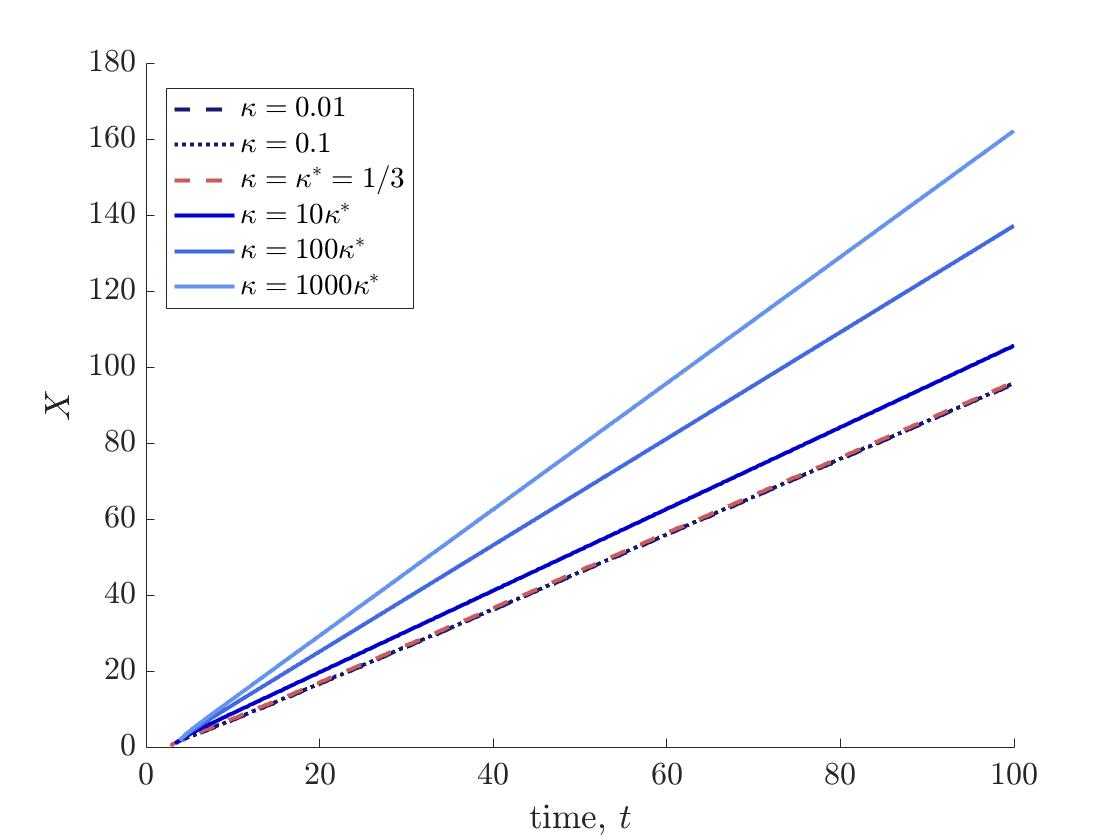}
\caption{}
\label{fig: thank 3.5.1b}
\end{subfigure}
\caption{We numerically solve system \eqref{eq:1.3} on the 1-D spatial domain, $x \in \mathcal{X} \coloneqq [0,200]$, and impose the initial conditions \eqref{eq:1.3b} with $\bar{M} = 0.25$ (a) and $\bar{M} = 0.75$ (b). Each plot represents $X(t)$ such that $N(X(t),t) = 0.5$ for $t \in (0,70]$ when $\kappa < \kappa^*$, with $\kappa^*$ defined by \eqref{eq:3.27c}, and when ${\kappa \in \{\kappa^*,10 \kappa^*, 100 \kappa^*, 1000 \kappa^*\}}$. We see that the front travels with a constant propagation speed that increases monotonically with $\kappa$. }
\label{fig:3.5.1}
\end{figure}

Now, suppose Conjecture \ref{conj:3.14} is true. Then, given $\kappa  >0$, for each $\bar{m} \in [0,m^*(\kappa)]$ and $c \geq 2 \sqrt{1-\bar{m}}$, $\alpha_{\bar{m}}(c)$ as defined by \eqref{eq:3.22} exists and is the unique $\alpha'$ mentioned in the statement of Conjecture~\ref{conj:3.14}. Using Remark \ref{rem:3.4} and Conjecture \ref{conj:3.14}, the subsequent result follows naturally (we omit the proof for brevity).
\begin{lem}
Suppose Conjecture \ref{conj:3.14} is true and fix $\kappa > 0$. If $c \geq 2 \sqrt{1-m^*(\kappa)}$, then, for all ${\bar{m}\in \left(m^*(\kappa), 1\right)}$, there exists a unique $\alpha \in (\alpha_{m^*(\kappa)}(c),\alpha_1(c))$ such that the solution of \eqref{eq:2.8a}-\eqref{eq:2.8c} that satisfies the asymptotic properties in Lemma \ref{lem:3.1} converges to $(0,0,\bar{m})$ as $y\to +\infty$.
\label{lem:3.14}
\end{lem}
This lemma allows us to obtain a sharper upper bound on the minimal wave speed for solutions of \eqref{eq:2.8a}-\eqref{eq:2.8c} subject to \eqref{eq:3.3} that converge to $(0,0,\bar{m})$ as $y \to +\infty$, where ${\bar{m} \in (m^*(\kappa),1)}$. We now summarise what we can conclude about the minimal wave speed $c^*_\kappa(\bar{m})$.

\begin{lem}
Suppose Conjecture \ref{conj:3.14} is true. Given $\kappa >0$, the minimal wave speed $c^*_\kappa(\bar{m})$ is a monotonically decreasing function on $[0,1)$, such that
\begin{equation}
\displaystyle c^*_\kappa(\bar{m})  \begin{cases}
    = 2 \sqrt{1-\bar{m}} \qquad \qquad \qquad \qquad \,\, \, \text{if } \bar{m} \in \left[0, m^*(\kappa)\right], \\
    \in \left[2 \sqrt{1-\bar{m}}, 2\sqrt{1-m^*(\kappa)}\right] \quad \text{if } \bar{m} \in \left(m^*(\kappa), 1\right). 
  \end{cases}
\label{eq:3.30}
\end{equation}
\label{lem:3.15}
\end{lem}

\begin{proof}
Fix $\kappa >0$. Suppose, for a contradiction, that $c^*_\kappa$ is not a monotonically decreasing function of $\bar{m}$ on $[0,1)$. Then, we can find $0 \leq \bar{m}' < \bar{m}'' < 1$ such that $c^*_\kappa(\bar{m}') < c^*_\kappa(\bar{m}'')$. Now, choose $c \in (c^*_\kappa(\bar{m}'),c^*_\kappa(\bar{m}''))$. Then, there exists a solution of \eqref{eq:2.8a}-\eqref{eq:2.8c} that satisfies the asymptotic conditions \eqref{eq:3.3}, stays in region $\mathcal{D}_{\bar{m}'}$ and converges to $(0,0,\bar{m}')$ as $y\to +\infty$, but there does not exist a solution of \eqref{eq:2.8a}-\eqref{eq:2.8c} that satisfies the asymptotic conditions \eqref{eq:3.3}, stays in region $\mathcal{D}_{\bar{m}''}$ and converges to $(0,0,\bar{m}'')$ as $y\to +\infty$. As $\bar{m}'' \in (\bar{m}',1)$, Remark \ref{rem:3.4} gives us a contradiction, hence $c^*_\kappa$ is a decreasing function of $\bar{m}$ on $[0,1)$.

From Lemma \ref{lem:3.11} and Conjecture \ref{conj:3.14}, we know that the minimal wave speed for all ${\bar{m} \in \left[0, m^*(\kappa)\right]}$ is $c^*_\kappa(\bar{m})  = 2 \sqrt{1-\bar{m}}$. Since $c^*_\kappa$ is a decreasing function of $\bar{m}$ on $[0,1)$, we must have $c^*_\kappa(\bar{m}) \leq 2\sqrt{1-m^*(\kappa)}$ for any $\bar{m} \in \left( m^*(\kappa),1 \right)$. Finally, by Lemma \ref{lem:3.11}, we know that ${c^*_\kappa(\bar{m}) \geq 2\sqrt{1-\bar{m}}}$ for any $\bar{m} \in \left( m^*(\kappa),1 \right)$. This completes the proof of Lemma \ref{lem:3.15}.
\end{proof}

While we do not have a complete characterisation of the minimal wave speed for all $\kappa > 0$ and $\bar{m} \in (0,1)$, we can now state our second main result. Its proof is similar to that of Theorem \ref{thm:3.10} and is contained in Supplementary Material S1. 

\begin{thm}
Suppose Conjecture \ref{conj:3.14} is true. Given $\kappa > 0$, for any $\bar{\mathcal{M}}\in [0,1)$, there exists a minimal wave speed $c^*_\kappa(\bar{\mathcal{M}})$ defined by \eqref{eq:3.30} such that: 
\begin{enumerate}
    \item For $0 < c < c^*_\kappa(\bar{\mathcal{M}})$, system $\eqref{eq:1.3}$ has no weak TWS $(\mathcal{N},\mathcal{M};c)$ connecting $(1,0)$ and $(0,\bar{\mathcal{M}})$ .
    \item For $ c \geq  c^*_\kappa(\bar{\mathcal{M}})$, system $\eqref{eq:1.3}$ has a weak TWS $(\mathcal{N},\mathcal{M};c)$ connecting $(1,0)$ and $(0,\bar{\mathcal{M}})$. Moreover, this solution is unique (up to translation) and $\mathcal{N},\mathcal{M}$ are monotonically strictly decreasing and increasing functions of $\xi = x - ct$, respectively.
\end{enumerate}
\label{thm:3.16}
\end{thm}

\section{Numerical solutions of the PDE model} \label{sec:4}
In this section, we present some numerical solutions of the PDE model \eqref{eq:1.3}, which complement our travelling-wave analysis. We solve \eqref{eq:1.3} on the 1-D spatial domain $\mathcal{X} \coloneqq [0,L]$, where $L > 0$. Similarly to \cite{46}, we assume that the tumour has already spread to a position $x=\sigma < L$ in the tissue and we impose initial conditions that satisfy, for $\bar{M} \in [0,1]$:
\begin{equation}
    \begin{cases}
      N(x,0) = 1, \; M(x,0) = 0, \qquad \qquad \qquad \qquad \qquad \quad \;\;\;\, \qquad \qquad \qquad \text{if} \; 0 \leq x < \sigma-\omega, \\
      N(x,0) =  \exp{\left( 1 - \frac{1}{1 - \left(\frac{x-\sigma + \omega}{\omega}\right)^2}\right)},  \;  M(x,0) = \bar{M}\left( 1- N(x,0) \right), \quad \text{if} \; \sigma-\omega \leq x < \sigma, \\
      N(x,0) = 0, \;  M(x,0) = \bar{M}, \qquad \qquad \qquad \qquad \qquad \qquad \qquad \qquad \quad \; \text{if} \; \sigma \leq x \leq L.
    \end{cases}
    \label{eq:1.3b}
\end{equation}
Here, $0 < \omega < \sigma$ represents how sharp the initial boundary between the tumour and healthy tissue is. We complete the mathematical problem by imposing zero-flux boundary conditions for $N$ at $x=0$ and $x=L$. We set $L = 200$, $\sigma =0.2$ and $\omega =0.1$ for our simulations. 

\subsection{Elucidating the wave speed that emerges in the PDE model} \label{sec:4.a}
A characteristic feature of the well-studied Fisher-KPP model is that any non-negative initial condition with compact support will evolve towards a travelling front with speed equal to the minimal wave speed, $c = 2$ \cite{21,30,31}. One may, therefore, question whether this result extends to more complex R-D systems that exhibit travelling waves. For our model, the results from Section~\ref{sec:3}\ref{sec:3.5} suggest that this does hold for solutions of \eqref{eq:1.3} subject to the initial conditions \eqref{eq:1.3b} for $\bar{M} \in [0,1)$. In contrast, the results from Section \ref{sec:3}\ref{sec:3.4} show that there is no strictly positive minimal wave speed for TWS of \eqref{eq:1.3} that satisfy the asymptotic conditions \eqref{eq:2.2a}. Yet, the solution of \eqref{eq:1.3} subject to the initial conditions \eqref{eq:1.3b} for $\bar{M} =1$ appears to evolve towards a travelling front with a strictly positive speed, as illustrated in Figure \ref{fig:4.1.1a} for different values of $\kappa$ (see Supplementary Material S3 for a travelling wave profile). In this way, the solutions of the PDE system preferentially select a wave speed in a way that the corresponding ODE system does not.

\begin{figure}[!ht]
\begin{subfigure}[t]{0.5\textwidth}
\centering
\includegraphics[scale=0.18]{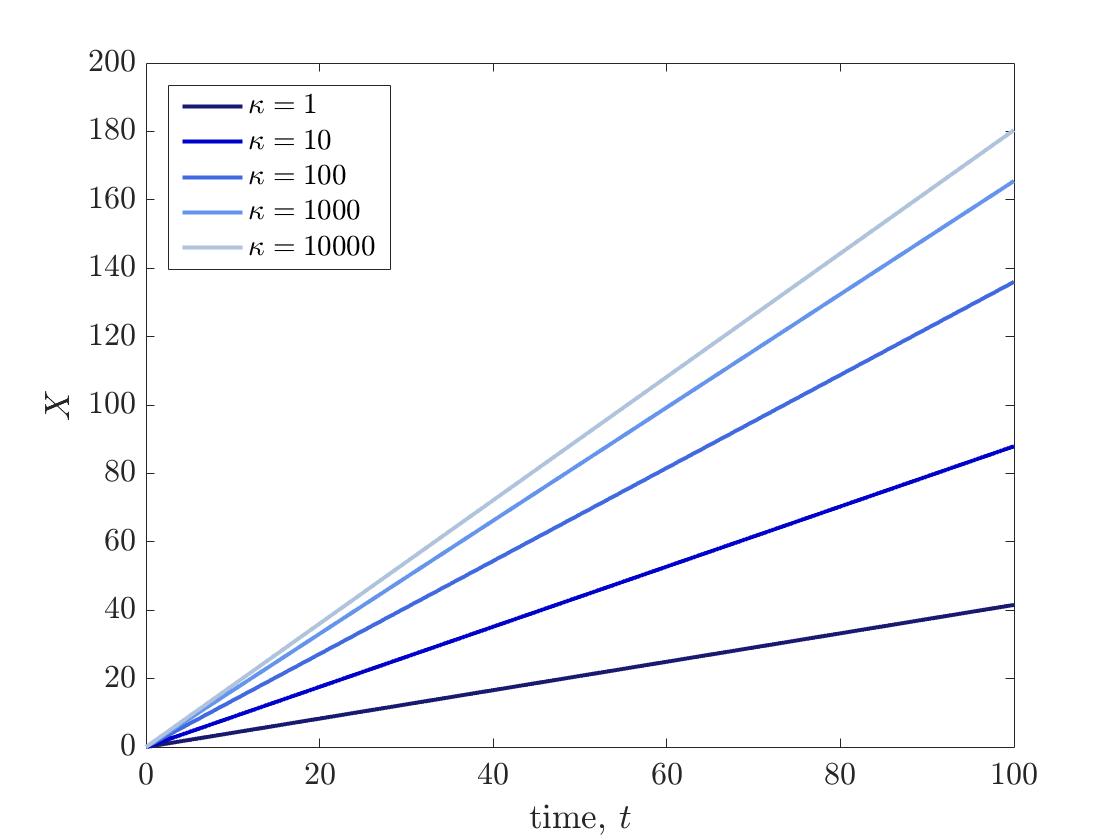} 
\caption{}
\label{fig:4.1.1a}
\end{subfigure}
\hfill
\begin{subfigure}[t]{0.5\textwidth}
\centering
\includegraphics[scale=0.17]{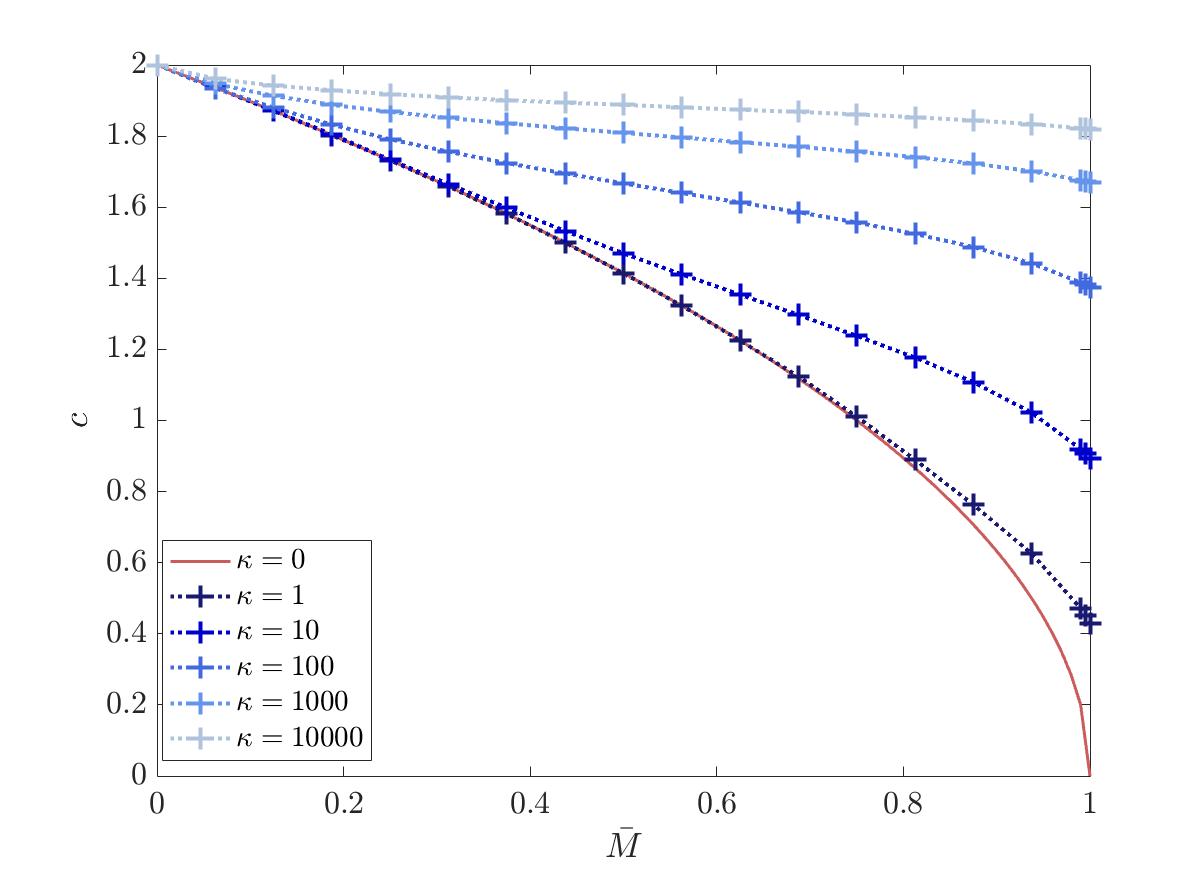} 
\caption{}
\label{fig:4.1.1b}
\end{subfigure}
\caption{We solve system \eqref{eq:1.3} on the 1-D spatial domain, $ x \in \mathcal{X} \coloneqq [0,200]$, and impose the initial condition \eqref{eq:1.3b} for $\bar{M}\in[0,1]$. In (a), we plot $X(t)$ such that $N(X(t),t) = 0.5$ for $t \in (0,100]$ in the cases where $\bar{M}=1$ and $\kappa \in \{1,10, 100, 1000, 10000 \}$, and we see that the front travels with a strictly positive, constant propagation speed that increases monotonically with $\kappa$. In (b), we plot the speed of the travelling front that emerges for $\bar{M} \in \{0.0625j | j \in \llbracket 1, 15 \rrbracket\} \cup \{0.99, 1\}$ and observe that this speed is monotonically decreasing with $\bar{M}$, given $\kappa >0$. }
\label{fig:4.1.1}
\end{figure}

Given different values of $\kappa > 0$, we calculated the speed of travelling fronts that emerge for solutions of \eqref{eq:1.3} subject to the initial conditions \eqref{eq:1.3b} with $\bar{M}\in[0,1]$. Our numerical simulations suggest that the wave speed selected by the PDE model is a continuous, decreasing function of $\bar{M}\in [0,1]$, as illustrated in Figure \ref{fig:4.1.1b}, which represents this wave speed as a function of $\bar{M}\in [0,1]$ for $\kappa \in \{0, 1, 10, 100, 1000, 10000\}$. This is consistent with Lemma $\ref{lem:3.15}$ and our observation that the speed of travelling fronts that emerge for solutions of \eqref{eq:1.3}, subject to the initial conditions \eqref{eq:1.3b} with $\bar{M}\in[0,1)$, appears to be equal to the minimal wave speed, $c^*_\kappa(\bar{M})$, defined by \eqref{eq:3.30}. This result is interesting because the speed selected by the PDE model appears to be left-continuous at $\bar{M}=1$, despite the fact that the minimal wave speed for the existence of TWS is not.

\subsection{Comparing trajectories of the PDE and ODE models}
From Theorems \ref{thm:3.10} and \ref{thm:3.16}, we know that system \eqref{eq:1.3} has TWS connecting $(1,0)$ and $(0,\bar{M})$, $\bar{M} \in [0,1]$, for all $c>0$ if $\bar{M}=1$ and for all $c \geq c_\kappa^*(\bar{M})$, defined by \eqref{eq:3.30}, otherwise. Furthermore, we saw that solutions of \eqref{eq:1.3} subject to the initial conditions \eqref{eq:1.3b} for $\bar{M}\in[0,1]$ evolve towards travelling waves and, in particular, the wave speed is approximately equal to $c_\kappa^*(\bar{M})$ for ${\bar{M}\in [0,1)}$.  We should therefore be able to find agreement between the wave profile of the solutions of the PDE system \eqref{eq:1.3}, subject to the initial conditions \eqref{eq:1.3b} for $\bar{M}\in[0,1]$, and of the ODE system \eqref{eq:2.1a}-\eqref{eq:2.1b}, subject to the asymptotic conditions \eqref{eq:2.2a}, if $\bar{M}=1$, and \eqref{eq:2.2b} otherwise, where we set $c$ to be the wave speed selected by the numerical solution of the PDE system to numerically solve the ODE system. We find good agreement between the wave profiles of the PDE and ODE solutions, and a couple of illustrative examples are shown in Figure \ref{fig:4.2.3}.

\begin{figure}[!ht]
\begin{subfigure}[t]{0.5\textwidth}
\centering
\includegraphics[scale=0.18]{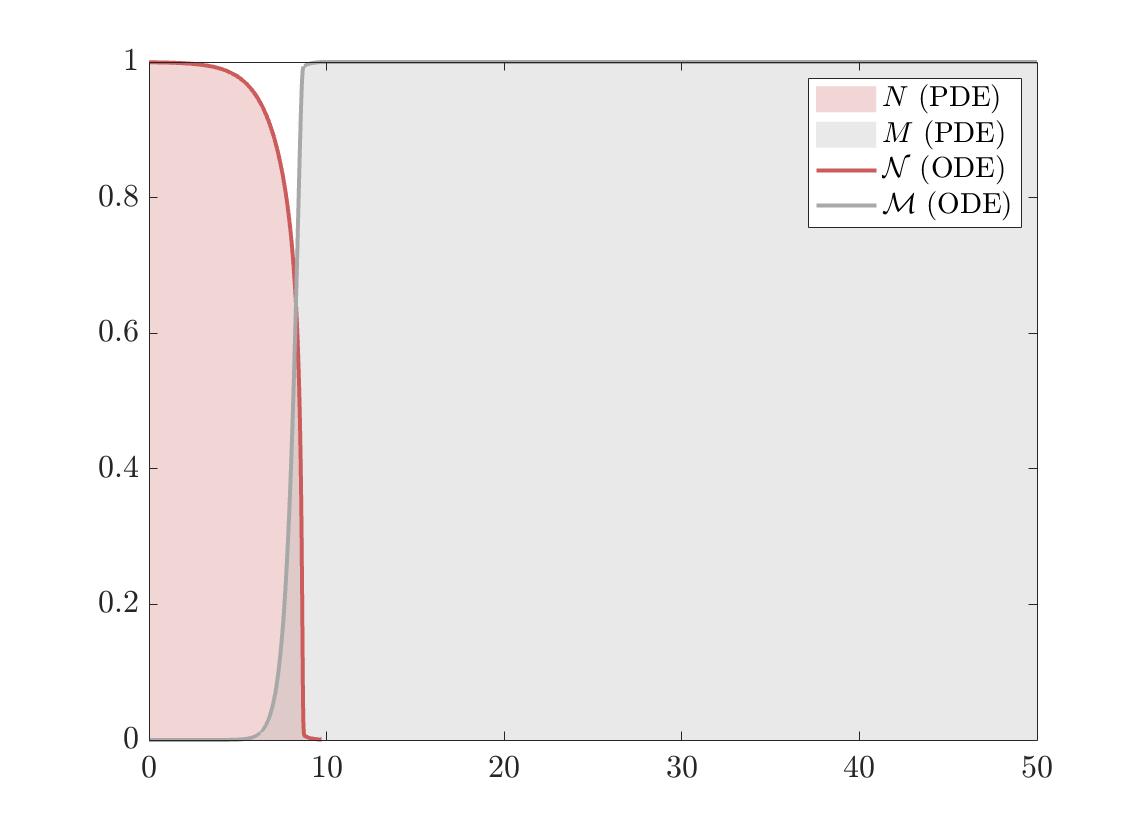} 
\caption{}
\label{fig:4.2.3a}
\end{subfigure}
\hfill
\begin{subfigure}[t]{0.5\textwidth}
\centering
\includegraphics[scale=0.18]{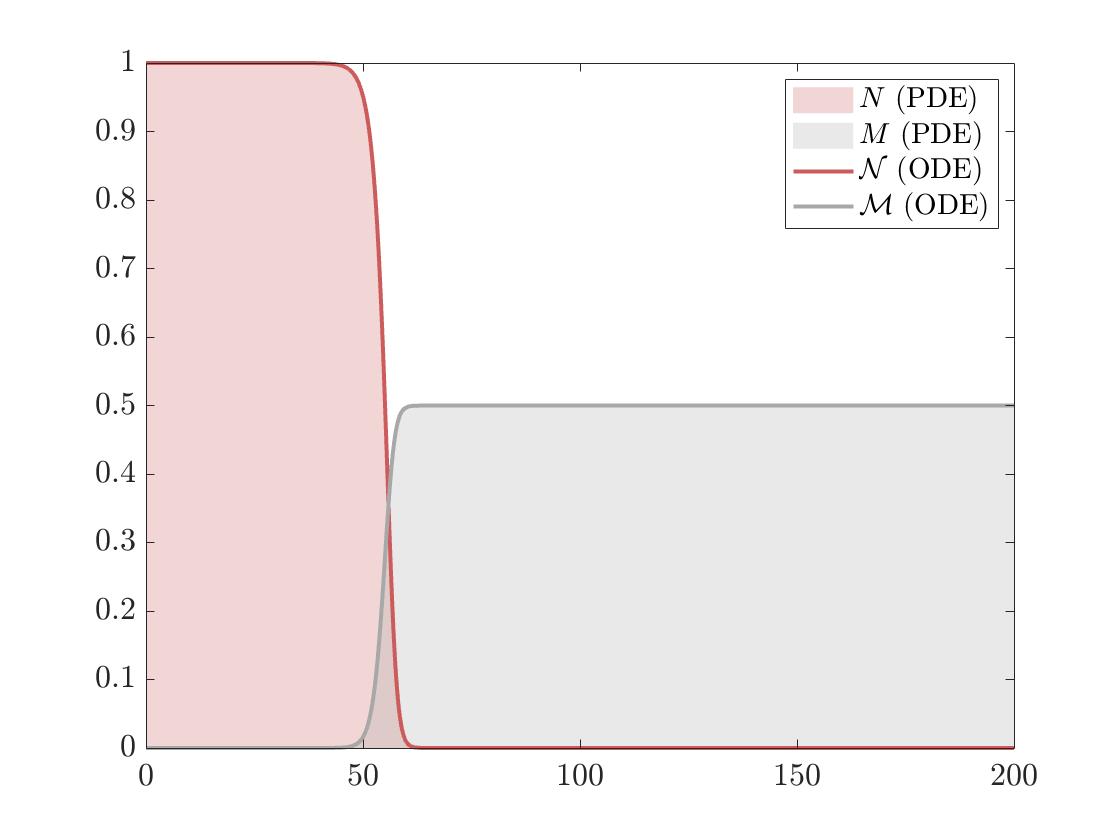} 
\caption{}
\label{fig:4.2.3b}
\end{subfigure}
\caption{We compare solutions of \eqref{eq:1.3}, subject to the initial conditions \eqref{eq:1.3b} with $\bar{M}= 1$ (a) and $\bar{M}=0.5$ (b), to solutions of \eqref{eq:2.1a}-\eqref{eq:2.1b}, subject to the asymptotic conditions \eqref{eq:2.2a} (a) and \eqref{eq:2.2b} with $\bar{\mathcal{M}}=0.5$ (b). We use the wave speed estimated from the numerical solution of the PDE model to solve the ODE model and set $\kappa =1$. Solutions of the PDE and ODE models agree in both cases.}
\label{fig:4.2.3}
\end{figure}

\section{Discussion and perspectives}
Understanding the process of tumour invasion is at the forefront of cancer research. The seminal model of acid-mediated tumour invasion developed by Gatenby and Gawlinski \cite{11} generated new biological insights and formed the basis for subsequent mathematical work on this topic. Due to the model's complexity, most existing results in the literature on the existence of TWS of the model stem from numerical investigations, which are complemented by partial analytical results. In particular, obtaining a complete understanding of the existence of TWS has proven difficult and this has prompted the derivation of simpler models \cite{1,20}. In this paper, we carried out a travelling-wave analysis for the simplified model \eqref{eq:1.3}.

We found that system \eqref{eq:1.3} can support a continuum of TWS, which represent the invasion of healthy tissue, comprised only of ECM, by tumour cells and differ according to the density of ECM far ahead of the wave front. More specifically, we characterised TWS connecting the two spatially homogeneous steady states $(1,0)$ and $(0,\bar{M})$, for $\bar{M} \in [0,1]$. Due to the degeneracy in the first equation of \eqref{eq:1.3} for $M=1$, we distinguished the cases where $\bar{M}=1$ and where $\bar{M }\in[0,1)$. 

In the first case, we proved the existence of TWS for any positive wave speed, $c>0$. This result is particularly interesting as it differs from previous results for degenerate diffusion in a scalar or multi-equation setting, where TWS exist if and only if the wave speed is greater than or equal to a strictly positive minimal wave speed \cite{14,26,27}. It is important to note that this does not imply that a positive wave speed which is preferentially selected does not exist for solutions of $\eqref{eq:1.3}$ that connect $(1,0)$ and $(0,1)$. In fact, we saw in Section \ref{sec:4}\ref{sec:4.a} that a strictly positive, $\kappa$-dependent wave speed appears to be selected by \eqref{eq:1.3} subject to the initial conditions \eqref{eq:1.3b} with $\bar{M}=1$. It would, therefore, be interesting to study the stability of the TWS defined by Theorem \ref{thm:3.10}. We may gain insight on the minimal wave speed for solutions of \eqref{eq:1.3} that connect $(1,0)$ and $(0,1)$ by determining parameter regimes in which solutions are unstable.

In the second case, we proved that TWS exist if and only if the wave speed is greater than or equal to a strictly positive minimal wave speed, $c^*_\kappa(\bar{\mathcal{M}})$, defined by \eqref{eq:3.30} for $\bar{\mathcal{M}}\in[0,1)$. Given $\kappa > 0$, this minimal speed appears to be a monotonically decreasing, continuous function of $\bar{\mathcal{M}}$. In particular, we conjectured that, given $\kappa > 0$ and $m^*(\kappa) \coloneqq \frac{1}{\kappa+1}$, we can precisely define $c^*_\kappa(\bar{\mathcal{M}})= 2\sqrt{1-\bar{\mathcal{M}}}$ for $\bar{\mathcal{M}}\in [0,m^*(\kappa)]$. Similarly to the generalised Fisher-KPP equation, this minimal wave speed is the smallest $c>0$ such that the equilibrium $(0,0,\bar{m})$, with $\bar{m}=\bar{\mathcal{M}}$, of system \eqref{eq:2.8a}-\eqref{eq:2.8c} is a stable node and not a stable spiral. For $\bar{\mathcal{M}}\in (m^*(\kappa),1)$, numerical simulations suggested that the wave speed selected by the PDE is strictly greater than $2\sqrt{1-\bar{\mathcal{M}}}$, which is consistent with \eqref{eq:3.30}. The fact that the equilibrium $(0,0,\bar{m})$ of system \eqref{eq:2.8a}-\eqref{eq:2.8c} is a stable node is then no longer a sufficient condition to ensure the positivity of the $n$-component of the TWS in the desingularised variables and thus of the $\mathcal{N}$-component of the TWS in the original variables. This reflects the differences that can be in observed in systems of equations compared to scalar equations, which can be attributed to the higher dimensionality of the problem. 

Our results regarding the dependence of the minimal wave speed on the model parameters $\kappa$ and $\bar{\mathcal{M}}$ for TWS of \eqref{eq:1.3} connecting $(1,0)$ and $(0,\bar{\mathcal{M}})$, $\bar{\mathcal{M}}\in[0,1)$ rely on a conjecture. Our aim is to rigorously prove this result in future work. In addition, we do not have an expression for the minimal wave speed if $\bar{\mathcal{M}} \in (m^*(\kappa),1)$. Yet, as $\kappa \to +\infty$,  $m^*(\kappa)\to 0$, and it is clear that, as $\kappa$ increases, we can precisely describe the minimal wave speed for a decreasing range of values of $\bar{\mathcal{M}}\in[0,1)$. We would therefore like to provide a complete characterisation of $c^*_\kappa(\bar{\mathcal{M}})$ for all $\kappa > 0$ and $\bar{\mathcal{M}}\in(m^*(\kappa),1)$. Now, we observed in Section \ref{sec:4}\ref{sec:4.a} that the solution of system $\eqref{eq:1.3}$ subject to initial conditions \eqref{eq:1.3b} with $\bar{M}\in[0,1]$ evolves towards a travelling front with a $\kappa$- and $\bar{M}$-dependent wave speed. Importantly, given $\kappa > 0$, it appears that this numerical wave speed is a continuous function of $\bar{M}$ in $[0,1]$, is equal to $c^*_\kappa(\bar{M}) = 2 \sqrt{1-\bar{M}}$ for all $\bar{M}\in[0,m^*(\kappa)]$ and is strictly greater than $2 \sqrt{1-\bar{M}}$ for all $\bar{M}\in (m^*(\kappa),1]$. We note that we have included $\bar{M}=1$ in our preceding observations, which highlights our hypothesis that elucidating the minimal wave speed for $\eqref{eq:1.3}$ in the case $\bar{\mathcal{M}}=1$ could perhaps help us elucidate the minimal wave speed for $\eqref{eq:1.3}$ in the case $\bar{\mathcal{M}}\in(m^*(\kappa),1)$, or vice versa. It is, therefore, important to also study the stability of the travelling waves defined by Theorem \ref{thm:3.16}.

Finally, while it is of mathematical interest to obtain a comprehensive description of the minimal wave speed for all TWS of $\eqref{eq:1.3}$, it is also of biological interest. Our results indicate that the minimal wave speed is highly dependent on the value of $\kappa$, which is the rescaled ECM degradation rate. Since this parameter represents, in a sense, the aggressivity of the tumour cell population towards the ECM, it is significant from an oncological perspective. Hence, our results have the long-term potential of revealing promising targets for therapeutic intervention. 

\pagebreak

\appendix
\numberwithin{figure}{section}
\renewcommand{\thesection}{S}
\renewcommand{\thesubsection}{S\arabic{subsection}}
\renewcommand{\thesubsubsection}{\thesubsection.{\arabic{subsubsection}}}
\section{Supplementary Material}
In the following, we provide supplementary material to the main paper. In Section S1, we provide the proofs of results from Sections 2 and 3 that were not included in the paper. In Section S2, we detail calculations that motivated Conjecture 1. In Section S3, we discuss the numerical methods used to solve the different models and present some additional numerical results.

\subsection{Proofs of Sections 2 and 3}

This section contains full details of proofs that were not included in the main paper due to their similarity to the corresponding proofs contained in \cite{10}. For the purpose of clarity, we recall that our paper focussed on studying weak travelling wave solutions (TWS) for the following PDE model:
\begin{equation}
 \begin{cases}
\displaystyle \frac{\partial N}{\partial t} =\frac{\partial}{\partial x} \left[ \left(1-M\right) \frac{\partial N}{\partial x} \right]+\left(1 - N\right)N,\\
\displaystyle \frac{\partial M}{\partial t} = -\kappa M N.
\end{cases}
\label{eq:S1}
\end{equation}

Introducing the travelling wave coordinate $\xi = x-ct$, where $c > 0$, and the ansatz $N(x,t) =\mathcal{N}(\xi)$ and $M(x,t)=\mathcal{M}(\xi)$, the TWS we seek must satisfy the following ODE system:
   \begin{subnumcases}{}
    {\displaystyle \diff{}{\xi}\left((1-\mathcal{M})\diff{ \mathcal{N}}{ \xi }\right) + c \diff{\mathcal{N}}{\xi} + (1-\mathcal{N})\mathcal{N} = 0;} \label{eq:1a} \\
 {\displaystyle c \diff{\mathcal{M}}{\xi} - \kappa \mathcal{M}\mathcal{N} = 0,} \label{eq:1b}
   \end{subnumcases}
and either of two sets of asymptotic conditions:
\begin{subnumcases}{}
   {\lim_{\xi \to -\infty}(\mathcal{N},\mathcal{M}) = (1,0), \, \lim_{\xi \to +\infty}(\mathcal{N},\mathcal{M}) = (0,1),} \label{eq:2a}\\ 
   {\lim_{ \xi \to -\infty}(\mathcal{N},\mathcal{M}) = (1,0), \lim_{\xi \to +\infty}(\mathcal{N},\mathcal{M}) = (0,\bar{\mathcal{M}}) \, \text{with } \bar{\mathcal{M}}\in[0,1).} 
   \label{eq:2b}
\end{subnumcases}

To simplify the analysis, we removed the singularity in system \eqref{eq:1a}-\eqref{eq:1b} by introducing a new independent variable $y$. Denoting derivatives with respect to $y$ using primes and further introducing a dependent variable $p=n'$, we studied solutions $(n_{\alpha,c},p_{\alpha,c},m_{\alpha,c})$ of the following system:
\begin{subnumcases}{}
{ n' = p},  \label{eq:1.1a}
\\[3pt]
{p' = -c p - (1-n)n(1-m),}\label{eq:1.1b}
\\
{m'= \frac{\kappa}{c} m (1-m) n.} \label{eq:1.1c}
\end{subnumcases}
subject to the following asymptotic conditions as $y \to -\infty$, for $\alpha \geq 0$:
\begin{equation}
\begin{aligned}
n(y) & = 1-e^{\lambda_2 y} + \mathcal{O}(e^{(\lambda_2 + \mu)y}), \\
p(y) & = -\lambda_2e^{\lambda_2 y} + \mathcal{O}(e^{(\lambda_2 + \mu)y}),\\
m(y) & = \alpha e^{\lambda_3 y} + \mathcal{O}(e^{(\lambda_3 + \mu)y}),
\label{eq:S1.2}
\end{aligned}
\end{equation}
where $\lambda_2 = (-c +\sqrt{c^2+4})/2 $, $\lambda_3= \kappa/c$ and $\mu = \min(\lambda_2,\lambda_3) >0$. 

\subsubsection{Proof of Lemma 2.1}
\begin{proof}
Suppose $(\mathcal{N}, \mathcal{M};c)$ is a weak TWS as defined in the paper. That is, 
\begin{enumerate}
    \item  $(\mathcal{N},\mathcal{M}) \in C(\mathbb{R},[0,1]) \times C(\mathbb{R},[0,1])$ and $(1-\mathcal{M}) \diff{\mathcal{N}}{\xi} \in L^2(\mathbb{R})$;
    \item $(\mathcal{N},\mathcal{M})$ is a weak solution of \eqref{eq:1a}-\eqref{eq:1b}, i.e. for all $(\phi,\psi) \in C^1(\mathbb{R}) \times C^1(\mathbb{R})$ with compact support
    \begin{equation}
      \int_\mathbb{R} \left\{\left[c \mathcal{N} + (1-\mathcal{M}) \diff{\mathcal{N}}{\xi} \right]\diff{\phi}{\xi} - (1-\mathcal{N})\mathcal{N}\phi \right\} \mathrm{d}\xi =0,
      \label{eq:2.1.1}
    \end{equation}
    \begin{equation}
     \int_\mathbb{R} \mathcal{M} \left\{c \diff{\psi}{\xi} + \kappa \mathcal{N} \psi \right\} \mathrm{d}\xi =0;
     \label{eq:2.1.2}
    \end{equation}
\item one of the pairs of asymptotic conditions given by \eqref{eq:2a} and \eqref{eq:2b}, respectively, are satisfied. 
\end{enumerate}

We first prove the regularity result for $\mathcal{N},\mathcal{M}$. $\mathcal{M}$ is a weak solution of \eqref{eq:1a}, therefore, by \eqref{eq:2.1.2}, its weak derivative is $\frac{\kappa}{c} \mathcal{MN}$. We have that $\mathcal{M}$ and $\mathcal{N}$ are continuous in $\mathbb{R}$, which implies the weak derivative of $\mathcal{M}$ is continuous in $\mathbb{R}$ since $\kappa, c > 0$. It follows that $\mathcal{M}$ is continuously differentiable in $\mathbb{R}$ and, hence, $\mathcal{M}$ is a classical solution to \eqref{eq:1b}. $\mathcal{N}$ is a weak solution of the ODE \eqref{eq:1a} posed on $I \coloneqq (-\infty,\bar{\xi})$. Using the Elliptic Regularity theorem and Sobolev Embedding theorem, we can show that $\mathcal{N} \in C^2(I)$ since $\mathcal{N}(1-\mathcal{N}) \in L^1(\mathbb{R})$. This means that $\mathcal{N}$ is a classical solution of \eqref{eq:1a} on $I$. Using a bootstrap argument, it is then straightforward to prove that $\mathcal{N}$ and $\mathcal{M}$ are smooth on $I$.

We now prove the strict bounds on $\mathcal{N}$ and $\mathcal{M}$ on $I$.  If there exists $\xi^* \in \mathbb{R}$ such that $\mathcal{M}(\xi^*) =0$, then the ODE \eqref{eq:1b} implies that $\mathcal{M}$ is identically zero for all $\xi \geq \xi^*$. This contradicts the requirement that $\mathcal{M}$ tends to $\bar{\mathcal{M}} > 0$ as $\xi \to +\infty$, so we must have $\mathcal{M}(\xi) > 0$ for all $\xi \in \mathbb{R}$. Since $\mathcal{M}$ converges to zero as $\xi \to -\infty$, we know that $\mathcal{M}(\xi) < \bar{\mathcal{M}}$ for $\xi$ negative with $|\xi|$ large. Therefore, we either have $ \mathcal{M}(\xi) < \bar{\mathcal{M}}$ for all $\xi \in \mathbb{R}$ and we set $\bar{\xi} = +\infty$, or there exists a unique $\bar{\xi}\in \mathbb{R}$ such that $\mathcal{M}(\bar{\xi}) = \bar{\mathcal{M}}$ and $\mathcal{M}(\xi) < \bar{\mathcal{M}}$ for any $\xi < \bar{\xi}$. Uniqueness follows from the fact that $\mathcal{M}$ is increasing for positive $\mathcal{N}$.

It is straightforward to show that $\mathcal{N}$ cannot have a non-degenerate minimum in $I$ that is equal to zero, without being identically zero in $I$, which contradicts the requirement that $\mathcal{N}$ converges to $1$ as $\xi \to -\infty$. Similarly, $\mathcal{N}$ cannot have a non-degenerate maximum in $I$ that is equal to $1$, without being identically equal to $1$ on this interval. If $\bar{\xi} = +\infty$, this contradicts the requirement that $\mathcal{N}$ converges to $0$ as $\xi \to +\infty$. If $\bar{\xi} < +\infty$, this means that $\diff{\mathcal{M}}{\xi}(\bar{\xi}) = \frac{\kappa}{c}\bar{\mathcal{M}} > 0$ and $\mathcal{M}(\xi) \geq \bar{\mathcal{M}}$ for $\xi \geq \bar{\xi}$ close to $\bar{\xi}$. Since $\mathcal{M}$ is increasing for positive $\mathcal{N}$, the former inequality contradicts the requirement that $\mathcal{M}$ tends to $\bar{\mathcal{M}}$ as $\xi \to +\infty$. We must therefore have $\mathcal{N}(\xi) < 1 $ for all $\xi \in I$. 

To prove \textit{(ii)}, we suppose $\bar{\xi} < +\infty$. By assumption, $\mathcal{M}$ tends to $\bar{\mathcal{M}}$ as $\xi \to +\infty$ and, by definition of $\bar{\xi} < +\infty$, $\mathcal{M}(\bar{\xi}) = \bar{\mathcal{M}}$. Now, $\mathcal{M}$ is increasing with $\mathcal{N} > 0$ and $\mathcal{M} > 0$, so the asymptotic condition for $\mathcal{M}$ can only be satisfied if $\mathcal{M}(\xi) = \bar{\mathcal{M}}$ and $\mathcal{N}(\xi) = 0$ for all $\xi \geq \bar{\xi}$.
\end{proof}

\subsubsection{Proof of Lemma 3.2}
\begin{proof}
Suppose that $(n_{\alpha,c},p_{\alpha,c},m_{\alpha,c})$  is defined on $J\coloneqq(-\infty,y_0)$, $y_0 \in \mathbb{R}$, and satisfies $n_{\alpha,c}(y) > 0$ for all $y \in J$. We need to check that we also have $ n_{\alpha,c}(y)<1$,  $p_{\alpha,c}(y) < 0$ and $0< m_{\alpha,c}(y)<1$ for all $y\in J$.

Since $n_{\alpha,c}(y) > 0$ for all $y \in J$, the right-hand side of \eqref{eq:1.1c} is strictly positive unless $m_{\alpha,c}(y) =0$ or $m_{\alpha,c}(y)  =1$, which means $m_{\alpha,c}$ is strictly increasing between $0$ and $1$. By construction, $m_{\alpha,c}(y) = 0$ only at $y = -\infty$, so $m_{\alpha,c}(y) > 0$ for all $y \in J$. Then, we can integrate \eqref{eq:1.1c} from the first $\bar{y} \in \mathbb{R}$ such that $m_{\alpha,c}(\bar{y}) > 0$ to any $y \leq y_0$ to obtain:
\begin{equation*}
m_{\alpha,c}(y) = \frac{1}{\frac{1-m_{\alpha,c}(\bar{y}) }{m_{\alpha,c}(\bar{y})} \exp{\left(-\frac{\kappa}{c} \int^y_{\bar{y}} n_{\alpha,c}(s) \mathrm{d}s\right)}+1},
\end{equation*}
which yields $m_{\alpha,c}(y)<1$ for all $y \leq y_0$. We therefore have $m_{\alpha,c}(y) < 1$ for all $y \in J$.

Suppose that $p_{\alpha,c}(y)$ is not strictly negative for all $y\in J$. Then, there exists $y_1 < y_0 \in J$ such that $p_{\alpha,c}(y_1) = 0$ and $p_{\alpha,c}(y) <0$  for all $y  \in (-\infty,y_1)$. Since $n_{\alpha,c}' = p_{\alpha,c}$, $\displaystyle \lim_{y \to -\infty} n_{\alpha,c}(y) = 1$ and $n_{\alpha,c}(y) > 0$ for all $y \in J$, we have $0 < n_{\alpha,c}(y) < 1$  for all $y  \in (-\infty,y_1)$. Then, $p_{\alpha,c}'(y_1) = -(1-n_{\alpha,c}(y_1))n_{\alpha,c}(y_1)(1-m_{\alpha,c}(y_1)) < 0$, as we have already shown that $ m_{\alpha,c}(y) < 1$ for all $y \in J$. This means that there is a strict, non-degenerate maximum for $n_{\alpha,c}$ at $y_1$, which contradicts the fact that $n_{\alpha,c}(y)$ is strictly decreasing in a left-neighbourhood of $y_1$. We therefore must have $p_{\alpha,c}(y) < 0$ for all $y \in J$, which finally implies that $ n_{\alpha,c}(y) <1 $ for all $y \in J$.
\end{proof}

\subsubsection{Proof of Lemma 3.3}
\begin{proof}
Fix $c>0$ and $\alpha_2 > \alpha_1 > 0$. In the rest of the proof, we denote $n_i = n_{\alpha_i,c}$ and  $m_i = m_{\alpha_i,c}$, for $i = 1,2$.

The first step in this proof is to show that the following inequalities 
\begin{equation}
n_{\alpha_2,c}(y) > n_{\alpha_1,c}(y), \quad m_{\alpha_2,c}(y) > m_{\alpha_1,c}(y),
\label{eq:S1.3}
\end{equation}
hold for $y \in (-\infty,y_1)$ for some $y_1<0$ sufficiently large and negative. We omit this part of the proof as it is identical to that in \cite{10}, except that the growth parameter $r$ in \cite{10} is equal to $1$ in our case.

The second part of this proof involves showing that $T(\alpha_1,c) \leq T(\alpha_2,c)$ and that inequalities \eqref{eq:S1.3} hold for all $y \in (-\infty,T(\alpha_1,c))$, where $T(\alpha,c)$ is defined for each $\alpha \geq 0, \, c > 0$ as:
\begin{equation}
T(\alpha,c) := \sup{\lbrace y_0 \in \mathbb{R} \mid n_{\alpha,c}(y) >0 \text{ for all } y < y_0\rbrace} \in \mathbb{R} \cup \{+\infty\}.
\label{eq:S1.4}
\end{equation}

Suppose for a contradiction that there exists $y_2 < \min{\left(T(\alpha_1,c),T(\alpha_2,c)\right)}$ such that inequalities \eqref{eq:S1.3} hold on the interval $[y_1,y_2)$ but at least one of them becomes an equality at $y = y_2$. We have that
\begin{equation*}
n_i'' + cn_i' + g_in_i = 0, \quad \text{where} \quad g_i = (1-n_i)(1-m_i).
\end{equation*}
Since $0 < n_1(y) < n_2(y) < 1$ and $0 < m_1(y) < m_2(y)< 1$ for $y \in [y_1,y_2)$, we know that $0 < g_2(y) < g_1(y)$ for $y \in [y_1,y_2)$. Now, the ratio $\rho(y) = n_2(y)/n_1(y)$ satisfies the following second order ODE:
\begin{equation}
\rho''(y) + \left( c + \frac{2n_1'(y)}{n_1(y)}\right)\rho'(y) +(g_2(y)-g_1(y))\rho(y) = 0, \quad y \in (y_1,y_2).
\label{eq:1.5}
\end{equation}
Using results from the first part of the proof (see \cite{10}), we know that $\rho(y_1) > 1$ and $\rho'(y_1) > 0$. From \eqref{eq:1.5}, we note that $\rho$ cannot have a positive maximum in $(y_1,y_2)$ so $\rho$ must be increasing for $y\in (y_1,y_2)$, i.e. $\rho(y) \geq \rho(y_1) > 1$ for $y \in [y_1,y_2)$. Taking the limit as $y \to y_2$, we find that $\rho(y_2) > 1$ and therefore $n_2(y) > n_1(y)$ for $y \in [y_1,y_2]$. 

We still need to show that we cannot have $m_2(y_2) = m_1(y_2)$. We know that, for $i =1,2$
\begin{equation}
m_i'(y) = \frac{\kappa}{c}m_i(y)(1-m_i(y))n_i(y), 
\label{eq:1.6}
\end{equation}
for $y \in [y_1,y_2]$. By assumption, $0 < n_1(y) < n_2(y) < 1$ and $0 < m_1(y) < m_2(y)< 1$ for $y \in [y_1,y_2)$, so we can integrate \eqref{eq:1.6} from $y_1$ to $y \leq y_2$ to obtain the following:
\begin{equation}
\frac{m_i(y)}{1-m_i(y)} = \frac{m_i(y_1)}{1-m_i(y_1)}  \exp{\left(\frac{\kappa}{c} \int_{y_1}^y n(s)\mathrm{d}s\right)},\; i = 1,2.
\label{eq:1.7}
\end{equation}
In particular, 
\begin{equation}
\frac{m_2(y_2)}{1-m_2(y_2)} \frac{1-m_1(y_2)}{m_1(y_2)} = \frac{m_2(y_1)}{1-m_2(y_1)}  \frac{1-m_1(y_1)}{m_1(y_1)}  \exp{\left(\frac{\kappa}{c} \int_{y_1}^{y_2} (n_2(s)-n_1(s)) \mathrm{d}s\right)}.
\label{eq:1.8}
\end{equation}
By assumption, $0 < n_1(y) < n_2(y) < 1$ for $y \in [y_1,y_2)$, so the exponential term in \eqref{eq:1.8} is larger than one. Since we also have $0 < m_1(y) < m_2(y)< 1$ for $y \in [y_1,y_2)$, we obtain the following inequalities:
\begin{equation}
\frac{m_2(y_2)}{1-m_2(y_2)} \frac{1-m_1(y_2)}{m_1(y_2)} > \frac{m_2(y_1)}{1-m_2(y_1)}  \frac{1-m_1(y_1)}{m_1(y_1)} > 1.
\label{eq:1.9a}
\end{equation}
This implies that $m_2(y_2) > m_1(y_2)$. 

We have therefore shown that both inequalities \eqref{eq:S1.3} hold at $y=y_2$ and we have reached the desired contradiction, which implies that inequalities \eqref{eq:S1.3} hold for all $y \in (-\infty, \min{\left(T(\alpha_1,c),T(\alpha_2,c)\right)})$. We must then also have $T(\alpha_1,c) \leq T(\alpha_2,c)$. Otherwise, $0 = n_{2}(T(\alpha_2,c)) <  n_{1}(T(\alpha_2,c))$ and, by continuity, we must have $n_{2,c}(y) =  n_{1,c}(y)$ for some $y < T(\alpha_2,c)$, which we just showed was impossible. 
\end{proof}

\subsubsection{Proof of Lemma 3.4}
\begin{proof}
We recall that $\alpha_0(c)$ is defined as follows:
\begin{equation}
\alpha_0(c) \coloneqq \inf{\lbrace \alpha > 0 \mid T(\alpha,c) = +\infty \rbrace} \in [0,+\infty].
\label{eq:1.9}
\end{equation}

Let $c \geq 2$. To show that $\alpha_0(c)= 0$, we must prove that $T(\alpha,c) =+\infty$ for all $\alpha > 0$. On any interval of the form $(-\infty,y_0]$, $y_0\in\mathbb{R}$, solutions $(n_{\alpha,c},p_{\alpha,c},m_{\alpha,c})$ of \eqref{eq:1.1a}-\eqref{eq:1.1c} that satisfy the asymptotic conditions \eqref{eq:S1.2} depend continuously on $\alpha$, uniformly in $y$. It is then straightforward to show that, as $\alpha \to 0$, the solution $(n_{\alpha,c},p_{\alpha,c},m_{\alpha,c})$ of \eqref{eq:1.1a}-\eqref{eq:1.1c} converges uniformly on compact intervals to $(n,n',0)$, where $n$ is the solution of the Fisher-KPP equation:
\begin{equation}
n'' + cn'+(1-n)n =0,
\label{eq:1.10}
\end{equation}
subject to the asymptotic conditions:
\begin{equation}
    \lim_{y\to- \infty} n(y) = 1, \quad \lim_{y\to+\infty} n(y) = 0, \quad \lim_{y\to \pm \infty} n'(y) = 0.
\label{eq:1.11}
\end{equation}
It is known that this solution satisfies $n(y) > 0$ for all $y\in\mathbb{R}$ for $c \geq 2$. By Lemma 3.3, we also know that $n_{\alpha,c(y)}$ is an increasing function of $\alpha$, so, for any $\alpha > 0$, we must have $n_{\alpha,c}(y) > n(y) > 0$ for all $y \in \mathbb{R}$. Therefore,  $T(\alpha,c) =+\infty$ for all $\alpha > 0$, as required. 

Let $0 < c < 2$.  It is known that the solution to the Fisher-KPP equation \eqref{eq:1.10} subject to the asymptotic conditions \eqref{eq:1.11} goes negative for such $c$, which means that there exists $\bar{y} \in \mathbb{R}$ such that $n(\bar{y}) < 0$. Since the solution $(n_{\alpha,c},p_{\alpha,c},m_{\alpha,c})$ of \eqref{eq:1.1a}-\eqref{eq:1.1c} converges uniformly on compact intervals to $(n,n',0)$, where $n$ is the solution of \eqref{eq:1.10} subject to \eqref{eq:1.11}, we can choose $\alpha > 0$ sufficiently small such that $n_{\alpha,c}(\bar{y})< 0$. In this case, $T(\alpha,c) < +\infty$ and this implies that $\alpha_0(c) > 0$.

To show that $ \alpha_0(c) < \infty$, we need to prove that there exists $\alpha> 0$ sufficiently large such that $T(\alpha,c) = +\infty$. We begin by introducing $y_0 = (\ln{\alpha})/\lambda_3$ and consider the following shifted asymptotic expansions of $(n_{\alpha,c},p_{\alpha,c},m_{\alpha,c})$ at $y\to -\infty$:
\begin{equation}
\begin{aligned}
\bar{n}(y) & = n_{\alpha,c}(y-y_0) =  1-\beta e^{\lambda_2 y} + \mathcal{O}(e^{(\lambda_2 + \mu)y}), \\
\bar{p}(y) & =p_{\alpha,c}(y-y_0) = -\lambda_2 \beta e^{\lambda_2 y} + \mathcal{O}(e^{(\lambda_2 + \mu)y}),\\
\bar{m}(y) & =  m_{\alpha,c}(y-y_0) = e^{\lambda_3 y} + \mathcal{O}(e^{(\lambda_3 + \mu)y}),
\label{eq:1.12}
\end{aligned}
\end{equation}
where $\beta = \alpha^{-\lambda_2/\lambda_3}$. On any interval of the form $(-\infty,y_0]$, these functions converge uniformly to $(1,0,\chi)$ as $\beta \to 0$, where $\chi$ is the unique solution of the following differential equation:
\begin{equation}
\chi' = \frac{\kappa}{c}\chi(1-\chi),
\label{eq:1.13}
\end{equation}
normalised so that $\chi(y) \to  e^{\lambda_3 y} + \mathcal{O}(e^{2\lambda_3 y})$ as $y\to -\infty$. Clearly, $\chi$ is increasing and converges to $1$ as $y \to +\infty$. 

By the uniform convergence of $(\bar{n},\bar{p},\bar{m})$ to $(1,0,\chi)$ as $\beta \to 0$ on  $(-\infty,y_0]$, for any $\epsilon > 0$, there exists $\beta^* > 0$ sufficiently small such that
\begin{equation}
\bar{n}(y) \geq 1-\epsilon, \quad \bar{p}(y) \geq -\epsilon, \quad \bar{m}(y) \geq \chi - \epsilon/2, \quad y \in (-\infty,y_0],
\label{eq:1.14}
\end{equation}
and these inequalities hold for all $0 < \beta \leq \beta^*$.
In addition, as $\chi$ converges to $1$ as $y \to +\infty$, we can find $y^* \in \mathbb{R}$ such that $\chi(y) \geq 1-\epsilon/2$ for $y \geq y^*$. 

For any $0 < \epsilon < 1$, it is therefore possible to choose $\beta > 0$ small enough and $y_1 \in \mathbb{R}$ large enough such that
\begin{equation}
\bar{n}(y) \geq 1-\epsilon, \quad \bar{p}(y) \geq -\epsilon, \; \forall y \in (-\infty, y_1], \quad \bar{m}(y_1) \geq 1-\epsilon.
\label{eq:1.15}
\end{equation}

We set $\epsilon = 1/(2K)$, where
\begin{equation}
K = 1 + \frac{1}{c} + \frac{2}{\kappa} > 1.
\label{eq:1.16}
\end{equation}
We can, therefore, find $\beta$ and $y_1$ such that \eqref{eq:1.15} holds for this $\epsilon$. Then, the idea for the rest of this proof is to show that we  have $\bar{n}(y) \geq 1-\epsilon K = 1/2$ for all $y \in [y_1,+\infty)$, which would imply that there exists $\alpha < +\infty$ such that $T(\alpha,c) = +\infty$, as required. 

Suppose for a contradiction that we can find $y > y_1$ such that $\bar{n}(y) < 1/2$. Then, there must exist $y_2 > y_1$ such that $\bar{n}(y_2) = 1/2$ and $\bar{n}(y) \geq 1/2$ for all $y \in (-\infty, y_2]$. Since $\bar{n}(y) \geq 1/2$ for $y \in [y_1,y_2]$, $\bar{m}$ satisfies the following for $y \in [y_1,y_2]$:
\begin{equation}
\bar{m}' = \frac{\kappa}{c} (1-\bar{m})\bar{m}\bar{n} \geq  \frac{\kappa}{2c} (1-\bar{m})\bar{m}.
\label{eq:1.17}
\end{equation}
We recall that, by assumption, $\bar{m}(y_1) \geq 1-\epsilon$ and we can therefore integrate \eqref{eq:1.17} from $y_1$ to $y \in (y_1,y_2]$ to obtain:
\begin{equation}
\frac{\bar{m}(y)}{1-\bar{m}(y)} \geq \frac{\bar{m}(y_1)}{1-\bar{m}(y_1)} e^{\left(\frac{\kappa}{2c}(y-y_1)\right)} \geq \frac{1-\epsilon}{\epsilon} e^{\left(\frac{\kappa}{2c}(y-y_1)\right)}.
\label{eq:1.18}
\end{equation}
From \eqref{eq:1.18}, we can show that $1-\bar{m}(y) \leq 2 \epsilon e^{-\left(\frac{\kappa}{2c}(y-y_1)\right)}$ for $y\in[y_1,y_2]$. Moreover, since $\bar{n}(y) \geq 1/2$ for $y \in [y_1,y_2]$, we can see that  $\bar{n}(y)$ satisfies the following inequality for $y\in[y_1,y_2]$:
\begin{equation}
\bar{n}''(y) + c \bar{n}'(y) + \epsilon e^{-\left(\frac{\kappa}{2c}(y-y_1)\right)} \geq \bar{n}''(y) + c \bar{n}'(y)+(1-\bar{n}(y))\bar{n}(y)(1-\bar{m}(y)) = 0.
\label{eq:1.19}
\end{equation}
We integrate this inequality with respect to $y$ from $y_1$ to any $y\in (y_1,y_2]$ and, assuming that $c-\kappa/(2c) \neq 0$, we obtain:
\begin{equation}
\bar{n}'(y) \geq \bar{n}'(y_1)e^{-c(y-y_1)}- \frac{\epsilon}{c-\kappa/(2c)}\left(e^{-(\kappa/(2c))(y-y_1)}- e^{-c(y-y_1)}\right).
\label{eq:1.20}
\end{equation}
This inequality can be integrated a second time to finally obtain:
 \begin{equation}
 \begin{split}
 \bar{n}(y)  & \geq  \bar{n}(y_1) + \bar{n}'(y_1) \frac{1-e^{-c(y-y_1)}}{c} - \frac{\epsilon}{c-\kappa/(2c)}\left(\frac{1-e^{-(\kappa/(2c))(y-y_1)}}{\kappa/(2c)}   +\frac{ -1 + e^{-c(y-y_1)}}{c}, \right)\\
 & >  \bar{n}(y_1) + \frac{1}{c} \bar{n}'(y_1) - \frac{\epsilon}{c\kappa/(2c)} \\
 & \geq 1-\epsilon - \frac{\epsilon}{c}- \frac{2\epsilon}{\kappa} = 1- K \epsilon = 1/2,
 \end{split}
\label{eq:1.21}
\end{equation}
where the second (strict) inequality relies on the fact that $\bar{n}(y_1) \geq 1-\epsilon$ and $\bar{n}'(y_1) \geq -\epsilon$, by assumption.
In particular, \eqref{eq:1.21} implies that $\bar{n}(y_2) > 1/2$, which contradicts our assumption that $\bar{n}(y_2) = 1/2$. We therefore have that $\bar{n}(y) \geq 1/2$ holds for all $y \in [y_1,+\infty)$, as required. 

In the case where $c-\kappa/(2c) = 0$, a similar proof follows taking $\epsilon = 1/(2K)$, where $K = 1 + 1/c + 1/c^2 > 1$. We note that equations \eqref{eq:1.17}-\eqref{eq:1.19} still hold and, integrating \eqref{eq:1.19} with respect to $y$ from $y_1$ to any $y\in (y_1,y_2]$, we obtain:
\begin{equation}
    \bar{n}'(y) \geq \bar{n}'(y_1)e^{-c(y-y_1)}- \epsilon (y-y_1) e^{-c(y-y_1)}.
\label{eq:1.22}
\end{equation}
Integrating this inequality once more we find:
 \begin{equation}
 \begin{split}
 \bar{n}(y)  & \geq  \bar{n}(y_1) + \bar{n}'(y_1) \frac{1-e^{-c(y-y_1)}}{c} + \epsilon \left( \frac{(y-y_1)e^{-c(y-y_1)}}{c} + \frac{e^{-c(y-y_1)} - 1}{c^2} \right)\\
 & >  \bar{n}(y_1) + \frac{1}{c} \bar{n}'(y_1) - \frac{\epsilon}{c^2} \\
 & \geq 1-\epsilon - \frac{\epsilon}{c}- \frac{\epsilon}{c^2} = 1- K \epsilon = 1/2,
 \end{split}
\label{eq:1.23}
\end{equation}
where the second (strict) inequality relies on the fact that $\bar{n}(y_1) \geq 1-\epsilon$ and $\bar{n}'(y_1) \geq -\epsilon$, by assumption. This proves the desired contradiction for $c-\kappa/(2c) = 0$.

To conclude, we have shown that $T(\alpha,c) =+\infty$ for $\alpha > 0$ large enough (equivalently $\beta > 0$ small enough).
\end{proof}

\subsubsection{Proof of Lemma 3.7}
\begin{proof}
We will begin by proving that a solution $(\tilde{n}, \tilde{p}, \tilde{m})$ of the following system:
\begin{subnumcases}{}
{  \diff{\tilde{n}}{y} = -\frac{1}{c}\tilde{m}\left(\tilde{n}-\frac{\tilde{p}}{c}\right)\left(1-\tilde{n}+ \frac{\tilde{p}}{c}\right)},  \label{eq:1.24a}
\\
{\diff{\tilde{p}}{y} = -c \tilde{p} - \tilde{m}\left(\tilde{n}-\frac{\tilde{p}}{c}\right)\left(1-\tilde{n}+ \frac{\tilde{p}}{c}\right)}\label{eq:1.24b}
\\
{\diff{\tilde{m}}{y} = - \frac{\kappa}{c} \tilde{m} (1-\tilde{m}) \left(\tilde{n}-\frac{\tilde{p}}{c}\right),} \label{eq:1.24c}
\end{subnumcases}
subject to the asymptotic conditions \eqref{eq:S1.2} and whose components converge to zero as $y\to +\infty$ exists on $\mathcal{W}_C^+$. We introduce the ansatz:
\begin{equation}
\tilde{n}(y) = \frac{a}{y} f(\ln{y}), \quad \tilde{m}(y) = \frac{b}{y}g(\ln{y}), \quad z = \ln{y},
\label{eq:1.25}
\end{equation}
where $a \coloneqq c/\kappa$ and $b \coloneqq c$. For $\tilde{n}$ and $\tilde{m}$ evolving according to \eqref{eq:1.24a} and \eqref{eq:1.24c}, with $\tilde{p} = \mathcal{P}(\tilde{n},\tilde{m})$, the functions $f$ and $g$ satisfy:
\begin{subnumcases}{}
{  \diff{f}{z} = f - \frac{1}{c} \frac{b}{a} g \left(a f - e^{z} \frac{\mathcal{P}}{c}\right) \left( 1 - e^{-z} a f +\frac{\mathcal{P}}{c}\right)  },  \label{eq:1.26a}
\\
{\diff{g}{z} = g + \frac{\kappa}{c} g (1-b e^{-z} g) \left(-a f +e^{z}\frac{\mathcal{P}}{c}\right)}. \label{eq:1.26b}
\end{subnumcases}

Now, we recall that
\begin{equation}
\mathcal{P}(\tilde{n},\tilde{m}) = -\frac{1}{c} \tilde{n}\tilde{m} (1 + \mathcal{O}(|\tilde{n}| + |\tilde{m}|)).
\label{eq:1.27a}
\end{equation}
Hence, from \eqref{eq:1.25}, we have
\begin{equation}
\Bigr|\mathcal{P}(\tilde{n},\tilde{m}) \Bigr|= \Bigr| -e^{-2z}\frac{c}{\kappa} f g\Bigr|\leq e^{-2z} \frac{c}{\kappa} | f| |g|.
\label{eq:1.27}
\end{equation}

Since $f$ and $g$ are bounded for all $z \in \mathbb{R}$, $\mathcal{P}$ converges to $0$ as $z \to +\infty$ and we can similarly show that $e^{z} \mathcal{P}$ converges to $0$ as $z \to +\infty$. Therefore, system \eqref{eq:1.26a}-\eqref{eq:1.26b} converges, as $z \to +\infty$, to:
\begin{subnumcases}{}
{  \diff{f}{z} = f -  g  f =f(1-g)  },  \label{eq:1.28a}
\\
{\diff{g}{z} = g - g f = g(1-f)}. \label{eq:1.28b}
\end{subnumcases}

This system has a unique, strictly positive fixed point $(f^*,g^*) =(1,1)$ and it can easily be verified through a linear stability analysis that this fixed point is hyperbolic with eigenvalues $\lambda_1 = 1, \lambda_2 = -1$. Therefore, there exists a one-dimensional stable manifold along which solutions to \eqref{eq:1.26a}-\eqref{eq:1.26b} can converge to $(1,1)$ as $z\to +\infty$ and they satisfy $ |f(z) - 1| +  |g(z) - 1 |  = \mathcal{O}(e^{-z})$ in this limit. We can therefore conclude that there exists a solution of \eqref{eq:1.24a}-\eqref{eq:1.24c}, subject to the asymptotic conditions \eqref{eq:S1.2}, on $\mathcal{W}_C^+$, whose components converge to $(0,0)$ as $y \to +\infty$, such that
\begin{equation}
\tilde{n}(y) = \frac{c}{\kappa y} + \mathcal{O}\left(\frac{1}{y^2}\right) \quad \text{and} \quad  \tilde{m}(y) = \frac{c}{ y}+ \mathcal{O}\left(\frac{1}{y^2}\right), \quad \text{as } y \to +\infty.
\label{eq:1.29}  
\end{equation}

We now need to prove uniqueness of the solution we have just constructed. We will do so by proving that any arbitrary positive solution to \eqref{eq:1.24a}-\eqref{eq:1.24c} on the centre manifold $\mathcal{W}_C^+$ that converges to the origin as $y \to +\infty$ must satisfy the asymptotic condition \eqref{eq:1.29}. This means that there exists $y_1>0$ large enough such that the chosen arbitrary solution coincides with the constructed solution for all $y \in [y_1,+\infty)$. This precisely implies that the arbitrary solution chosen coincides with the constructed solution on the centre manifold, from which uniqueness follows.

For ease of notation, we write the evolution of $\tilde{n}$ and $\tilde{m}$ on $\mathcal{W}_C^+$  as follows:
\begin{subnumcases}{}
{  \diff{\tilde{n}}{y} = -\frac{1}{c}\tilde{m}\left(\tilde{n}-\frac{\mathcal{P}(\tilde{n},\tilde{m})}{c}\right)\left(1-\tilde{n}+ \frac{\mathcal{P}(\tilde{n},\tilde{m})}{c}\right)  = \mathcal{G}(\tilde{n},\tilde{m})},  \label{eq:1.29a}
\\
{\diff{\tilde{m}}{y} = - \frac{\kappa}{c} \tilde{m} (1-\tilde{m}) \left(\tilde{n}-\frac{\mathcal{P}(\tilde{n},\tilde{m})}{c}\right) = \mathcal{H}(\tilde{n},\tilde{m}).} \label{eq:1.29b}
\end{subnumcases}

The solution $(\tilde{n},\tilde{m})$ constructed above satisfies $\tilde{n} = \Phi(\tilde{m})$ in some $\epsilon$-neighbourhood of the origin, where $\epsilon >0$ and $\Phi:(0,\epsilon) \to \mathbb{R}^+$ is a $C^k$ function satisfying the following functional relation for $w \in (0,\epsilon)$
\begin{equation}
\mathcal{G}(\Phi(w),w) = \Phi'(w) \mathcal{H}(\Phi(w),w).
\label{eq:1.30}
\end{equation}
In particular, the solution $(\tilde{n},\tilde{m})$ satisfies \eqref{eq:1.29}, so we have $\Phi(w) = w/\kappa + \mathcal{O}(w^2)$ as $w\to 0$. 

Now, consider any arbitrary positive solution $(\tilde{n},\tilde{m})$ of \eqref{eq:1.29a}-\eqref{eq:1.29b} that converges to the origin as $y \to +\infty$. Using \eqref{eq:1.29a}-\eqref{eq:1.29b}, we find:
\begin{equation}
    \diff{}{y} (\tilde{n}-\Phi(\tilde{m})) = \mathcal{G}(\tilde{n},\tilde{m}) - \Phi'(\tilde{m}) \mathcal{H}(\tilde{n},\tilde{m}).
    \label{eq:1.31a}
\end{equation}

Using \eqref{eq:1.30} and adding and subtracting $\mathcal{G}(\Phi(\tilde{m}),\tilde{m}) = \Phi'(\tilde{m})\mathcal{H}(\Phi(\tilde{m}),\tilde{m})$ to the right-hand side of \eqref{eq:1.31a}, we obtain:
\begin{equation}
 \diff{}{y} (\tilde{n}-\Phi(\tilde{m})) = \mathcal{G}(\tilde{n},\tilde{m}) - \mathcal{G}(\Phi(\tilde{m}),\tilde{m})  - \Phi'(\tilde{m})( \mathcal{H}(\tilde{n},\tilde{m})- \mathcal{H}(\Phi(\tilde{m}),\tilde{m}))\\
\label{eq:1.31b}
\end{equation}
Now, we note that, by the fundamental theorem of calculus for line integrals,
\begin{equation}
\begin{cases}
  \mathcal{G}(\tilde{n},\tilde{m}) - \mathcal{G}(\Phi(\tilde{m}),\tilde{m}) = \int^{\tilde{n}}_{\Phi(\tilde{m})} \partial_1 \mathcal{G}(s,\tilde{m}) \mathrm{d}s,  \\
  \Phi'(\tilde{m})( \mathcal{H}(\tilde{n},\tilde{m})- \mathcal{H}(\Phi(\tilde{m}),\tilde{m})) = \Phi'(\tilde{m}) \int^{\tilde{n}}_{\Phi(\tilde{m})} \partial_1 \mathcal{H}(s,\tilde{m}) \mathrm{d}s,
  \end{cases}
  \label{eq:1.31c}
\end{equation}
where we have denoted the derivative with respect to the first variable, $\tilde{n}$, by $\partial_1$. Finally, using \eqref{eq:1.31c} and reparametrising the curve from $(\Phi(\tilde{m}),\tilde{m})$ to $(\tilde{n},\tilde{m})$, Equation \eqref{eq:1.31b} yields 
\begin{equation}
 \diff{}{y} (\tilde{n}-\Phi(\tilde{m})) = \Delta (\tilde{n},\tilde{m})(\tilde{n} - \Phi(\tilde{m})),
\label{eq:1.31d}
\end{equation}
where
\begin{equation}
\Delta (\tilde{n},\tilde{m}) \coloneqq \int^1_0 \left(\partial_1 \mathcal{G}\left((1-t)\Phi(\tilde{m})+t\tilde{n},\tilde{m}\right) - \Phi'(\tilde{m}) \partial_1 \mathcal{H}\left((1-t)\Phi(\tilde{m})+t\tilde{n},\tilde{m}\right)\right) \mathrm{d}t.
\label{eq:1.32}
\end{equation}

Using \eqref{eq:1.27a} and \eqref{eq:1.29a}-\eqref{eq:1.29b}, it is straightforward to find the following asymptotic expansion of $\Delta (\tilde{n},\tilde{m})$ as $(\tilde{n},\tilde{m}) \to (0,0)$:
\begin{equation}
\Delta (\tilde{n},\tilde{m}) = \frac{\tilde{m}}{c}\left(\frac{1-\kappa}{\kappa}+n\right)(1+\mathcal{O}(|\tilde{n}| + |\tilde{m}|)).
\label{eq:1.33}
\end{equation}
In particular, if $0 < \kappa \leq 1$, then $\Delta (\tilde{n},\tilde{m}) > 0$ for sufficiently small $(\tilde{n},\tilde{m})$ on $\mathcal{W}_C^+$. Now, we can choose $y_1 > 0$ sufficiently large (to ensure that $(\tilde{n},\tilde{m})$ is sufficiently small) and integrate \eqref{eq:1.31d} from $y_1$ to $y_2$ to obtain
\begin{equation}
\tilde{n}(y_2)-\Phi(\tilde{m}(y_2)) = \exp{\left(\int_{y_1}^{y_2} \Delta (\tilde{n}(y),\tilde{m}(y)) \mathrm{d}y\right)} \left(\tilde{n}(y_1)-\Phi(\tilde{m}(y_1))\right),
\label{eq:1.34}
\end{equation}
from which we have that $|\tilde{n}(y_2)-\Phi(\tilde{m}(y_2))| \geq  |\tilde{n}(y_1)-\Phi(\tilde{m}(y_1))|$, whenever $0 < \kappa \leq 1$. By assumption, $(\tilde{n}, \tilde{m})$ converges to $(0,0)$ as $y \to +\infty$, i.e. $\tilde{n}$ and $\Phi(\tilde{m})$ both converge to $0$ as $y \to +\infty$. Therefore, the left-hand side of the above inequality converges to zero as $y_2 \to +\infty$, which implies that,  if $0 < \kappa \leq 1$, then $\tilde{n}(y_1)= \Phi(\tilde{m}(y_1))$ for all sufficiently large $y_1 > 0$. We can then conclude that, if $0 < \kappa \leq 1$, then any positive solution $(\tilde{n},\tilde{m})$ to \eqref{eq:1.29a}-\eqref{eq:1.29b} on the centre manifold, which converges to the origin as $y \to +\infty$,  coincides, up to a translation in $y$, with the solution constructed in the first part of the proof.

 We can similarly show that we have uniqueness for $\kappa > 1$. The solution $(\tilde{n},\tilde{m})$ constructed in the first part of the proof satisfies $\tilde{m} = \Psi(\tilde{n})$ in some $\epsilon$-neighbourhood of the origin, where $\epsilon >0$ and $\Psi:(0,\epsilon) \to \mathbb{R}^+$ is a $C^k$ function satisfying the following functional relation for $w \in (0,\epsilon)$:
\begin{equation}
\mathcal{H}(w,\Psi(w)) = \Psi'(w) \mathcal{G}(w,\Psi(w)).
\label{eq:1.35}
\end{equation}
In particular, the solution $(\tilde{n},\tilde{m})$ satisfies \eqref{eq:1.29}, so we have $\Psi(w) = \kappa w + \mathcal{O}(w^2)$ as $w\to 0$. 

Now, consider again any arbitrary positive solution $(\tilde{n},\tilde{m})$ of \eqref{eq:1.29a}-\eqref{eq:1.29b} that converges to the origin as $y \to +\infty$. Similarly to the derivation of \eqref{eq:1.31d}, we find
\begin{equation}
\diff{}{y} (\tilde{m}-\Psi(\tilde{n})) = \Delta (\tilde{n},\tilde{m})(\tilde{m} - \Psi(\tilde{n})),
\label{eq:1.36}
\end{equation}
where, denoting the derivative with respect to the second variable, $\tilde{m}$, by $\partial_2$, we have
\begin{equation}
\Delta (\tilde{n},\tilde{m}) = \int^1_0 \left(\partial_2 \mathcal{H}\left(\tilde{n}, (1-t)\Psi(\tilde{n})+t\tilde{m}\right) - \Psi'(\tilde{n}) \partial_2 \mathcal{G}\left(\tilde{n}, (1-t)\Phi(\tilde{m})+t\tilde{n},\tilde{m}\right)\right) \mathrm{d}t.
\label{eq:1.37}
\end{equation}

Using \eqref{eq:1.27a} and \eqref{eq:1.29a}-\eqref{eq:1.29b}, it is straightforward to find the following asymptotic expansion of $\Delta (\tilde{n},\tilde{m})$ as $(\tilde{n},\tilde{m}) \to (0,0)$
\begin{equation}
\Delta (\tilde{n},\tilde{m}) = \frac{\kappa}{c} \tilde{n}((\kappa-1)n +m)(1+\mathcal{O}(|\tilde{n}| + |\tilde{m}|)).
\label{eq:1.38}
\end{equation}
In particular, if $\kappa > 1$, then $\Delta (\tilde{n},\tilde{m}) > 0$ for sufficiently small $(\tilde{n},\tilde{m})$ on $\mathcal{W}_C^+$. Now, integrating \eqref{eq:1.36} from $y_1$ to $y_2$ for $y_1 >0$ sufficiently large, we obtain
\begin{equation}
\tilde{m}(y_2)-\Psi(\tilde{n}(y_2)) = \exp{\left(\int_{y_1}^{y_2} \Delta (\tilde{n}(y),\tilde{m}(y)) \mathrm{d}y\right)} \left(\tilde{m}(y_1)-\Psi(\tilde{n}(y_1))\right),
\label{eq:1.39}
\end{equation}
from which we have that $|\tilde{m}(y_2)-\Psi(\tilde{n}(y_2))| \geq \left(\tilde{m}(y_1)-\Psi(\tilde{n}(y_1))\right)|$, whenever $\kappa > 1$. By assumption, $(\tilde{n}, \tilde{m})$ converges to $(0,0)$ as $y \to +\infty$, i.e. $\tilde{m}$ and $\Psi(\tilde{n})$ both converge to $0$ as $y \to +\infty$. Therefore, the left-hand side of the above inequality converges to zero as $y_2 \to +\infty$, which implies that,  if $\kappa > 1$, then $\tilde{m}(y_1)= \Psi(\tilde{n}(y_1))$ for all sufficiently large $y_1 > 0$. Therefore, if $ \kappa > 1$, then any positive solution $(\tilde{n},\tilde{m})$ to \eqref{eq:1.29a}-\eqref{eq:1.29b} on the centre manifold $\mathcal{W_C^+}$, converging to the origin as $y \to +\infty$, coincides, up to a translation in $y$, with the solution constructed in the first part of the proof. 

We have now proved the uniqueness of the solution to \eqref{eq:1.29a}-\eqref{eq:1.29b} that converges to the origin as $y\to +\infty$ and satisfies the asymptotic conditions \eqref{eq:1.29}.
\end{proof}

\subsubsection{Proof of Theorem 3.10}

\begin{proof}
Given any $c > 0$, we choose $\alpha = \alpha_1(c) \in (0,+\infty)$, where $\alpha_1(c)$ is defined for each $c > 0$ as:
\begin{equation}
\alpha_1(c) \coloneqq \inf{\lbrace \alpha > \alpha_0(c) \mid m_\infty(\alpha,c) =  1\rbrace}.
\label{eq:1.40}
\end{equation}
Then, by Lemma 3.19, we know that the solution $(n_{\alpha,c}, p_{\alpha,c},m_{\alpha,c})$ of \eqref{eq:1.1a}-\eqref{eq:1.1c} that satisfies the asymptotic conditions \eqref{eq:S1.2} converges to $(0,0,\bar{m})$ as $y \to +\infty$. We now reverse the change of variables from $\xi$ to $y$ given by
\begin{equation}
\diff{y}{\xi} \equiv \Phi'(\xi) = \frac{1}{1-\mathcal{M}(\xi)} \,\, \forall \xi \in \mathbb{R},
\label{eq:1.41}
\end{equation}
or, equivalently,
\begin{equation}
\diff{\xi}{y} \equiv (\Phi^{-1})'(y) = 1-m(y) \,\, \forall y \in \mathbb{R}.
\label{eq:1.42}
\end{equation}

More specifically, we first define 
\begin{equation}
\xi(y) = \Phi^{-1}(y) = y - \int^y_{-\infty} m(s) \mathrm{d}s \,\, \forall y \in \mathbb{R}.
\label{eq:1.43}
\end{equation}

This allows us to define $\xi(y)$ as $y\to \pm \infty$ using the asymptotic expansions \eqref{eq:S1.2} and \eqref{eq:1.29}, respectively. We have
\begin{equation}
    \xi(y) =
    \begin{cases}
        y - \frac{\alpha}{\lambda_3} e^{\lambda_3 y} + \mathcal{O}\left(e^{(\lambda_3+\mu)y}\right) \quad \text{as } y \to -\infty, \\
        c\ln{(y)} + \xi_0 + \mathcal{O}\left(\frac{1}{y}\right) \quad \quad \,\,\,\,\;\; \text{as } y \to +\infty,
    \end{cases}
\label{eq:1.44}
\end{equation}
where $\xi_0 \in \mathbb{R}$.

From \eqref{eq:1.44}, it is clear that, as $y \to +\infty$, $\xi \to +\infty$. We therefore do not have sharp fronts for system \eqref{eq:1a}-\eqref{eq:1b}. This is consistent with the fact that the solution defined by Lemma 3.7 is a smooth solution in the sense that $m(y)< 1$ for all $y\in \mathbb{R}$ and $m = 1$ only at $+\infty$. By inverting \eqref{eq:1.44}, we obtain
\begin{equation}
    y = \Phi(\xi) = 
    \begin{cases}
        \xi + \frac{\alpha}{\lambda_3} e^{\lambda_3 \xi} + \mathcal{O}\left(e^{(\lambda_3+\mu)\xi}\right) \quad \text{as } \xi \to -\infty, \\
        e^{(\xi - \xi_0)/c} + \mathcal{O}\left(1\right) \quad \quad \quad \quad \quad \text{as } \xi \to +\infty.
    \end{cases}
\label{eq:1.45}
\end{equation}

We can now construct a solution of \eqref{eq:1a}-\eqref{eq:2a} which satisfies the asymptotic conditions \eqref{eq:2a} by defining $\mathcal{N}(\xi)=n_{\alpha,c}(\Phi(\xi))$ and $\mathcal{M}(\xi)=m_{\alpha,c}(\Phi(\xi))$. Since we know that $n_{\alpha,c}'(y) = n_{\alpha,c}'(\Phi(\xi)) < 0$ and $m_{\alpha,c}'(y) = m_{\alpha,c}'(\Phi(\xi)) > 0$ for all $y \in \mathbb{R}$, it is easy to see that $\mathcal{N}'(\xi) < 0$ and $\mathcal{M}'(\xi) > 0$ for all $\xi \in \mathbb{R}$. This shows that $\mathcal{N}$ and $\mathcal{M}$ are monotonically strictly decreasing and increasing functions of $\xi$, respectively. Finally, the uniqueness of the constructed solution is ensured by the uniqueness of the solution $(n_{\alpha,c},p_{\alpha,c},m_{\alpha,c})$ proven in Lemma 3.18.
\end{proof}

\subsubsection{Proof of Theorem 3.16}
\begin{proof}
We first prove item \textit{(ii)}. Fix $\bar{\mathcal{M}} \in[0,1)$ and let $\bar{m} =\bar{\mathcal{M}}$. Given any $c \geq c^*_\kappa(\bar{m})$, we choose $\alpha = \alpha_{\bar{m}}(c) \in [\alpha_0(c),\alpha_1(c))$, where $\alpha_{\bar{m}}(c)$ is defined for each $c \geq c^*_\kappa(\bar{m})$ as
\begin{equation}
\alpha_{\bar{m}}(c) \coloneqq {\lbrace \alpha \geq \alpha_0(c) \mid m_\infty(\alpha,c) =  \bar{m}\rbrace}.
\label{eq:1.46}
\end{equation}
Then, by Lemma 3.5, we know that the solution $(n_{\alpha,c}, p_{\alpha,c},m_{\alpha,c})$ of \eqref{eq:1.1a}-\eqref{eq:1.1c} satisfying the asymptotic conditions \eqref{eq:S1.2} converges to $(0,0,\bar{m})$ as $y \to +\infty$. Similarly to the proof of Theorem 3.10, we now reverse the change of variables from $\xi$ to $y$ given by \eqref{eq:1.41}, or, equivalently, \eqref{eq:1.42}. Since we want to define $\xi(y)$ as $y\to \pm \infty$ and we already have an asymptotic expansion of $m$ as $y \to -\infty$ in \eqref{eq:S1.2}, we only need to compute an asymptotic expansion of $m$ as $y\to + \infty$.

We saw in the proof of Lemma 3.12 that $(n_{\alpha,c}, p_{\alpha,c},m_{\alpha,c})$ approaches $(0,0,\bar{m})$ via a two-dimensional stable manifold. Computing a general asymptotic expansion of solutions of \eqref{eq:1.1a}-\eqref{eq:1.1c} in a neighbourhood of $(0,0,\bar{m})$ that lie on this stable manifold, we find that there must exist $C_1,C_2 \in \mathbb{R}$ such that $n \in (0,1)$, $p < 0$ and $m \in (0,\bar{m})$ (i.e. such that the solution remains in $\mathcal{D}_{\bar{m}}$) and, as $y\to +\infty$, $(n_{\alpha,c}(y), p_{\alpha,c}(y),m_{\alpha,c}(y))$ satisfies:
\begin{equation}
\begin{aligned}
n_{\alpha,c}(y) & = C_1 \left(\frac{c \nu_1}{\kappa (1-\bar{m})\bar{m}}\right) e^{\nu_1 y} + C_2 \left(\frac{c \nu_2}{\kappa (1-\bar{m})\bar{m}}\right) e^{\nu_2 y} + \mathcal{O}(e^{2\nu_1 y}), \\
p_{\alpha,c}(y) & = C_1 \left( \frac{-c^2 \nu_1}{\kappa (1-\bar{m})\bar{m}} -\frac{c}{\kappa \bar{m}}\right) e^{\nu_1 y} + C_2 \left( \frac{-c^2 \nu_2}{\kappa (1-\bar{m})\bar{m}} -\frac{c}{\kappa \bar{m}}\right) e^{\nu_2 y} + \mathcal{O}(e^{2\nu_1y}),\\[4pt]
m_{\alpha,c}(y) & = \bar{m} +C_1 e^{\nu_1 y} + C_2 e^{\nu_2 y} + \mathcal{O}(e^{2\nu_1 y}),
\label{eq:1.47}
\end{aligned}
\end{equation}
where 
\begin{equation*}
\nu_{1} = \frac{- c + \sqrt{c^2 - 4(1-\bar{m})}}{2}, \quad \nu_{2} = \frac{- c - \sqrt{c^2 - 4(1-\bar{m})}}{2}.
\end{equation*}

Using \eqref{eq:S1.2} and \eqref{eq:1.47}, we can now define $\xi(y)$ as $y \to \pm \infty$ in the following way:
\begin{equation}
    \xi = \Phi^{-1}(y) =
    \begin{cases}
        y - \frac{\alpha}{\lambda_3} e^{\lambda_3 y} + \mathcal{O}\left(e^{(\lambda_3+\mu)y}\right), \qquad \qquad \qquad \quad \, \text{as } y \to -\infty, \\[5pt]
        y(1-\bar{m}) - \frac{C_1}{\nu_1} e^{\nu_1 y} - \frac{C_2}{\nu_2} e^{\nu_2 y}+ \mathcal{O}\left(e^{2\nu_1 y}\right), \quad  \text{as } y \to +\infty.
    \end{cases}
\label{eq:1.48}
\end{equation}

From \eqref{eq:1.48}, it is clear that, as $y \to +\infty$, $\xi \to +\infty$. We therefore do not have sharp fronts for system \eqref{eq:1a}-\eqref{eq:1b} subject to the asymptotic conditions \eqref{eq:2b}. By inverting \eqref{eq:1.48}, we obtain
\begin{equation}
    y = \Phi(\xi) =
    \begin{cases}
        \xi + \frac{\alpha}{\lambda_3} e^{\lambda_3 \xi} + \mathcal{O}\left(e^{(\lambda_3+\mu)y}\right), \qquad \qquad \qquad \qquad \qquad \qquad \quad \text{as } \xi \to -\infty, \\[5pt]
        \frac{\xi}{1-\bar{m}} + \frac{C_1}{\nu_1 (1-\bar{m})}e^{\frac{\nu_1\xi}{(1-\bar{m})}}+ \frac{C_2}{\nu_2 (1-\bar{m})}e^{\frac{\nu_2\xi}{(1-\bar{m})}}+ \mathcal{O}\left(e^{(\nu_1+\nu_2)}\right), \quad  \text{as } \xi \to +\infty.
    \end{cases}
\label{eq:1.49}
\end{equation}

We can now construct a solution of \eqref{eq:1a}-\eqref{eq:1b} which satisfies the asymptotic conditions \eqref{eq:2b} for $\bar{\mathcal{M}}=\bar{m}$ by defining $\mathcal{N}(\xi)=n_{\alpha,c}(\Phi(\xi))$ and $\mathcal{M}(\xi)=m_{\alpha,c}(\Phi(\xi))$. Since we know that $n_{\alpha,c}'(y) = n_{\alpha,c}'(\Phi(\xi)) < 0$ and $m_{\alpha,c}'(y) = m_{\alpha,c}'(\Phi(\xi)) > 0$ for all $y \in \mathbb{R}$, it is easy to see that $\diff{\mathcal{N}}{\xi} < 0$ and $\diff{\mathcal{M}}{\xi} > 0$ for all $\xi \in \mathbb{R}$. This shows that $\mathcal{N}$ and $\mathcal{M}$ are monotonically strictly decreasing and increasing in $\xi$, respectively. Finally, the uniqueness of the constructed solution follows from the uniqueness of the solution $(n_{\alpha,c},p_{\alpha,c},m_{\alpha,c})$.

We now prove item \textit{(i)} by contradiction. Fix $\bar{\mathcal{M}} \in [0,1)$. Let $0 < c < c^*_\kappa(\bar{\mathcal{M}})$. Suppose that system $\eqref{eq:S1.3}$ has a unique (up to translation) weak travelling wave solution $(\mathcal{N},\mathcal{M};c)$ connecting $(1,0)$ and $(0,\bar{\mathcal{M}})$, where $\mathcal{N}$ and $\mathcal{M}$ are monotonically strictly decreasing and increasing in $\xi = x - ct$, respectively. Then, we can construct a solution of \eqref{eq:1.1a}-\eqref{eq:1.1c} that satisfies the asymptotic conditions $\displaystyle \lim_{y \to - \infty} (n(y),p(y),m(y)) = (1,0,0)$, $\displaystyle \lim_{y \to +\infty}( n(y),p(y),m(y)) = (0,0,\bar{m})$ for $\bar{m}=\bar{\mathcal{M}}$ by defining $n(y)=\mathcal{N}(\Phi^{-1}(y))$, $p(y) = n'(y)$ and $m(y)=\mathcal{M}(\Phi^{-1}(y))$. Since $\mathcal{N}$ and $\mathcal{M}$ are monotonically strictly decreasing and increasing with respect to $\xi = \Phi^{-1}(y)$, it is easy to see that $n$ and $m$ are monotonically strictly decreasing and increasing with respect to $y$ and that $p(y) < 0 \; \forall y \in \mathbb{R}$. This implies that there exists $\alpha = \alpha(c)$ such that the solution of \eqref{eq:1.1a}-\eqref{eq:1.1c} that satisfies \eqref{eq:S1.2} stays in region $\mathcal{D}_{\bar{m}}$ and converges to $(0,0,\bar{m})$ as $y\to\infty$. Since $0 < c < c^*_\kappa(\bar{\mathcal{M}}) = c^*_\kappa(\bar{m})$, this contradicts the definition of the minimal wave speed.
\end{proof}

\subsection{Supporting results for Conjecture 1}
In this section, we detail calculations that support Conjecture 1. Let $\bar{m}\in(0,1)$. We consider the following boundary value problem:
\begin{equation}
\begin{cases}
  P' = -c - \frac{(1-n)n(1-M(n))}{P}, \\
  M' = \frac{\kappa}{c}\frac{M(1-M)n}{P}, \\
  P(0)= 0, M(0) = \bar{m},
\end{cases}
\label{eq:2.1}
\end{equation}
subject to the additional conditions 
\begin{equation}
    P(n) < 0 \,\, \text{and} \,\,  M(n) \in (0, \bar{m}) \,\, \forall n \in (0,1), \quad P(1) = M(1) = 0.
\label{eq:2.2}
\end{equation}

The result underpinning Conjecture 1 is the following. If the reaction term $g(n) = (1-n)n(1-M(n))$ is of Fisher-KPP type and $g''(n) < 0 \,\, \forall n \in [0,1]$, then there exists a unique solution to \eqref{eq:2.1} that satisfies \eqref{eq:2.2} for any $c \geq 2 \sqrt{g'(0)}$. 

Suppose $(P,M)$ is the unique solution of the Cauchy problem \eqref{eq:2.1}, which exists by Cauchy-Lipschitz theory. We will check that $g$ is of Fisher-KPP type, calculate $g'(0)$ and study the sign of $g''$, making a hypothesis about the necessary conditions for $g''(n) < 0 \; \forall n \in [0,1]$ to hold. To do this, we first need to prove three preliminary results.

\paragraph{Result 1: the value of $P'(0)$.} By letting $n \to 0$ in the differential equation~\eqref{eq:2.1}$_1$ for $P(n)$ and using l'H\^opital's rule to compute $\displaystyle{\lim_{n \to 0} \dfrac{n}{P(n)}}$, we obtain
 \begin{equation}
     P'(0) = -c - (1-\bar{m}) \, \dfrac{1}{P'(0)} \quad \Longrightarrow \quad P_\pm'(0) = \dfrac{-c \pm \sqrt{c^2 - 4 (1-\bar{m})}}{2}.
     \label{eq:S2.3}
 \end{equation}
Since $\bar{m} \in (0,1)$, \eqref{eq:S2.3} implies that, if $c \geq 2 \sqrt{1-\bar{m}}$, then
\begin{equation}
    P_\pm'(0) \in \mathbb{R}, \quad -\infty < P_\pm'(0) < 0.
    \label{eq:S2.4}
\end{equation}

\paragraph{Result 2: Boundedness of $M$.} Solving the differential equation \eqref{eq:2.1}$_2$ for $M$ yields
$$ M(n) = \dfrac{1}{1 + A \exp\left(\dfrac{\kappa}{c} \, \displaystyle{\int_0^n \dfrac{q}{-P(q)} \, {\rm d}q} \right)},$$
where $A$ is a constant of integration. If $c \geq 2 \sqrt{1-\bar{m}}$, then \eqref{eq:S2.4} holds, which implies that
$$
\lim_{q \to 0} \dfrac{q}{-P(q)} = \dfrac{1}{-P'(0)} < +\infty.
$$
and, imposing the condition $M(0)=\bar{m}$, we find $A = \dfrac{1-\bar{m}}{\bar{m}}$. Hence,
 \begin{equation}
     M(n) = \dfrac{\overline{m}}{\overline{m}+(1-\overline{m})\exp\left(\dfrac{\kappa}{c} \, \displaystyle{\int_0^n \dfrac{q}{-P(q)} \, {\rm d}q} \right)}
     \label{eq:S2.5}
 \end{equation}
and, therefore,
\begin{equation}
    0 \leq M(n) \leq 1 \quad \forall \, n \in [0,1].
    \label{eq:S2.6}
\end{equation}

\paragraph{Result 3: Non-positivity of $P$.} Let $c \geq 2 \sqrt{1-\bar{m}}$, then \eqref{eq:S2.4} implies that $P'(0)<0$ and, therefore, $P(n)<0$ in a right-neighbourhood of $n=0$. Suppose, for a contradiction, that there exists $n_1 \in (0,1]$ such that $P(n_1)>0$. Then, we can find $n_0 \in (0,n_1)$ such that $P(n_0)=0$ and $P(n) \geq 0$ for $n \in (n_0,n_1)$. Multiplying both sides of \eqref{eq:2.1}$_1$ by $P$ and integrating between $n_0$ and $n_1$, we obtain
\begin{equation}
    \dfrac{1}{2} \int_{n_0}^{n_1} (P^2(n))' \, {\rm d}n = - c \, \int_{n_0}^{n_1} P(n) \, {\rm d}n - \int_{n_0}^{n_1} n \, (1-n) \, (1-M(n)) \, {\rm d}n.
    \label{eq:2.7}
\end{equation}
Using the fact that $P(n) \geq 0$ for $n \in (n_0,n_1)$ and \eqref{eq:S2.6}, we have
$$
 - c \, \int_{n_0}^{n_1} P(n) \, {\rm d}n \leq 0, \,\, \text{and } \,\, - \int_{n_0}^{n_1} n \, (1-n) \, (1-M(n)) \, {\rm d}n \leq 0.
$$
Therefore, \eqref{eq:2.7} implies that
$$
\dfrac{1}{2} P^2(n_1) = \dfrac{1}{2} \int_{n_0}^{n_1} (P^2(n))' \, {\rm d}n \leq 0, \text{ i.e. } P(n_1) =0.
$$ 
This contradicts the fact that $P(n_1)>0$, and, thus, we must have
\begin{equation}
    P(n) \leq 0 \quad \forall \, n \in [0,1].
    \label{eq:2.8}
\end{equation}

Further, given the expression \eqref{eq:S2.5} for $M$, \eqref{eq:S2.4} and \eqref{eq:2.8} imply that $M(n)$ is strictly decreasing in a neighbourhood of $n=0$ and non-increasing in $n \in (0,1]$ , i.e. 
\begin{equation}
    0 \leq M(n) < \bar{m} < 1 \,\, \forall n \in (0,1].
    \label{eq:S2.9}
\end{equation}

Now, the reaction term $g$ is of Fisher-KPP type if $g \in C([0,1])$, $g(0)=g(1)=0$, and $g(n) > 0 \; \forall n \in (0,1)$. Since $M$ is a component of a classical solution to the Cauchy problem \eqref{eq:2.1}, $M$ is a $C^1([0,1])$ function and it follows that $g \in C([0,1])$. By direct calculation, we have
\begin{equation*}
g(0) = (1-0)0(1-\bar{m}) = 0, \quad g(1) = (1-1)1(1-M(1)) = 0,
\end{equation*}
and \eqref{eq:S2.9} implies that $g(n)= (1-n)n(1-M(n)) > 0 \, \forall n \in (0,1)$. Therefore, we have shown that $g$ is of Fisher-KPP type.

Moreover, we have $g'(n) = (1-2 \, n) (1-M(n)) - n \, (1-n) \, M'(n)$, and, thus
\begin{equation}
\begin{aligned}
 g'(0) & = \lim_{n \to 0} \left[(1-2 \, n) (1-M(n)) - n \, (1-n) \,\frac{\kappa}{c}\frac{M(n)\,(1-M(n))\,n}{P(n)}\right], \\
 & = 1-\bar{m} - \frac{\kappa}{c} \bar{m}(1-\bar{m}) \lim_{n \to 0} \frac{2n}{P'(n)},
\end{aligned}
\label{eq:S2.10}
\end{equation}
by l'H\^{o}pital's rule. If $c \geq 2 \sqrt{1-\bar{m}}$, then $P'(0) \in (-\infty,0)$ is well-defined by \eqref{eq:S2.4} and, using \eqref{eq:S2.10}, we can conclude that
\begin{equation}
    g'(0) = 1-\bar{m}.
    \label{eq:2.11}
\end{equation}

This value of $g'(0)$ yields a minimal wave speed $c = 2\sqrt{g'(0)}=2\sqrt{1-\bar{m}}$, which is consistent with our recurrent assumption that $c \geq 2\sqrt{1-\bar{m}}$.

Finally, we study the sign of $g''$ as a function of $c$ and $\kappa$ and make a hypothesis about the necessary conditions these two parameters must satisfy for the condition $g''(n) < 0 \;\; \forall n \in [0,1]$ to hold. 
\begin{eqnarray*}
g'' &=& -2 (1-M) - 2 \, (1-2 n) \, M' - n \, (1-n) \, M'' 
\\
&=& -2 (1-M) - 2 \, (1-2 n) \, (1-M) \, M \, \dfrac{\kappa}{c} \, \dfrac{n}{P} +
\\
&& - n \, (1-n) \, (1-M) \, M \, \dfrac{\kappa}{c} \left[\dfrac{\kappa}{c} \, \left(\dfrac{n}{P} \right)^2 (1-2 M) + \dfrac{1}{P} + \dfrac{n}{P^2} \left(c + \dfrac{n \, (1-n) \, (1-M)}{P} \right) \right]
\\
&=&  -2 (1-M) +
\\
&& - 2 (1-M) \, M \, \dfrac{\kappa}{c} \, \left(\dfrac{n}{P}\right)^2 \left\{\left[1-2 n + \dfrac{(1-n)}{2}\right]  \dfrac{P}{n} + \, \dfrac{\kappa}{c} \dfrac{n \, (1-n)}{2} \, (1-2 M) - \dfrac{(1-n)}{2} \, P' \right\}
\end{eqnarray*}
that is,
$$
g''(n) =  -2 (1-M(n)) \left(1 - H(n) \right) 
$$
where
$$
H(n) := -M(n) \, \dfrac{\kappa}{c} \, \left(\dfrac{n}{P(n)}\right)^2 \left\{\left[1-2 n + \dfrac{(1-n)}{2}\right]  \dfrac{P(n)}{n} + \, \dfrac{\kappa}{c} \dfrac{n (1-n)}{2} \, (1-2 M(n)) - \dfrac{(1-n)}{2} \, P'(n) \right\}.
$$

Given $M(0) = \bar{m}$ and \eqref{eq:S2.9}, $0 \leq M(n) < 1$ for all $n \in [0,1]$, and, thus, $g''(n) < 0 \; \forall n\in[0,1]$ if and only if $H(n) < 1 \; \forall n\in[0,1]$. Using l'H\^opital's rule to compute $\displaystyle{\lim_{n \to 0} \dfrac{n}{P(n)}}$ and $\displaystyle{\lim_{n \to 0} \dfrac{P(n)}{n}}$ we find that
$$
H(0) = -\bar{m} \, \dfrac{\kappa}{c} \, \left(\dfrac{1}{P'(0)}\right)^2 \left\{\dfrac{3}{2} \, P'(0) - \dfrac{1}{2} \, P'(0)  \right\} = -\bar{m} \, \dfrac{\kappa}{c} \, \dfrac{1}{P'(0)}.
$$
Using the fact that
$$
{P_{\pm}'(0) = \dfrac{-c \pm \sqrt{c^2 - 4 (1-\bar{m})}}{2}},
$$
we obtain
$$
H(0) < 1 \quad \iff \quad \kappa \, \dfrac{\bar{m}}{c} \, \dfrac{2}{c - \sqrt{c^2 - 4 (1-\bar{m})}} < 1,
$$
since $-[P_+'(0)]^{-1} \geq -[P_-'(0)]^{-1}$. Solving this inequality, we find that, if $c \geq 2 \sqrt{1-\bar{m}}$ and $0 < \kappa < \frac{1-\bar{m}}{\bar{m}}$, then $H(0)<1$. We now make the hypothesis that, if $c \geq 2 \sqrt{1-\bar{m}}$ and $0 < \kappa < \frac{1-\bar{m}}{\bar{m}}$, then
$$
{H'(n) \leq 0 \quad \forall \,  n \in [0,1].}
$$
This would allow us to conclude that 
\begin{equation}
    H(n) < 1 \;\; \forall \, n \in [0,1],  \quad \text{i.e.} \quad g''(n) < 0 \;\; \forall \, n \in [0,1]
    \label{eq:2.12}
\end{equation}

To summarise, if $0 < \kappa \leq \frac{1-\bar{m}}{\bar{m}}$, then, for any $c \geq 2 \sqrt{1-\bar{m}}$, $g$ is of Fisher-KPP type, $g''(0) < 0$ and we conjecture that we also have $g''(n) < 0 \; \forall n \in (0,1]$. This implies that there exists a unique solution to \eqref{eq:2.1} that satisfies \eqref{eq:2.2} for any $c \geq 2 \sqrt{g'(0)} = 2 \sqrt{1-\bar{m}}$. Equivalently, there exists a unique solution to \eqref{eq:1.1a}-\eqref{eq:1.1c} that satisfies $\displaystyle \lim_{y\to-\infty} (n,p,m) = (1,0,0)$ and   $\displaystyle \lim_{y\to+\infty} (n,p,m) = (0,0,\bar{m})$ for any $c \geq 2 \sqrt{1-\bar{m}}$.

\subsection{Numerical simulations}
\paragraph{Numerical methods.}
We solve numerically the PDE model \eqref{eq:S1} on the 1-D spatial domain $\mathcal{X} \coloneqq [0,L]$ with $L > 0$, subject to the initial conditions \eqref{eq:S3.3}, using the method of lines. Setting $L=200$, we discretise the spatial domain $\mathcal{X}$ using a uniform grid comprising 2000 points and impose no flux boundary conditions. Similarly to \cite{46}, we use the following explicit central difference scheme to approximate the nonlinear diffusion terms:
\begin{equation}
\begin{aligned}
     \frac{\partial}{\partial x}\left((D(M) \frac{\partial N}{\partial x}\right) & \approx \frac{1}{2 (\delta x)^2} \biggl(\left(D(M)\Bigr \rvert_{r-1} +D(M)\Bigr \rvert_{r} \right)N_{r-1} \\
     & - \left(D(M)\Bigr \rvert_{r-1} +2D(M)\Bigr \rvert_{r} + D(M)\Bigr \rvert_{r+1}\right )N_r\\ & - \left(D(M)\Bigr \rvert_{r} +D(M)\Bigr \rvert_{r+1} \right) N_{r+1} \biggr),
\end{aligned}
\end{equation}
where $\mid_r$ denotes evaluation at the $r^{th}$ spatial grid point and $\delta_x = 0.1$ is the spatial grid size. This spatial discretisation results in a system of $2000$ time-dependent ODEs, which we solve using ODE15s, a variable step, variable order MATLAB built-in solver for stiff ODEs that is based on the numerical differentiation formulas (NDF1-NDF5). In line with the initial conditions \eqref{eq:S3.3}, for each $r\in \llbracket 1, 2000 \rrbracket$, we impose the following initial conditions:
\begin{equation}
    \begin{cases}
      N_r(0) = 1, \; M_r(0) = 0, \qquad \qquad \qquad \qquad \qquad \quad \;\;\;\, \qquad \qquad \qquad \text{if} \; 0 \leq x_r < \sigma-\omega, \\
      N_r(0) =  \exp{\left( 1 - \frac{1}{1 - \left(\frac{x_r-\sigma + \omega}{\omega}\right)^2}\right)},  \;  M_r(0) = \bar{M}\left( 1- N_r(0) \right), \qquad \; \text{if} \; \sigma-\omega \leq x_r < \sigma, \\
      N_r(0) = 0, \;  M_r(0) = \bar{M}, \qquad \qquad \qquad \qquad \qquad \qquad \qquad \qquad \quad \;\; \text{if} \; \sigma \leq x_r \leq 200,
    \end{cases}
    \label{eq:S3.1}
\end{equation}
where $\bar{M} \in [0,1]$. Unless otherwise stated, we run the simulations for $t \in [0,100]$. 

To numerically solve the ODE models \eqref{eq:1a}-\eqref{eq:1b} and \eqref{eq:1.1a}-\eqref{eq:1.1c}, subject to their respective initial conditions, we use the MATLAB built-in solvers ODE15s and ODE45, respectively.

\pagebreak 

\paragraph{Travelling wave profiles for TWS of the ODE model in the desingularised variables.} \hfill

\begin{figure}[!ht]
\begin{subfigure}[t]{0.5\textwidth}
\centering
\includegraphics[scale=0.18]{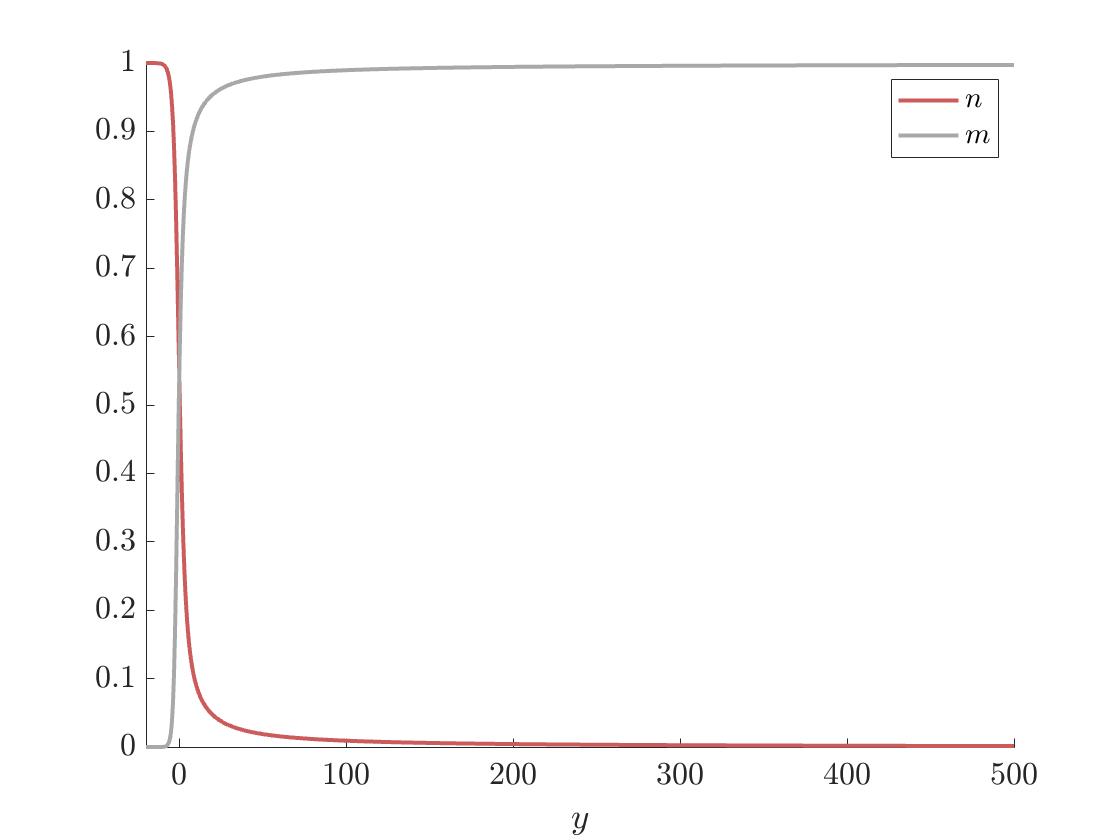}
\caption{}
\label{fig:3.1a}
\end{subfigure}
\hfill
\begin{subfigure}[t]{0.5\textwidth}
\centering
\includegraphics[scale=0.175]{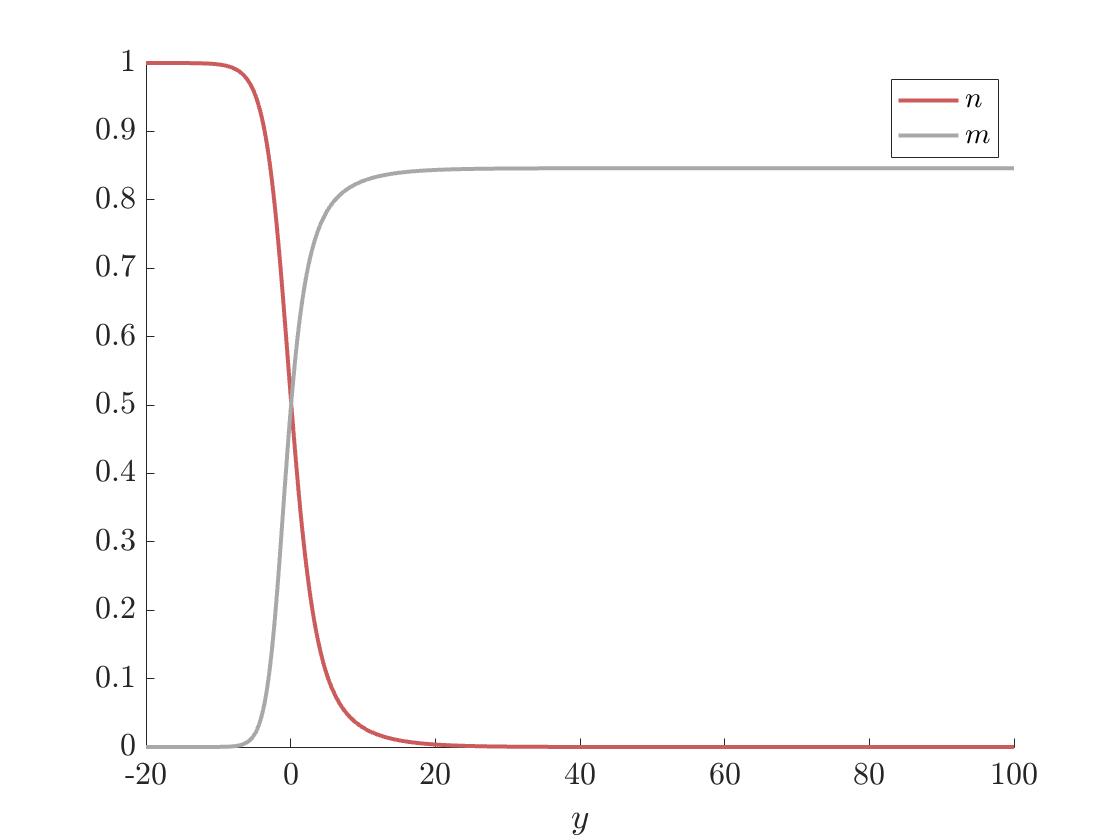} 
\caption{}
\label{fig:3.1b}
\end{subfigure}
\begin{subfigure}[t]{0.5\textwidth}
\centering
\includegraphics[scale=0.18]{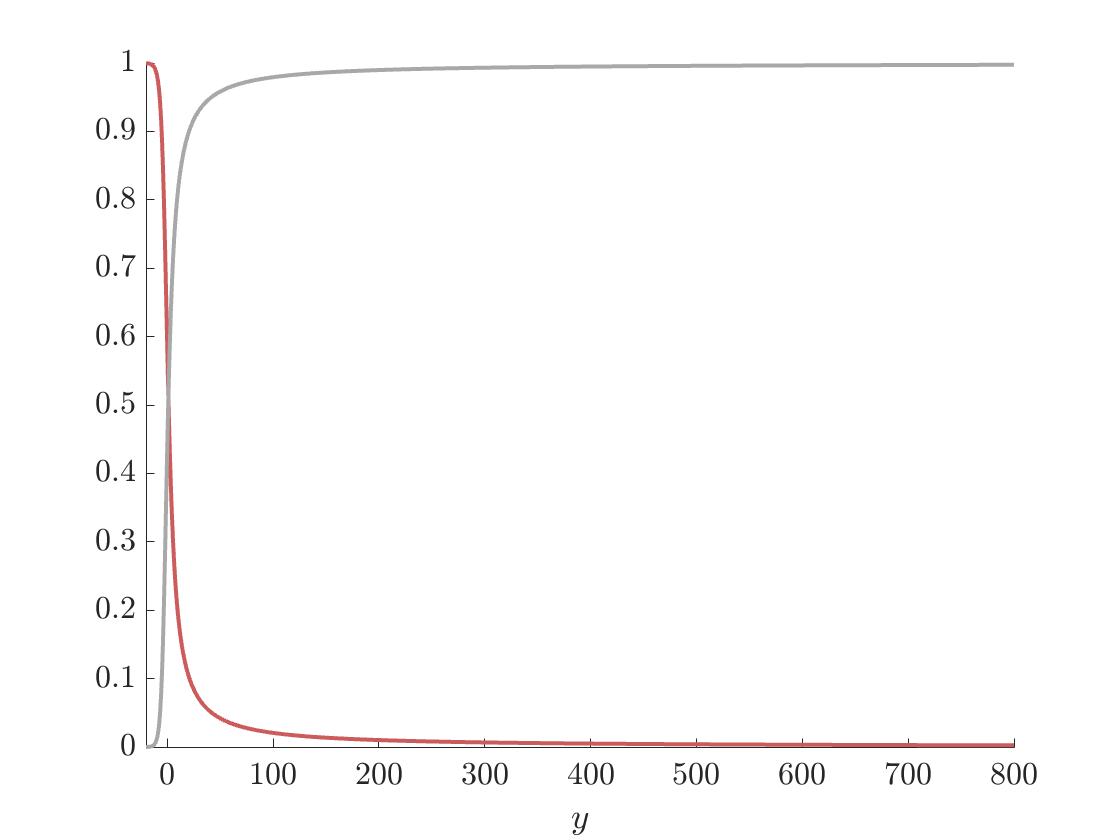} 
\label{fig:3.1c}
\caption{}
\end{subfigure}
\hfill
\begin{subfigure}[t]{0.5\textwidth}
\centering
\includegraphics[scale=0.18]{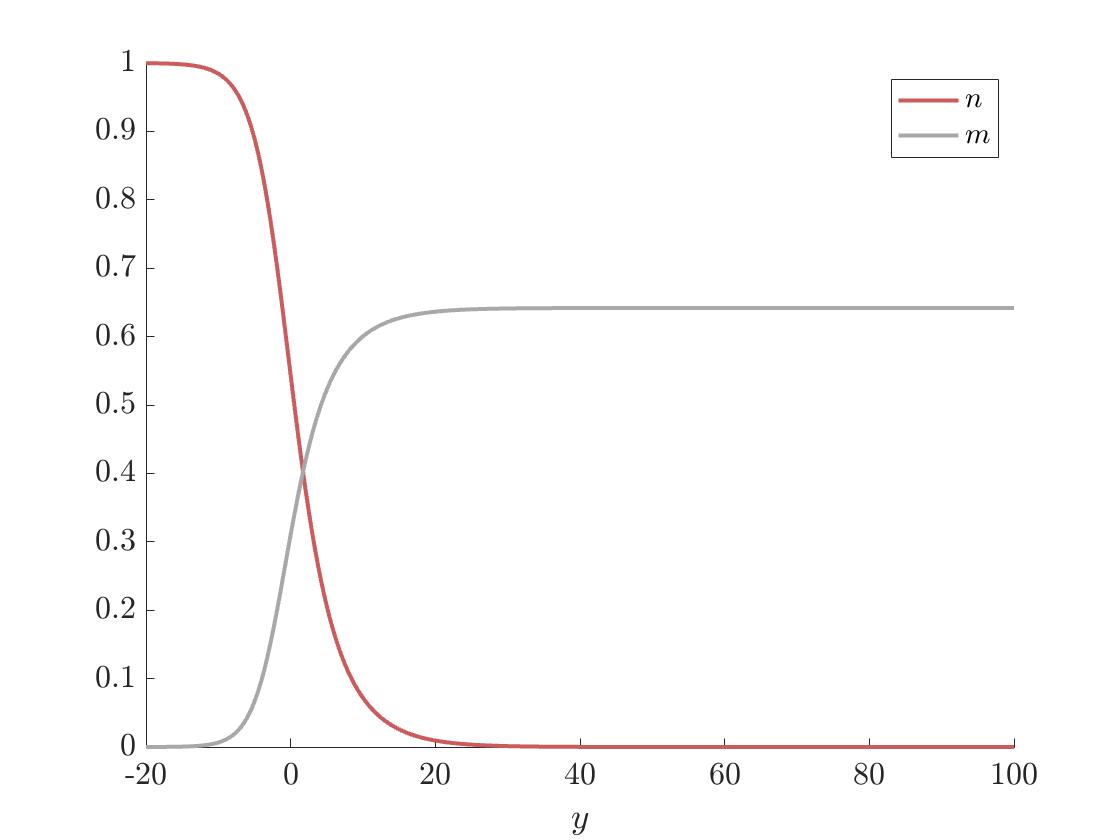} 
\label{fig:3.1d}
\caption{}
\end{subfigure}
\caption{In (a) and (b), we plot the $n$ and $m$ components of the solution of the desingularised system \eqref{eq:1.1a}-\eqref{eq:1.1c} subject to the asymptotic conditions \eqref{eq:S1.2} with $\kappa =1$, $c = 1$ and $\alpha = 3.72$ (a) or $\alpha = 3$ (b). In (c) and (d), we plot the $n$ and $m$ components of the solution of the desingularised system \eqref{eq:1.1a}-\eqref{eq:1.1c} subject to the asymptotic conditions \eqref{eq:S1.2} with $\kappa =1$, $c = 2$ and $\alpha = 1.161$ (c) or $\alpha = 1$ (d). The travelling wave profiles in plots (a) and (c) correspond to TWS that satisfy the asymptotic condition $\displaystyle \lim_{y\to+\infty} (n(y),p(y),m(y)) = (0,0,1)$ and those in plots (b) and (d) correspond to TWS that satisfy the asymptotic condition $\displaystyle \lim_{y\to+\infty} (n(y),p(y),m(y)) = (0,0,\bar{m})$, with $\bar{m} \in [0,1)$. We observe that, in the former case, $n(y)$ and $m(y)$ converge slowly to $0$ and $1$, respectively, as $y \to +\infty$, whereas, in the latter case, $n(y)$ and $m(y)$ converge fast to $0$ and $\bar{m}$, respectively, as $y \to +\infty$. }
\label{fig:3.1}
\end{figure}

\pagebreak

\paragraph{Travelling waves of the PDE model.} \hfill

We recall that we solve \eqref{eq:S1} on the 1-D spatial domain $\mathcal{X} \coloneqq [0,L]$, where $L > 0$. Similarly to \cite{46}, we assume that the tumour has already spread to a position $x=\sigma < L$ in the tissue and we impose initial conditions that satisfy, for $\bar{M} \in [0,1]$:
\begin{equation}
    \begin{cases}
      N(x,0) = 1, \; M(x,0) = 0, \qquad \qquad \qquad \qquad \qquad \quad \;\;\;\, \qquad \qquad \qquad \text{if} \; 0 \leq x < \sigma-\omega, \\
      N(x,0) =  \exp{\left( 1 - \frac{1}{1 - \left(\frac{x-\sigma + \omega}{\omega}\right)^2}\right)},  \;  M(x,0) = \bar{M}\left( 1- N(x,0) \right), \quad \text{if} \; \sigma-\omega \leq x < \sigma, \\
      N(x,0) = 0, \;  M(x,0) = \bar{M}, \qquad \qquad \qquad \qquad \qquad \qquad \qquad \qquad \quad \; \text{if} \; \sigma \leq x \leq L.
    \end{cases}
    \label{eq:S3.3}
\end{equation}
Here, $0 < \omega < \sigma$ represents how sharp the initial boundary between the tumour and healthy tissue is. 

\begin{figure}[!h]
\begin{subfigure}[t]{0.5\textwidth}
\centering
\includegraphics[scale=0.18]{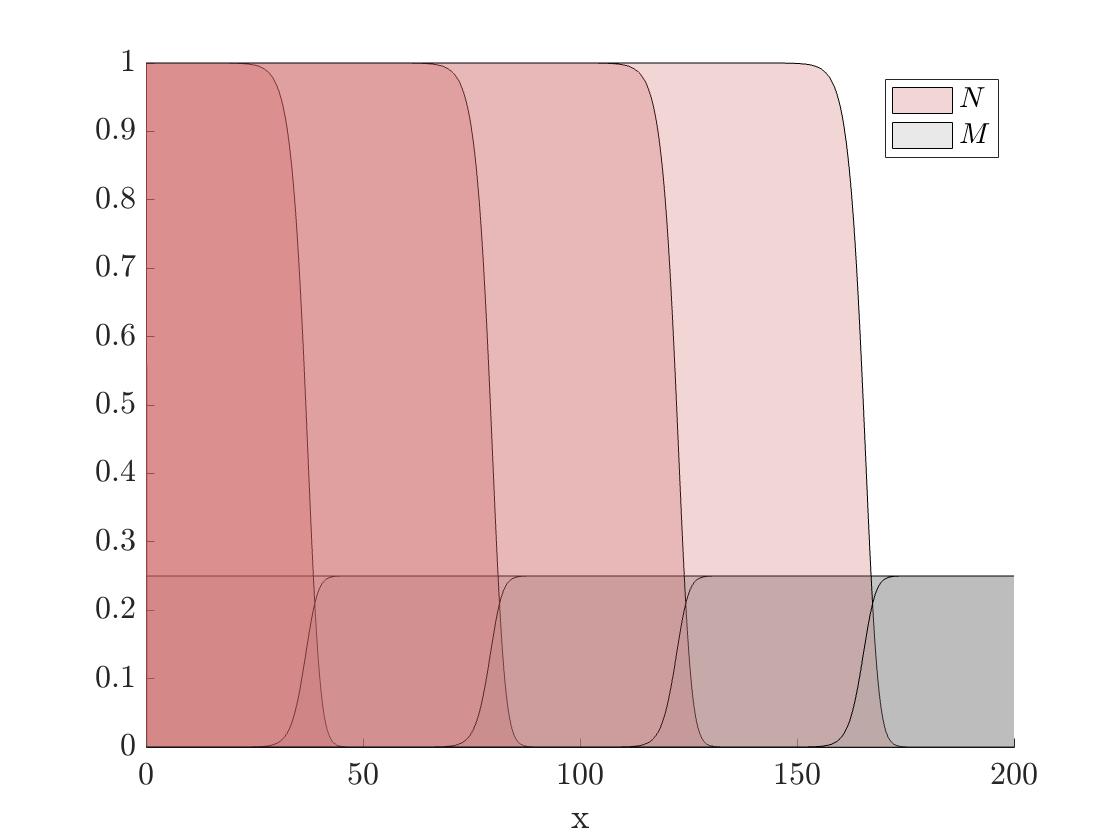} 
\caption{}
\label{fig:3.2a}
\end{subfigure}
\hfill
\begin{subfigure}[t]{0.5\textwidth}
\centering
\includegraphics[scale=0.18]{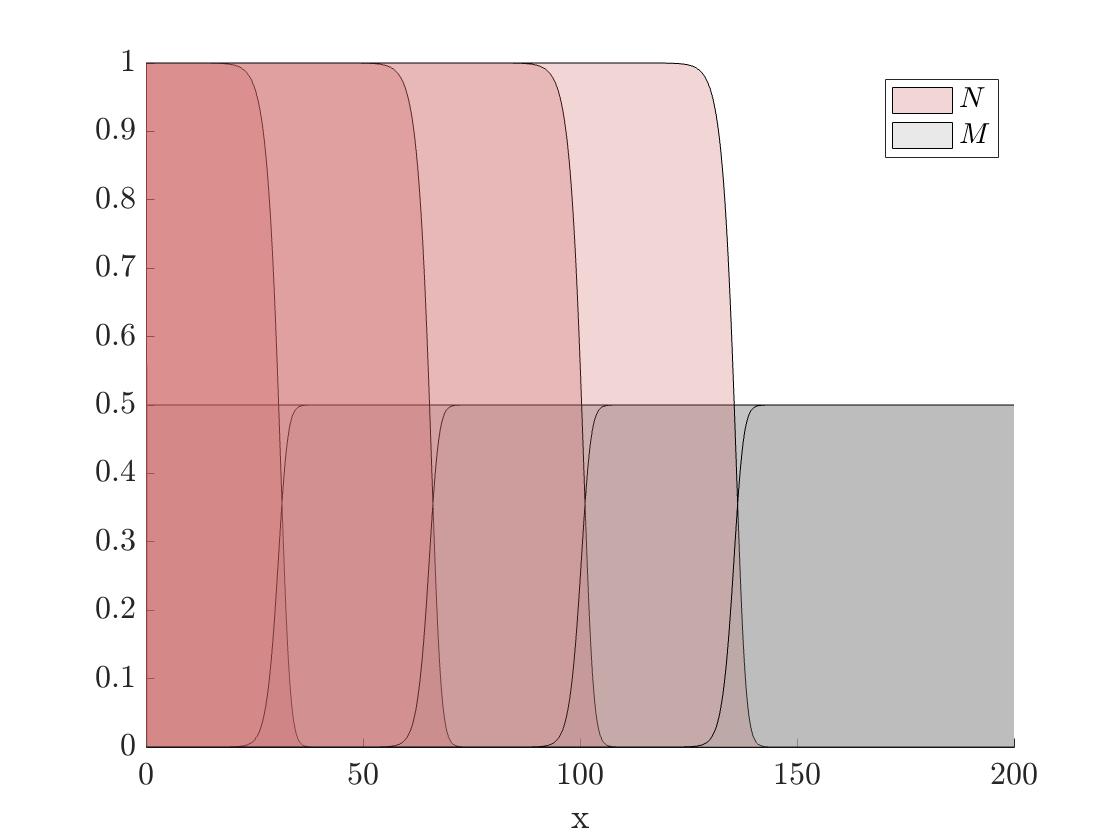} 
\caption{}
\label{fig:3.2b}
\end{subfigure}

\begin{subfigure}[t]{0.5\textwidth}
\centering
\includegraphics[scale=0.18]{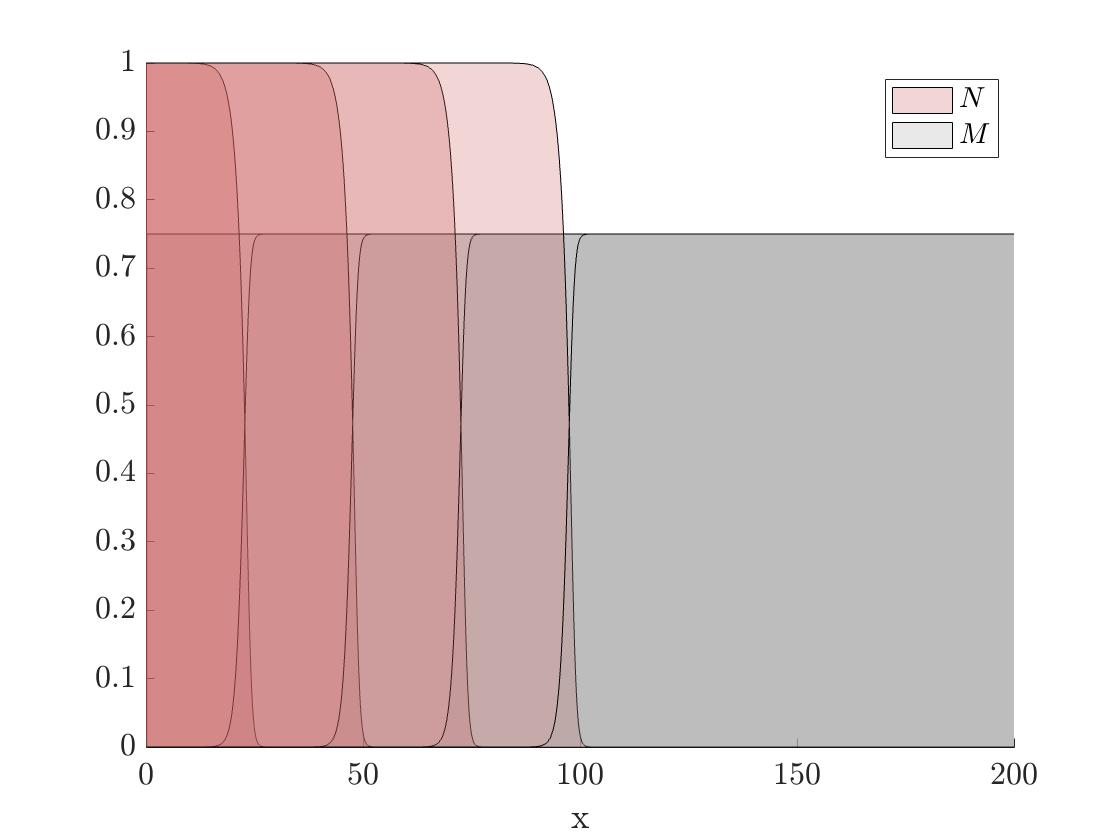} 
\caption{}
\label{fig:3.2c}
\end{subfigure}
\hfill
\begin{subfigure}[t]{0.5\textwidth}
\centering
\includegraphics[scale=0.18]{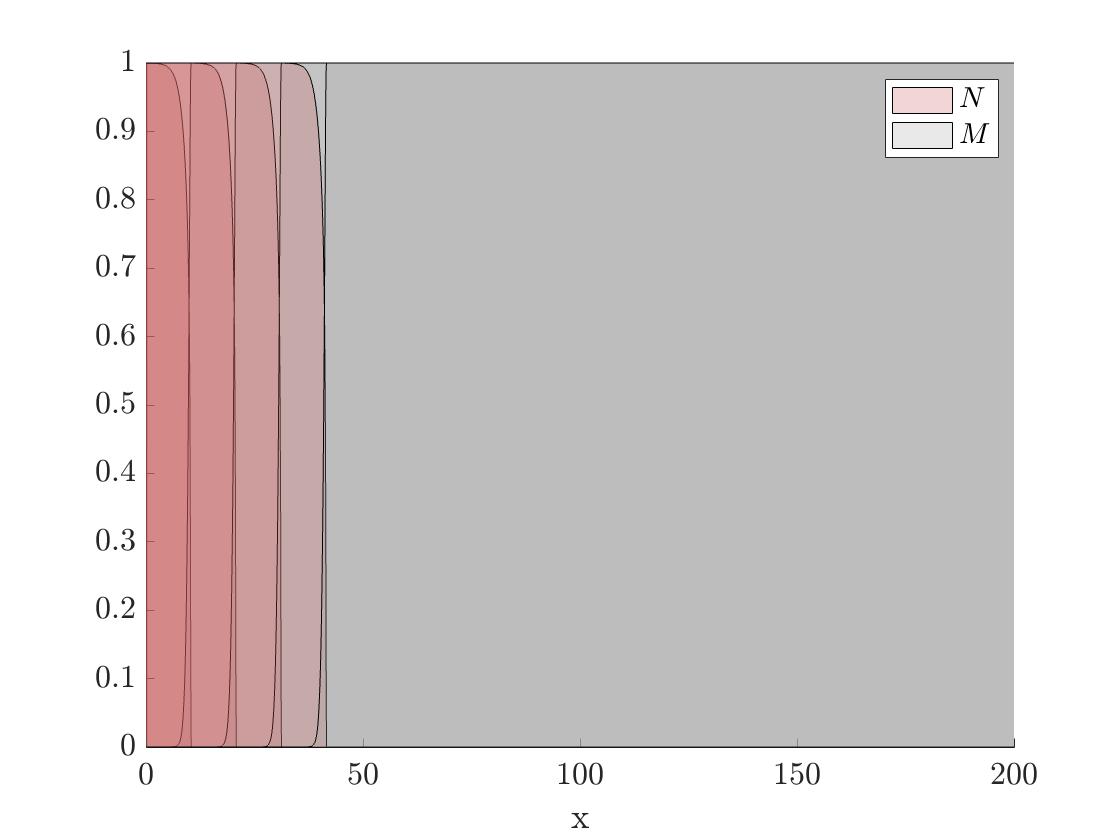} \caption{}
\label{fig:3.2d}
\end{subfigure}
\caption{We solve system \eqref{eq:S1} on the 1-D spatial domain, $x \in \mathcal{X} = [0,200]$, and impose the initial conditions \eqref{eq:S3.3} with $\sigma = 0.2$, $\omega = 0.1$ and $\bar{M}=0.25$ (a), $\bar{M}=0.5$ (b), $\bar{M}=0.75$ (c) and $\bar{M}=1$ (d). We plot the respective solutions for $t\in \{25,50,75,100\}$ and observe the emergence of a constant profile, constant speed TWS in all cases.}
\end{figure}

\pagebreak

\end{document}